%% file: Ver2.tex
\documentclass{article}
\usepackage{graphicx} 
\usepackage{amsthm}
\title{A Rank for Discontinuities of Multivalued Functions}
\author{Daniel S. Mourad}
\date{\today}
\usepackage{color}
\usepackage{fancyhdr} 
\usepackage{lastpage} 
\usepackage{lipsum} 
\usepackage{amsfonts}
\usepackage{stmaryrd}
\usepackage{amsmath}
\usepackage{multirow}
\usepackage{mathtools}
\usepackage[left=3cm, right=3cm]{geometry}
\usepackage{enumerate}
\usepackage{enumitem}
\usepackage{comment}
\usepackage{amssymb}
\usepackage{bbm}
\usepackage[style=numeric,maxnames=99,sortcites]{biblatex}
\usepackage{hyperref}
\usepackage[all]{hypcap}
\usepackage{tikz}
\usepackage{tikz-cd}
\usetikzlibrary{calc}
\makeatletter
\@ifundefined{T@LS2}{\DeclareFontEncoding{LS2}{}{\noaccents@}}{}
\makeatother
\DeclareFontSubstitution{LS2}{stix2}{m}{n}

\DeclareSymbolFont{stixsmashtimesfont}{LS2}{stix2tt}{m}{n}
\SetSymbolFont{stixsmashtimesfont}{bold}{LS2}{stix2tt}{b}{n}

\DeclareMathSymbol{\STIXsmashtimesSymbol}{\mathbin}{stixsmashtimesfont}{"A4}

\newcommand{\smashtimes}{\STIXsmashtimesSymbol}

\usetikzlibrary{positioning,trees}
\newtheorem{Theorem}{Theorem}[subsection]

\newtheorem{Lemma}[Theorem]{Lemma}

\newtheorem*{Claim}{Claim}

\newtheorem{Proposition}[Theorem]{Proposition}

\newtheorem{Corollary}[Theorem]{Corollary}

\newtheorem{Question}[Theorem]{Question}
\theoremstyle{definition}
\newtheorem*{Remark}{Remark}
\newtheorem{RemarkNum}[Theorem]{Remark}

\newtheorem{Definition}[Theorem]{Definition}
\newcounter{MainTheorem}
\renewcommand{\theMainTheorem}{\arabic{MainTheorem}}
\newenvironment{MainTheoremDisplay}{%
    \refstepcounter{MainTheorem}%
    \par\medskip\noindent\textbf{Theorem~\theMainTheorem. }\itshape
}{%
    \par\medskip
}

\newcommand{\ACCN}{\mathsf{ACC}_\mathbb{N}}

\newcommand{\Q}{\mathbb{Q}}
\newcommand{\W}{\mathrm{W}}
\newcommand{\sW}{\mathrm{sW}}

\newcommand{\N}{\mathbb{N}}

\newcommand{\mc}[1]{\mathcal{#1}}
\newcommand{\mb}[1]{\mathbf{#1}}
\newcommand{\tto}{\rightrightarrows}
\newcommand{\llb}{\llbracket}
\newcommand{\fop}[1]{{^1}{#1}}
\newcommand{\rrb}{\rrbracket}
\newcommand{\B}{\N^\N}
\newcommand{\TS}{\mathsf{TS}}

\newcommand{\A}{\forall}
\newcommand{\E}{\exists}

\DeclareMathOperator{\dom}{dom}
\DeclareMathOperator{\rank}{rank}
\DeclareMathOperator{\rankp}{rankp}
\DeclareMathOperator{\range}{range}
\newcommand{\converges}{{\downarrow}}

\newcommand{\cat}{{^\frown}}
\newcommand{\DIS}{\mathsf{DIS}}

\newcommand{\repr}[1]{\delta_{\mathbf{#1}}}
\newcommand{\deltarep}[1]{\repr{#1}}
\newcommand{\tup}[1]{\langle #1 \rangle}
\newcommand{\realizerver}[1]{#1^{\mathrm{r}}}
\newcommand{\extendedBaireSpace}{(\N \cup \{\hash\})^\N}

\newcommand{\hash}{\#}
\newcommand{\RT}{\mathsf{RT}}
\newcommand{\id}{\mathsf{id}}
\newcommand{\closure}{\mathrm{cl}}

\newcommand{\MC}{\mathrm{MC}}
\renewcommand{\vec}[1]{\bar{#1}}

\begin{document}
\maketitle

\begin{abstract}
We study discontinuity of multivalued functions (also known as problems) $P$ on Baire space by assigning an ordinal rank to points in the domain of $P$ that have no local realizers. For each countable ordinal $\alpha$,  let $\mathsf{ACC}_{\mathbb N}^{\alpha}$ be the problem of solving $\mathsf{ACC}_{\mathbb N}$ in at most $\alpha$ many attempts, with a new instance provided for each attempt. Our main theorem shows that, for any problem $P$, the following are equivalent: (i) $P$ is discontinuous on some set all of whose points have Cantor--Bendixson rank at most $\alpha$, and (ii) $\mathsf{ACC}_{\mathbb N}^{\alpha}\leq_{\mathrm W}^{*}P$. This extends to points: $P \geq_{\mathrm{W}}^* \mathsf{ACC}_{\mathbb{N}}^\alpha$ via a forward function that sends $\#^{\mathbb{N}}$ to $p \in \operatorname{dom}(P)$ if and only if $P$ is discontinuous on a set $A$ with $\operatorname{rank}_A(p) \leq \alpha$ and $P$ has no continuous realizer for any neighborhood of $p$. 
We also characterize these properties via a Wadge-style discontinuity game for $P$.

We apply this framework to the thin set and achromatic Ramsey theorems. Extending $\mathsf{RT}^{n}_{k,j}$ to ordinal parameters, we define $\mathsf{RT}^{n}_{\alpha,\beta}$ and compute their ranks of discontinuity. The problems $\mathsf{RT}^n_{\alpha,\beta}$ provide examples of problems with discontinuities of each countable rank which are not reducible to $\mathsf{ACC}_{\mathbb{N}}$. The separation from $\mathsf{ACC}_{\mathbb{N}}$ is obtained via the notion of guessability with identified errors.
\end{abstract}
\tableofcontents

\section{Introduction and Contents}
\subsection{Summary of Main Results}
It is well known that transfinite ordinals can be used to represent the number of attempts in which an agent promises to complete a task (see, e.g., \cite{ToyUses,FeirvaldSmithProcastination,luoMindChangeEfficient2006,debrechtLevelsDiscontinuityLimitcomputability2014}) in the following sense. If an agent promises to complete a task in at most $\alpha$ many attempts, then if they fail their next attempt they will provide a new ordinal $\beta < \alpha$ and subsequently promise to complete the task in at most $\beta$ many attempts. Pauly and Sold\`a \cite{SoldaPaulySequentialDiscontinuity} show that solving the problem $\ACCN$ in one attempt is the least difficult (in the sense of continuous Weihrauch reducibility $\leq_\W^*$) discontinuous task whose instances all have Cantor-Bendixson rank $\leq 1$. They conjectured that this result is the $n=1$ case of a more general phenomenon. In this paper, we show that this is indeed the case: there is a $\leq_\W^*$-least discontinuous task whose instances all have Cantor-Bendixson rank $\leq \alpha$; namely, the task of solving $\ACCN$ in at most $\alpha$ many attempts, in the sense that a new instance of $\ACCN$ is given for each attempt. We refer to this task as $\ACCN^\alpha$. 

\begin{MainTheoremDisplay}\label{Thm:MainTheorem}
    Let $P: \subset \B \tto \B$ be a problem, and let $\alpha$ be a countable ordinal. The following are equivalent. 
    \begin{enumerate}
        \item \label{MainTheoremItem:DiscontinuousOnPointWiseRankedSet} There is a set $A \subset \B$ such that $P|_A$ is discontinuous and each point $p \in A$ has Cantor-Bendixson rank $\rank_A(p) \leq \alpha$.     
        \item \label{MainTheoremItem:AboveACCN} $ \ACCN^\alpha \leq_\W^* P$.\footnote{We work with a version of Weihrauch reducibility which allows problems to have instances without solutions. See Section~\ref{Sec:WeihrauchReducibility}.}
    \end{enumerate}
\end{MainTheoremDisplay}

Problems for which Theorem~\ref{Thm:MainTheorem} is most relevant generally occur below a substantial piece of the Weihrauch degrees. Indeed, it immediately follows from Theorem~\ref{Thm:MainTheorem} and Pauly and Sold\`a's result that for any discontinuous first order problem $P$ we have $P >_\W^* \ACCN^\alpha$ for every countable $\alpha > 1$. Hence, the entirety of our $\ACCN^\alpha$ hierarchy is below every problem with non-trivial first order part.

We also obtain a local account of how points of discontinuity (i.e., points with no local realizers) manifest in multivalued problems in Theorem~\ref{Thm:ExtendingSolutions}. This observation leads us to a Cantor-Bendixson style ordinal ranking of points of discontinuity (Definition~\ref{Def:RankOfDiscontinuity}). These ranks play nicely with Weihrauch reducibility (Corollary~\ref{Cor:RanksAreWeihrauchInvariant}) and are also characterized by the rank-$\alpha$ discontinuity game, a two player game (Definition~\ref{def:DiscontinuityGame}) in which players take turns witnessing Theorem~\ref{Thm:ExtendingSolutions}. 

Brattka \cite{BrattkaDiscontinuity} uses a Wadge style game developed by Nobrega and Pauly \cite{nobregaGameCharacterizationsLower2019} to show, assuming dependent choice and enough determinacy, that the discontinuity problem $\DIS$ is the least discontinuous problem in the Weihrauch degrees. Pauly and Sold\`a  \cite{SoldaPaulySequentialDiscontinuity} generalize this result to show that the inequality problem for a represented space $\mathbf{X}$ is the least discontinuous problem whose codomain is $\mathbf{X}$. In fact, in the case of the standard representation of $\N$, no determinacy assumptions are needed. 

The rank-$\alpha$ discontinuity game for $P$ is similar to the Wadge game for $P$ in that players adversarially compete to demonstrate the continuity (or lack thereof) of $P$.
Similarly to the Wadge game, Player~1 having a winning strategy in the rank-$\alpha$ discontinuity game for $P$ implies that $P$ is discontinuous. Unlike the Wadge game, however, Player~2 having a winning strategy in the rank-$\alpha$ game for $P$ does \emph{not} imply that $P$ is continuous. Instead, it implies that $P$ is continuous on every subset of its domain whose points all have Cantor-Bendixson rank less than or equal to $\alpha$ (Corollary~\ref{Cor:Player2WinConsequence}). 

Theorem~\ref{Thm:MainTheoremPointWise} restates Theorem~\ref{Thm:MainTheorem} with reference to this game, points of discontinuity, and their ranking.

\begin{MainTheoremDisplay}\label{Thm:MainTheoremPointWise}
    Let $P: \subset \B \tto \B$ be a problem, let $p \in \dom(P)$, and let $\alpha$ be a countable ordinal. The following are equivalent. 
    \begin{enumerate}
        \item \label{MainTheoremPointwiseItem:DiscontinuousOnPointWiseRankedSet} There is a set $A \subset \B$ such that $P|_A$ is discontinuous, $p$ is a point of discontinuity of $P|_A$, and $p$ has Cantor-Bendixson rank $\rank_A(p) \leq \alpha$.     
        \item \label{MainTheoremPointwiseItem:AboveACCN} $P \geq_\W^* \ACCN^\alpha$ via forward function $\Phi$ such that $\Phi(\hash^\N) = p$.
        \item \label{MainTheoremPointwiseItem:WinningDiscontinuityGameInAlphaTurns} Player~1 has a winning strategy in the rank-$\alpha$ discontinuity game for $P$ with initial move $p$.
        \item \label{MainTheoremPointwiseItem:HaveARankLeqAlphaDiscontinuity} $p$ is a rank-$\beta$ discontinuity of $P$ for some $\beta \leq \alpha$.
    \end{enumerate}
\end{MainTheoremDisplay}

Since Condition~\ref{MainTheoremItem:DiscontinuousOnPointWiseRankedSet} (resp., Condition~\ref{MainTheoremItem:AboveACCN}) of Theorem~\ref{Thm:MainTheorem} is equivalent to there existing a $p$ satisfying Condition~\ref{MainTheoremPointwiseItem:DiscontinuousOnPointWiseRankedSet} (resp., Condition~\ref{MainTheoremPointwiseItem:AboveACCN}) of Theorem~\ref{Thm:MainTheoremPointWise},  we have that Theorem~\ref{Thm:MainTheoremPointWise} implies Theorem~\ref{Thm:MainTheorem}. 
Theorem~\ref{Thm:WinningStrategyImpliesAboveACCN} shows (\ref{MainTheoremPointwiseItem:WinningDiscontinuityGameInAlphaTurns}) $\Rightarrow$ (\ref{MainTheoremPointwiseItem:AboveACCN}). 
Theorem~\ref{Thm:FollowPhiOnACompactSubsetOfDiscontinuity} shows that (\ref{MainTheoremPointwiseItem:AboveACCN}) $\Rightarrow$ (\ref{MainTheoremPointwiseItem:DiscontinuousOnPointWiseRankedSet}).
Theorem~\ref{Thm:RankDiscontinuityEquivalence} shows (\ref{MainTheoremPointwiseItem:DiscontinuousOnPointWiseRankedSet}) $\Rightarrow$ (\ref{MainTheoremPointwiseItem:WinningDiscontinuityGameInAlphaTurns}). 
Finally, we prove that (\ref{MainTheoremPointwiseItem:HaveARankLeqAlphaDiscontinuity}) $\iff$ (\ref{MainTheoremPointwiseItem:WinningDiscontinuityGameInAlphaTurns}) in Theorem~\ref{Thm:DiscontinuityGameCharacterization}.

\subsection{Problems with Trivial First Order Part}

Dzhafarov, Solomon and Yokoyama \cite{dzhafarovFirstorderPartsProblems2024} introduce the notion of the first order part of a problem.  We provide the definition below.
\begin{Definition}
	\label{Def:FirstOrderPart}
	Let $P: \subset \B \tto \B$ be a problem. We define the first order part of $P$, $\fop{P}$, to be the problem with instances $\dom(\fop{P}) = \{(p, \Psi): p \in \dom(P), \text{ } \Psi: \subset \B \to \N \text{ is a continuous function} \text{ with }P(p) \subset \dom(\Psi)\}$ and solutions $\fop{P}(p, \Psi) = \{\Psi(q): q \in P(p)\}$.
\end{Definition}
Essentially, the first order part of a problem $P$ is the problem of, given an instance $p$ of $P$, to output a sufficiently long initial segment of a $P$-solution to $p$, with the length determined by $\Psi$. 

Dzhafarov, Solomon and Yokoyama show that any problem below $P$ whose codomain is $\N$ is Weihrauch reducible to $\fop{P}$. As previously mentioned, Pauly and Sold\`a \cite{SoldaPaulySequentialDiscontinuity} show that $\ACCN$ is Weihrauch reducible to every problem with non-trivial first order part. Hence, the entirety of our $\ACCN^\alpha$ hierarchy is below every problem with non-trivial first order part. 

Sold\`a and Valenti show that the set of problems with trivial first order part forms a sublattice of the Weihrauch degrees \cite[cf. Proposition~4.1]{soldaAlgebraicPropertiesFirstorder2023}. We can view Theorems~\ref{Thm:MainTheorem} and \ref{Thm:MainTheoremPointWise} as providing some additional structure to this sublattice by characterizing the discontinuities of problems with trivial first order part.

\subsubsection{Problems Which Have no Rank-\texorpdfstring{$\alpha$}{alpha} Discontinuities for Any Countable \texorpdfstring{$\alpha$}{alpha}}

In light of Theorem~\ref{Thm:MainTheoremPointWise}, the class of problems whose discontinuities all have rank $\infty$ (in the sense of Definition~\ref{Def:RankOfDiscontinuity}) is a natural object of consideration.  Equivalently, this is the class of problems that are continuous on any scattered subset of their domain. It turns out that this is a relatively weak condition, being implied by several separate notions which are already present in the literature. 

In the process of performing Weihrauch analysis on classical computability theory, Brattka, Hendtlass, and Kreuzer \cite{HendtlassKreuzerBrattkaUniformComputationalContent2017} study \emph{densely realized} problems, those problems whose solution sets are always dense in $\N^\N$. Since Player~1 can only win the rank-$\alpha$ discontinuity game for a problem $P$ if $P$ has an instance with no solution above a certain $\rho \in \N^{<\N}$, densely realized problems are continuous on any scattered subset of their domain by Theorem~\ref{Thm:MainTheoremPointWise}. 

Hirschfeldt and Jockusch \cite{hirschfeldtJockuschNotionsComputabilitytheoreticReduction2016} introduce the notion of \emph{undiagonalizable} problems, those problems for which there is a uniform algorithm for determining whether a particular $\sigma \in \N^{<\N}$ is an initial segment of a solution to a given instance. Dzhafarov, Solomon and Yokoyama show that undiagonalizable problems have trivial first order part. In Theorem~\ref{Thm:undiagonalizableImpliesPlayer2Win}, we strengthen this result to show that they are continuous on any scattered subset of their domain.

Hirschfeldt and Jockusch also introduce the contrasting notion of a problem having \emph{diagonalization opportunities}. A problem $P$ being undiagonalizable implies that $P$ does not have diagonalization opportunities, but the converse does not hold. While we have not obtained a direct connection between having diagonalization opportunities and having a rank-$\alpha$ discontinuity for some countable $\alpha$ (and indeed, having diagonalization opportunities is not Weihrauch invariant so such a connection would not be straightforward), it turns out that having diagonalization opportunities characterizes a slight strengthening of Player~1 having a winning strategy in the rank-$\alpha$ discontinuity game for some countable $\alpha$. By allowing Player~2 to play any $q \in \B$ instead of only playing solutions to the most recent instance of $P$ given by Player~1, we obtain a game in which Player~1 having a winning strategy is equivalent to $P$ having diagonalization opportunities. We discuss this in more detail in Section~\ref{Sec:DiagonalizationOpps}.

The class of problems with trivial first order part which are also continuous on all scattered subsets of their domain also contains the very bottom of the discontinuous Weihrauch degrees, assuming enough determinacy. Brattka \cite{BrattkaDiscontinuity} shows that, assuming the determinacy of the Wadge game, the discontinuity problem $\DIS$ is the least discontinuous problem in the Weihrauch degrees. Since $\DIS$ is densely realized, for all countable ordinals $\alpha$ we have that $\DIS$ does not have a rank-$\alpha$ discontinuity and that $\DIS <_\W \ACCN^\alpha$. Furthermore, in a construction relying on the axiom of choice, Brattka provides examples of problems strictly below $\DIS$. By Corollary~\ref{Cor:RanksAreWeihrauchInvariant}, any such problems also do not have any rank-$\alpha$ discontinuities for any countable $\alpha$ and hence are also part of this class.

\subsubsection{Thin Set and Achromatic Ramsey Theorems}
\label{Sec:ThinSetTheorems}
\paragraph{Thin Set Theorem}
The thin set theorems provide natural intermediate examples between problems with nontrivial first order part (i.e., having a rank-$1$ discontinuity) and problems with only rank-$\beta$ discontinuities for $\infty \geq \beta \geq 2$.

Fix $n \in \N$ and $k \in \N \cup \{\N\}$. For a set of natural numbers $A \subset \N$, we write $[A]^n$ for the set of $n$-element subsets of $A$. For a coloring $c: [\N]^n \to k$, we say that a set $A \subset \N$ is \emph{thin} for $c$ if $c([A]^n) \neq k$. The thin set theorem says that every coloring $c: [\N]^n \to \N$ has an infinite thin set. We write $\TS^n_k$ for the problem with instances
\[
    \dom(\TS^n_k) = \{c| c: [\N]^n \to k\}
\]
and solutions $\TS^n_k(c) = \{A \subset \N: A \text{ is infinite and thin for } c\}$.

The reverse mathematical analysis of the thin set theorem was initiated by Friedman \cite{FriedmanBooleanRelationTheory}. Dorais, Dzhafarov, Hirst, Mileti and Shafer \cite{doraisUniformRelationshipsCombinatorial2015} show that, for all $n$,
\[
\TS^n_2 \geq_\sW \TS^n_3 \geq_\sW  \TS^n_4 \geq_\sW\dots \geq_\sW \TS^n_\N
\] 
and furthermore that these reductions are strict. They also show for $n = 1$ the stronger statement that $\TS^1_{k+1} \not\geq_\W \TS^1_k$. 
Hirschfeldt and Jockusch \cite{hirschfeldtJockuschNotionsComputabilitytheoreticReduction2016} extend this result to all $n$.  In fact, the proofs in \cite{doraisUniformRelationshipsCombinatorial2015} and \cite{hirschfeldtJockuschNotionsComputabilitytheoreticReduction2016} relativize to arbitrary oracles, so they actually show that $\TS^n_{k} \not\leq_\W^* \TS^n_{k+1}$. 

Patey \cite{pateyWeaknessBeingCohesive2016} proves these non-reductions in the non-uniform setting, showing that for all $n \geq 2$,
\[
\TS^n_2 \not\leq_\mathsf{c} \TS^n_3 \not\leq_\mathsf{c} \TS^n_4 \not\leq_\mathsf{c} \dots \not\leq_\mathsf{c} \TS^n_\N
\]
with $\not\leq_\mathsf{c}$ replaced by $\not\leq_\mathsf{sc}$ for $n = 1$ (since $\TS^1_k$ is computably true, this weakening is essential). Here, $\leq_\mathsf{c}$ represents computable reducibility. 

In Section~\ref{Sec:AchromaticRamseyTheorems}, we show that $\TS^1_{k+1}$ has a rank $k$ discontinuity but no rank $k - 1$ discontinuity, while $\TS^1_\N$ has only rank $\infty$ discontinuities. The main idea of the proof is essentially the same as the proof in \cite[Theorem~5.27]{doraisUniformRelationshipsCombinatorial2015} of the non-reduction of $\TS^1_k$ to $\TS^1_{k+1}$. On the other hand, for all $n \geq 2$, the problem $\TS^n_k$ has a rank $1$ discontinuity. This quantifies a way in which the additional complexity in the proofs of \cite{hirschfeldtJockuschNotionsComputabilitytheoreticReduction2016} and \cite{pateyWeaknessBeingCohesive2016} for the $n \geq 2$ case is essential.

\paragraph{Achromatic Ramsey Theorems for Ordinal-Valued Numbers of Colors}
\label{sec:AchromaticRamseyTheoremsIntro}

The achromatic Ramsey theorem was introduced by Erd\H{o}s, Hajnal and Rado \cite{erdosPartitionRelationsCardinal1965}. Its reverse mathematical analysis was proposed by Joe Miller (see \cite[Section 2.2.3]{montalbanOpenQuestionsReverse2011}) and studied by Wang \cite{wangLogicallyWeakRamseyan2014} in the context of reverse mathematics. 

We extend the definition of $\RT^n_{k,j}$ to allow for both the number of colors in $c$ and the number of colors in the solution to be ordinal-valued. For each ordinal $\beta < \omega_1$ and each $\alpha \in \{\N\} \cup \omega_1$ we define the problem $\RT^n_{\alpha,\beta}$ (see Definition~\ref{def:AchromaticRamseyTheorem}). 
If $\alpha = \N$, then instances of $\RT^n_{\alpha,\beta}$ are colorings $c:[\N]^n \to \N$ with finite range. If $\alpha \in \omega_1$, then instances are colorings $c:[\N]^n \to \N$ coupled with a decreasing sequence of ordinals strictly less than $\alpha$ whose decreases are in one-to-one correspondence with each new color appearing in $c([\N]^n)$. 
The set of solutions to such an instance is an infinite set $A \subset \N$ coupled with a decreasing sequence of ordinals strictly less than $\beta$ whose decreases are in one-to-one correspondence with each new color appearing in $c([A]^n)$. 
For both the instances and solutions, once the sequence reaches $0$, no new colors can appear. Then, $\TS^n_k$ is just $\RT^n_{k, k-1}$. Wang \cite{wangLogicallyWeakRamseyan2014} gives conditions on $n$ and $j$ for which $\RT^n_{\omega,j}$ admits strong cone avoidance. Cholak and Patey \cite{cholakThinSetTheorems2020} tighten these conditions to be exact and obtain specific information on which cones cannot be avoided when those conditions are not met.

In Theorems~\ref{thm:rankOfRT1}, \ref{thm:rankOfRTn}, and \ref{Thm:RTNotGuessableWithErrors}, we show that these problems provide examples of problems with discontinuities of each countable rank which are not reducible to $\ACCN$, and in particular are not equivalent to $\ACCN^\alpha$ for any countable $\alpha$. The non-reduction is accomplished by showing for $\alpha > 0$ that $\ACCN^\alpha$ is $2$-guessable with identified errors while $\RT^n_{\alpha,\beta}$ is not $2$-guessable with identified errors if $\beta < \alpha$.

\section{Background}
\label{Sec:Background}
\input{Background.tex}

\section{The \texorpdfstring{$\ACCN^\alpha$}{ACCN-alpha} Hierarchy}
\label{Sec:TheProblemsACCN}
\input{TheProblemsACCN.tex}
\section{Local Discontinuities of Problems}
\label{Sec:LocalDiscontinuities}
\input{LocalDiscontinuities.tex}




\section{Achromatic Ramsey Theorems}
\label{Sec:AchromaticRamseyTheorems}
\input{AchromaticRamseyTheorems.tex}

\section{Concluding Remarks}
\label{Sec:Conclusion}
\input{Conclusion.tex}

\section{Acknowledgements}

The author thanks Jun Le Goh, Arno Pauly, and Giovanni Sold\`a for helpful discussions.

Generative AI suggestions were used to assist in writing and editing this paper. The author accepted suggestions provided by OpenAI, Anthropic, and DeepSeek models through GitHub Copilot. In most cases, these suggestions were typographical in nature. In a notable exception, the proof of Theorem~\ref{Thm:undiagonalizableImpliesPlayer2Win} was initially produced by generative AI and completely correct. It was then modified by the author to both match style with the rest of this paper and adjust the proof to a modified statement. The author takes full responsibility for the content of this paper. Later, the paper was proofread by GPT6. The large language model discovered a mistake in a previous proof of Proposition~\ref{Prop:WinningStratNotRepeating} and provided the argument to repair it. 

\printbibliography
\end{document}

%% file: Background.tex
\subsection{Notation}
\label{sec:Notation}
We typically use the letters $p$ and $q$, sometimes using $a$ and $b$, for elements of $\B$, the set of infinite sequences of natural numbers (sometimes we also use $a$ and $b$ for natural numbers). We use Greek letters $\sigma, \tau$, and $\rho$ for elements of $\N^{<\N}$, the set of finite sequences of natural numbers. For $\sigma, \tau \in \N^{<\N}$, we write $\sigma \sqsubset \tau$ if $\sigma$ is an initial segment of $\tau$ (not necessarily a proper initial segment; it is possible that $\sigma = \tau$ as well). For $\sigma \in \N^{<\N}$ and $p \in \B$, we write $\sigma \sqsubset p$ if $\sigma$ is an initial segment of $p$. For $\sigma \in \N^{<\N}$, we write $\llb \sigma \rrb$ for the set of $p \in \B$ such that $\sigma \sqsubset p$. 

Let $X$ and $Y$ be spaces and let $A \subset X$ and $B \subset Y$. We use letters $P$ and $Q$ to denote partial multifunctions $P: \subset X \tto Y$. For a multivalued partial function $P: \subset X \tto Y$, we write $P|[A \times B]$ for the function with domain $\dom(P) \cap A$ and such that 
\[
	P|[A \times B](x) = P(x) \cap B.
\]
Note that it is possible for $x \in X$ to be in the domain of $P|[A \times B]$ but for $P|[A \times B](x) = \emptyset$.

We will borrow the symbol $\hash$ from the language learnability literature to represent ``waiting'' while producing an element of $\B$ (see e.g. \cite{luoMindChangeEfficient2006,brechtTopologicalPropertiesConcept2010}). Sometimes we will also use $\hash$ to represent a pause in a sequence to insert information from a different sequence. We identify $\B$ with $\extendedBaireSpace$ via the natural bijection induced by $\hash \mapsto 0$ and $n \mapsto (n + 1)$. 

We use Greek letters $\alpha, \beta$, and $\gamma$ to denote ordinals. We use $\omega$ for the least infinite ordinal and $\omega_1$ for the least uncountable ordinal. For an ordinal $\alpha$, we identify ordinals with the set of their predecessors, so that $\beta < \alpha$ if and only if $\beta \in \alpha$. We include natural numbers in this notation, so in the following, $[k]^n$ denotes the set of subsets of $\{0, \ldots, k-1\}$ with size $n$.

The following notation will be relevant in Section~\ref{Sec:AchromaticRamseyTheorems}.
For $n \in \N$ and $A \subset \N$, we write $[A]^n$ for the set of all subsets of $A$ with size $n$. Fix an enumeration $\langle \cdot \rangle$ of size $n$ subsets of $\N$. We code colorings $c: [\N]^n \to k$ as functions $\hat{c}: \N \to k$ by setting $c(a_1, \ldots, a_n) = \hat{c}(\langle a_1, \ldots, a_n \rangle)$. We identify $c$ with $\hat{c}$. We reserve the variable $\theta$ for an initial segment $\theta \sqsubset \hat{c}$ of a coloring. 

 We identify a subset $A \subset \N$ of natural numbers with the $\{0,1\}$-valued sequence that is the characteristic function $\chi_A \in 2^{\N}$ of $A$. We reserve the variable $\phi$ for finite $\{0,1\}$-valued sequences and identify $\phi \in 2^{<\N}$ with $\{i: \phi(i) = 1\}$. Thus, for an initial segment $\theta \in \N^{<\N}$ of a coloring and $\phi \in 2^{<\N}$, we write $\theta([\phi]^n)$ for
\[
\theta([\phi]^n) = \{k: (\E a_0, \dots, a_{n-1})((\A i < n)(\phi(a_i) = 1) \text{ and } \theta(a_0, \dots, a_{n-1}) = k)\}.
\]

\subsection{Weihrauch Reducibility}
\label{Sec:WeihrauchReducibility}
We give a brief introduction to Weihrauch reducibility. See \cite{brattkaGherardiPaulyWeihrauchComplexityComputable2021} and \cite{dzhafarovMummertReverseMathematicsProblems2022} for a more comprehensive account. 

\emph{Problems} are the object of study in Weihrauch reducibility. Many problems studied in Weihrauch reducibility originate from statements $P$ of the form 
\begin{equation}
	\label{Line:StatementFormOfProblem}
	P: \forall x (\Phi(x) \rightarrow \exists y\Psi(x,y)). 
\end{equation}
The statement form of $P$ naturally gives rise to a multifunction form $P: \{x: \Phi(x)\} \tto \{y: (\E x)\Psi(x,y)\}$, where for $x \in \dom(P)$, 
\[
	P(x) = \{y: \Psi(x,y)\}.
\]
We call $\dom(P)$ the set of \emph{instances} of $P$ and, for $x \in \dom(P)$, we call $P(x)$ the set of $P$-\emph{solutions} to $x$. We also call $P$ a \emph{problem} in the sense that solving an instance of $P$ means finding a $P$-solution to that instance. 

Usually, it is assumed that the statement form of $P$ is true, so that each instance of $P$ has a $P$-solution. However, in this paper we will adopt a different convention by also considering problems with instances that have no solutions. This is essential to our analysis, in which we take restrictions of problems on both their domain and codomain.

Weihrauch reducibility compares the uniform computational strength of problems by asking ``If we could solve an instance of $Q$, could we `make use' of the solution to solve an instance of $P$?'' The precise meaning of `make use' varies depending on the version of Weihrauch reducibility one is working with. We define a few variants below.

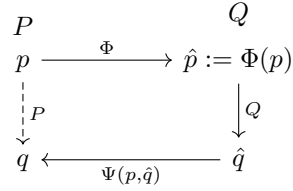
\begin{figure}[ht]
	\begin{center}
		\tikzcdset{diagrams={row sep={3.75em,between origins}}}
		\begin{tikzcd}
			|[label = P]| p \arrow[rr,"\Phi"] \arrow[d, dashed,"P"] & & |[label = Q]| \hat{p} := \Phi(p) \arrow[d, "Q"]\\
			q & & \hat{q} \arrow[ll, "\Psi(p{,}\hat{q})"] 
		\end{tikzcd}
	\end{center}
	\caption{$P \leq_\W Q$ via $\Phi$ and $\Psi$}
	\label{fig:WeihrauchReducibility}
\end{figure}

\begin{Definition}
	\label{Def:WeihrauchReducibility}
	Let $P: \subset \B \tto \B$ and $Q: \subset \B \tto \B$ be problems.
	\begin{itemize}
		\item We say that $P$ is \emph{Weihrauch reducible} to $Q$, denoted $P \leq_\W Q$, if there are computable functions $\Phi, \Psi: \subset \B \to \B$ such that for all $p \in \dom(P)$ we have that $\Phi(p)\converges \in \dom(Q)$ and for all $\hat{q} \in Q(\Phi(p))$ we have $\Psi(p, \hat{q}) \in P(p)$.
		\item  We say that $P$ is \emph{continuously Weihrauch reducible} to $Q$, denoted $P \leq_\W^* Q$, if there are continuous functions $\Phi, \Psi: \subset \B \to \B$ such that for all $p \in \dom(P)$ we have that $\Phi(p)\converges \in \dom(Q)$ and for all $\hat{q} \in Q(\Phi(p))$, we have $\Psi(p, \hat{q}) \in P(p)$.
		\item We say that $P$ is \emph{strongly Weihrauch reducible} to $Q$, denoted $P \leq_\sW Q$, if there are computable functions $\Phi, \Psi: \subset \B \to \B$ such that for all $p \in \dom(P)$ we have that $\Phi(p)\converges \in \dom(Q)$ and for all $\hat{q} \in Q(\Phi(p))$ we have that $\Psi(\hat{q}) \in P(p)$.
		\item We say that $P$ is \emph{continuously strongly Weihrauch reducible} to $Q$, denoted $P \leq_\sW^* Q$, if there are continuous functions $\Phi, \Psi: \subset \B \to \B$ such that for all $p \in \dom(P)$ we have that $\Phi(p)\converges \in \dom(Q)$ and for all $\hat{q} \in Q(\Phi(p))$ we have $\Psi(\hat{q}) \in P(p)$.
	\end{itemize} 
\end{Definition}
The implications in Figure~\ref{fig:TikzcdDiamond} are well known and easy to check. 
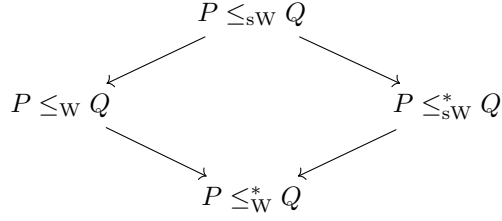
\begin{figure}[ht]
	\begin{center}
		\begin{tikzcd}
			& P \leq_\sW Q \arrow[dr] \arrow[dl] & \\
			P \leq_\W Q \arrow[dr]  & & P \leq_\sW^* Q \arrow[dl]\\
			& P \leq_\W^* Q  &
		\end{tikzcd}
	\end{center}
	\caption{The implications between the different versions of Weihrauch reducibility.}
	\label{fig:TikzcdDiamond}
\end{figure}

Each of $\leq_\W$, $\leq_\W^*$, $\leq_\sW$, and $\leq_\sW^*$ is transitive and therefore induces a degree structure on the set of all problems. Lempp, Marcone, and Valenti \cite{lempp.marcone.valentiCHAINSANTICHAINSWEIHRAUCH2026} study the possible lengths of chains and antichains in the $\leq_\W$-degrees. The problems $\ACCN^\alpha$ form a reverse embedding of $\omega_1$ into the $\leq_\W^*$-degrees, providing an initial result for the continuous case. 

Although algebraic features of the degree structure do not impact the results of the present paper, we make the following note. Since we are allowing problems to have instances with no solutions, following Definition~\ref{Def:WeihrauchReducibility} adds a degree at the top of each of these hierarchies. In the case of $\leq_\W$ and $\leq_{\sW}$, even more degrees are added because a problem may have instances $p$ without solutions for $p$ of arbitrary Turing complexity. 

Brattka and Pauly \cite{brattkaAlgebraicStructureWeihrauch2018} consider a form of Weihrauch reducibility which has a top degree if and only if the Axiom of Choice fails for Baire space. This occurs when Weihrauch reducibility is defined in terms of realizers and when problems exist with instances and solutions \emph{but no realizers}. However, the top degree $\infty$ they consider satisfies $0 \times \infty = \infty$, where $0$ is the degree of the problem with empty domain. On the other hand, the top degree $\infty$ obtained by allowing problems to have instances with no solutions satisfies $0 \times \infty = 0$. We mention this as a cautionary note although it does not affect our analysis. 

For a problem $P: \subset \B \tto \B$, we say that $P$ is \emph{continuous} if there is a continuous function $\Phi: \subset \B \to \B$ such that for all $p \in \dom(P)$ we have $\Phi(p) \in P(p)$. In this case, we call $\Phi$ a \emph{continuous realizer} for $P$. The cone below the problem $\id: \B \to \B$ defined by $\id(p) = \{p\}$ characterizes continuity: a problem $P$ is continuous if and only if $P \leq_\W^* \id$. 

So far, we have only discussed problems with domain and codomain $\B$. However, statements of the form in Line~\ref{Line:StatementFormOfProblem} can have their quantifiers range over many different sets, some with different topologies than $\B$ or even different cardinality than $\B$. To accommodate a wider range of contexts, we use names given by elements of $\B$ to represent points in other sets. 

\begin{Definition}
	A represented space $\mathbf{X} = (X, \delta_\mathbf{X})$ is a pair consisting of a set $X$ and a surjective map 
	\[
		\delta_\mathbf{X}: \subset \N^\N \to X. 
	\]
\end{Definition}

Suppose that we are understood to be working with represented spaces $\mathbf{X} = (X, \delta_\mathbf{X})$ and $\mathbf{Y} = (Y, \delta_\mathbf{Y})$ of sets $X$ and $Y$. If $P$ is a partial multifunction $P: \subset X \tto Y$, we define the realizer version $\realizerver{P}: \subset \B \tto \B$ of $P$ by 
\[ 
	\dom(\realizerver{P}) = \{p \in \B: \deltarep{X}(p) \in \dom(P)\}
\]
and
\[
	\realizerver{P}(p) = \{q \in \B: \deltarep{Y}(q) \in P(\deltarep{X}(p))\}.
\]

We now define Weihrauch reducibility for problems between represented spaces.
\begin{Definition}
	Let $\mathbf{X} = (X, \delta_\mathbf{X})$, $\widehat{\mathbf{X}} = (\widehat{X}, \deltarep{\widehat{\mathbf{X}}})$, $\mathbf{Y} = (Y, \delta_\mathbf{Y})$, and $\widehat{\mathbf{Y}} = (\widehat{Y}, \deltarep{\widehat{\mathbf{Y}}})$ be represented spaces. Let $P: \subset X \tto Y$ and $Q: \subset \widehat{X} \tto \widehat{Y}$ be problems. We say that $P$ \emph{is (continuously) [strongly] Weihrauch reducible to} $Q$ if $\realizerver{P}$ is (continuously) [strongly] Weihrauch reducible to $\realizerver{Q}$ in the sense of Definition~\ref{Def:WeihrauchReducibility}.
\end{Definition} 

Since each problem $P$ between represented spaces is strongly Weihrauch equivalent to its realizer version $\realizerver{P}$, statements about the Weihrauch degrees of problems  $P: \subset \B \tto \B$ on Baire space maintain full generality.

\subsection{Topology of Baire Space}

Recall that an open set $U$ isolates a point $x$ in $A$ if $U \cap A = \{x\}$. Since the results of the present paper are presented in terms of the cones above initial segments of points (particularly Theorem~\ref{Thm:ExtendingSolutions}), we find it convenient to work with an equivalent formulation of Cantor-Bendixson rank obtained by a generalization of isolating neighborhoods rather than derivatives. 
\begin{Definition}
	\label{CBRank}
	Let $p \in A \subset \N^\N$. 
	\begin{itemize}
		\item We say that an open set $U$ is a \emph{$\leq \!\!0$-isolating neighborhood of $p$ in $A$} if $U \cap A = \{p\}$.
	
		\item For $\alpha > 0$, we say that a neighborhood $U$ of $p$ \emph{is a $\leq \!\!\alpha$-isolating neighborhood of $p$ in $A$} if $U$ is open and, for all points $q \in A \cap U$ with $q \neq p$, there is $\beta < \alpha$ and a neighborhood $V \subset U$ of $q$ such that $V$ is a $\leq \!\!\beta$-isolating neighborhood of $q$ in $A$.
		\item If $p \in A$ has a $\leq \!\! \alpha$-isolating neighborhood in $A$, then we say that $\rank_A(p) \leq \alpha$. 
		\item By a temporary abuse of notation, we define $\rank_A(p) := \alpha$ for the least ordinal $\alpha$ such that $\rank_A(p) \leq \alpha$. If no such $\alpha$ exists then we write $\rank_A(p) = \infty$.
		\item For a set $A$, if $\alpha$ is the least ordinal such that $\rank_A(p) \leq \alpha$ for all $p \in A$ then we say that the pointwise Cantor-Bendixson rank of $A$, $\rankp(A)$, is equal to $\alpha$. If no such $\alpha$ exists then we write $\rankp(A) = \infty$.
	\end{itemize}
\end{Definition}

Of course, this definition is equivalent to the standard definition of Cantor-Bendixson rank using derivatives. 

We will need a few basic facts about rank, which we now prove.

\begin{Lemma}
	\label{Lem:SurroundedByLowerRanks}
	Fix a subset $A \subset \B$, and a point $p \in A$. Suppose that $\alpha$ is a countable ordinal and that $U$ is an open neighborhood of $p$ in $A$. Then, $U$ is a $\leq\!\! \alpha$-isolating neighborhood of $p$ in $A$ if and only if, for all $q \in A \cap U$ with $q \neq p$, we have that $\rank_A(q) < \alpha$. 
\end{Lemma}

\begin{proof}
	Suppose first that $U$ is a $\leq\!\!\alpha$-isolating neighborhood of $p$ in $A$. Fix $q \in A \cap U$ with $q \neq p$. By definition of $\leq\!\!\alpha$-isolating, there is some $\beta < \alpha$ and some neighborhood $V$ of $q$ such that $V$ is a $\leq\!\!\beta$-isolating neighborhood of $q$ in $A$. Therefore $\rank_A(q) \leq \beta < \alpha$.
	
	Conversely, suppose that for all $q \in A \cap U$ with $q \neq p$, we have $\rank_A(q) < \alpha$. Fix $q \in A \cap U$ with $q \neq p$. Then there is some $\beta < \alpha$ such that $\rank_A(q) = \beta$. By definition of rank, $q$ has a $\leq\!\!\beta$-isolating neighborhood in $A$. Hence the defining condition for $U$ to be a $\leq\!\!\alpha$-isolating neighborhood of $p$ in $A$ is satisfied.
\end{proof}

Next we use Lemma~\ref{Lem:SurroundedByLowerRanks} to show that our definition of $\rank_A(p)$ corresponds with the standard definition of the Cantor-Bendixson rank of a point. 

\begin{Proposition}
	Let $A \subset \B$. For each $p \in A$, $\rank_A(p)$ is the least ordinal $\alpha$ such that $p \in A^{(\alpha)} \setminus A^{(\alpha + 1)}$, where $A^{(\alpha)}$ is the $\alpha$'th derivative of $A$. 
\end{Proposition}
\begin{proof}
	Fix $p \in A$. We argue by induction on $\alpha = \rank_A(p)$. If $\alpha = 0$ then $p$ is isolated in $A$. Then, $p$ is removed by the first Cantor-Bendixson derivative and therefore has Cantor-Bendixson rank 0 in $A$.

	Suppose that the claim is true for all $\beta < \alpha$. Since $\rank_A(p) = \alpha$, we know that $p$ has no $\leq\!\!\beta$-isolating neighborhood for any $\beta < \alpha$. Hence, for each $\beta < \alpha$ and neighborhood $V$ of $p$, there is $q \in V \cap A$ with $q \neq p$ and $\rank_A(q) \not< \beta$. By the induction hypothesis, this implies that $q$ is not removed by the $\beta$'th derivative. So, for each neighborhood $V$ of $p$ and each $\beta < \alpha$, there is $q \in V \cap A^{(\beta)}$ with $q \neq p$. So, $p$ is not isolated in $A^{(\beta)}$ for any $\beta < \alpha$. Therefore, $p \in A^{(\alpha)}$.
	
	Now, let $U$ be a $\leq\!\!\alpha$-isolating neighborhood of $p$ in $A$. By Lemma~\ref{Lem:SurroundedByLowerRanks} we have that each $q \in A \cap U$ with $q \neq p$ has $\rank_A(q) < \alpha$. By the induction hypothesis, all such $q$ are removed by the $\alpha$'th derivative. Therefore, $U$ isolates $p$ in $A^{(\alpha)}$. Hence, the $(\alpha + 1)$'th derivative removes $p$, so $\alpha$ is the least ordinal such that $p \in A^{(\alpha)} \setminus A^{(\alpha + 1)}$.
\end{proof}

By contrast, pointwise rank does not correspond to the usual Cantor-Bendixson rank; sets containing only isolated points have pointwise rank $0$ but Cantor-Bendixson rank $1$. On the other hand, a set with a point of each rank $\alpha < \omega$ but no points of rank $\omega$ has pointwise rank equal to its Cantor-Bendixson rank. In general, we have the inequality
\[
	\rankp(A) \leq \rank(A) \leq \rankp(A) + 1,
\]
where $\rank(A)$ denotes the usual Cantor-Bendixson rank of $A$. 

Another basic fact about rank relative to a represented space that is carried over from the standard notion of rank is that rank is preserved by open sets. 

\begin{Lemma}
	\label{Lem:OpenSetsPreserveRank}
	Suppose $A \subset \B$ and let $U \subset A$ be an open set. Then, for all $p \in A \cap U$, 
	\[
		\rank_A(p) = \rank_{A \cap U}(p).
	\]
\end{Lemma}
\begin{proof}
	We prove, by induction on $\alpha$, that for every $p \in A \cap U$, 
	\[
		\rank_A(p) \leq \alpha \iff \rank_{A\cap U}(p) \leq \alpha.
	\]
	This is sufficient because equality of ranks is defined by the least such ordinal $\alpha$.
	
	For $\alpha = 0$, a point $p$ has $\rank_{A \cap U}(p) = 0$ if and only if there is an open set $V$ such that $V \cap A \cap U = \{p\}$. Then, $V \cap U$ witnesses that $\rank_A(p) = 0$. Conversely, if $\rank_A(p) = 0$ then there is an open set $V$ such that $V \cap A = \{p\}$. Then, $V \cap U$ witnesses that $\rank_{A\cap U}(p) = 0$. Hence, the claim holds for $\alpha = 0$.

	Now fix $\alpha > 0$ and assume the claim holds for every $\beta < \alpha$.
	
	Suppose first that $\rank_A(p) \leq \alpha$, and let $W$ be a $\leq\!\!\alpha$-isolating neighborhood of $p$ in $A$. Then $W \cap U$ is an open neighborhood of $p$ in $A \cap U$. For each $q \in A \cap W \cap U$ with $q \neq p$, Lemma~\ref{Lem:SurroundedByLowerRanks} gives that $\rank_A(q) < \alpha$, so by the induction hypothesis $\rank_{A\cap U}(q) < \alpha$. Applying Lemma~\ref{Lem:SurroundedByLowerRanks} again, now inside $A \cap U$, it follows that $W \cap U$ is a $\leq\!\!\alpha$-isolating neighborhood of $p$ in $A \cap U$. Hence $\rank_{A\cap U}(p) \leq \alpha$.
	
	Conversely, suppose that $\rank_{A\cap U}(p) \leq \alpha$, and let $W$ be a $\leq\!\!\alpha$-isolating neighborhood of $p$ in $A \cap U$. For each $q \in A \cap W \cap U$ with $q \neq p$, Lemma~\ref{Lem:SurroundedByLowerRanks} gives $\rank_{A\cap U}(q) < \alpha$, so by the induction hypothesis, we obtain $\rank_A(q) < \alpha$. Applying Lemma~\ref{Lem:SurroundedByLowerRanks} in $A$, it follows that $W \cap U$ is a $\leq\!\!\alpha$-isolating neighborhood of $p$ in $A$. Hence $\rank_A(p) \leq \alpha$.
\end{proof}

The core argument for showing (\ref{MainTheoremItem:AboveACCN}) implies (\ref{MainTheoremItem:DiscontinuousOnPointWiseRankedSet}) in Theorem~\ref{Thm:MainTheorem} is due to Cenzer and Smith \cite{CenzerSmithRanks1989}. We provide an alternative proof in terms of ranked isolating neighborhoods.
\begin{Lemma}[\cite{CenzerSmithRanks1989}]
	\label{Lem:ContinuousImageMonotone}
	Suppose that $A \subset \N^\N$ is compact, countable, and that $\rankp(A) \leq \alpha$. Let $\Phi:A \to \N^\N$ be a continuous function. Then, $\rankp(\Phi(A)) \leq \alpha$.
\end{Lemma}
\begin{proof}
	We argue by induction on $\alpha$. For the base case $\alpha = 0$, since $A$ is compact and $\rankp(A) = 0$, we have that $A$ is the union of finitely many isolating neighborhoods. Hence, $A$ is finite, so $\Phi(A)$ is finite. Therefore, $\rankp(\Phi(A)) = 0$.

	Assume that the lemma is true for all $\beta < \alpha$. Let $q \in \Phi(A)$. We will show that $\rank_{\Phi(A)}(q) \leq \alpha$. Since $A$ is compact, there are only finitely many points $p \in A$ with $\rank_A(p) = \alpha$.	
	For incomparable strings $\rho$ and $\rho'$, $\Phi^{-1}(\llb \rho \rrb)$ and $\Phi^{-1}(\llb\rho'\rrb)$ are disjoint sets, hence only finitely many $\tau \sqsubset q$ have an extension $\rho\not\sqsubset q$ such that $\rankp(\Phi^{-1}(\llb \rho \rrb)) = \alpha$.
	Let $\tau \sqsubset q$ be such that for all $\rho$ with $\tau \sqsubset \rho \not\sqsubset q$ we have that $\rankp(\Phi^{-1}(\llb \rho \rrb)) < \alpha$. 
	
	We will show that $\llb\tau\rrb$ is a $\leq \!\!\alpha$-isolating neighborhood of $q$ in $\Phi(A)$.
	Let $r \sqsupset \tau$ be such that $r \neq q$. Let $\rho$ be such that $\tau \sqsubset \rho \sqsubset r$ and $\rho \not\sqsubset q$. By our assumption on $\tau$, there is $\beta < \alpha$ such that $\rankp(\Phi^{-1}(\llb \rho \rrb)) \leq \beta.$  
	
	Since $\llb \rho \rrb$ is closed, so is $\Phi^{-1}(\llb \rho \rrb)$. Hence, $\Phi^{-1}(\llb \rho \rrb) \subset A$ is a closed subset of a compact set, so is also compact. By the induction hypothesis and by Lemma~\ref{Lem:OpenSetsPreserveRank}, we have that 
	\[
	\rankp(\Phi(\Phi^{-1}(\llb \rho \rrb))) = \rankp(\Phi(A) \cap \llb \rho \rrb) \leq \beta.
	\] 
	Hence, $\rank_{\Phi(A) \cap \llb \rho \rrb}(r) \leq \beta$. Since $r$ is an arbitrary point in $\Phi(A) \cap \llb \tau \rrb$ with $r \neq q$, we have that $\llb \tau \rrb$ is a $\leq\!\!\alpha$-isolating neighborhood of $q$ in $\Phi(A)$ by Lemma~\ref{Lem:SurroundedByLowerRanks}. Hence, $\rank_{\Phi(A)}(q) \leq \alpha$. This is true for each $q \in \Phi(A)$, so we conclude that $\rankp(\Phi(A)) \leq \alpha$. 
\end{proof}

We will also need the following lemma, which is essentially Brouwer's bar theorem (see \cite{troelstraConstructivismMathematicsIntroduction1988}). We provide a self-contained proof for the reader's convenience.  

\begin{Lemma}\label{Lemma:CoverByNodes}
    Let $A \subset \N^\N$. Suppose that for all $a \in A$ there is a $\sigma \sqsubset a$ such that $R(\sigma)$ holds for some predicate $R$. Then, there are (possibly finitely many) incomparable strings $\rho_1, \rho_2, \dots$ such that 
    \[
    A = \bigsqcup_{i \in \N} A \cap \llb \rho_i \rrb
    \]
    and such that $R(\rho_i)$ holds for each $i \in \N$. 
\end{Lemma}
\begin{proof}
	Let $\sigma_1, \sigma_2, \dots$ be an enumeration of $\N^{<\N}$ such that, for all $i,j \in \N$ we have that $\sigma_i \sqsubset \sigma_j$ implies that $i \leq j$. We construct $\rho_1, \rho_2, \dots$ in stages. At the end of stage $n$, we will have defined $\rho_1, \dots, \rho_{k_n}$, where $k_n$ is the number of $\rho_i$ which have been enumerated after stage $n$. Before the construction starts, no $\rho_i$ have been chosen, so $k_0 = 0$. 
    
    At stage $n \geq 1$, suppose that $\rho_1, \dots, \rho_{k_{n-1}}$ have already been chosen. If $R(\sigma_n)$ holds and $\rho_i \not\sqsubset \sigma_n$ for all $i \leq k_{n-1}$ then let $\rho_{k_{n-1} + 1} = \sigma_n$ and let $k_n = k_{n-1} + 1$ and move on to stage $n + 1$. Otherwise, let $k_n = k_{n-1}$ and move on to stage $n + 1$.
    
	We claim that $\rho_1, \rho_2, \dots$ satisfy the lemma. 
    By construction, we have that $R(\rho_i)$ holds for each $i$. 
    Furthermore, $\rho_i$ is incomparable with $\rho_j$ for each $j \neq i \in \N$, so the $\llb \rho_i \rrb$ are pairwise disjoint and hence so are the $A \cap \llb \rho_i \rrb$. It remains to show that the $A \cap \llb \rho_i \rrb$ cover $A$.
    
	Fix $a \in A$. Then there is a string $\sigma \sqsubset a$ such that $R(\sigma)$ holds. Let $t$ be such that $\sigma = \sigma_t$. Then, we have that either $\rho_{k_t} = \sigma$ or that there is $i < k_t$ such that $\rho_i \sqsubset \sigma$. In either case, there is $j$ such that $a \in A \cap \llb \rho_j \rrb$, as desired. 
\end{proof}

Finally, we will need the following application of Lemma~\ref{Lemma:CoverByNodes} in Section~\ref{Sec:Conclusion}. 

\begin{Lemma}
	\label{Lem:CoverByIsolating}
	Suppose that $A \subset \N^\N$ has $\rankp(A) \leq \alpha$. Then, there are (possibly finitely many) $\sigma_1, \sigma_2, \dots$ and $p_1 \in \llb \sigma_1 \rrb$, $p_2 \in \llb \sigma_2 \rrb$, $\dots$ such that for each $i \in \N$, $\sigma_i$ is a $\rank_A(p_i)$-isolating neighborhood of $p_i$ in $A$ and such that $A = \bigsqcup_{i \in \N} A \cap \llb \sigma_i \rrb$.
\end{Lemma}
\begin{proof}
	Apply Lemma~\ref{Lemma:CoverByNodes} to the predicate 
	\[
		R(\sigma) := \text{ there is } p \in A \cap \llb \sigma \rrb \text{ such that } \sigma \text{ is a } \rank_A(p)\text{-isolating neighborhood of } p \text{ in } A.
	\]
	Since every point $p \in A$ has a $\rank_A(p)$-isolating neighborhood, the predicate $R(\sigma)$ holds for some $\sigma \sqsubset p$. Hence, by Lemma~\ref{Lemma:CoverByNodes}, the lemma is proved. 
\end{proof}

%% file: TheProblemsACCN.tex
In this section, we define the problem $\ACCN^\alpha$ for each countable ordinal $\alpha$. Informally, $\ACCN^\alpha$ is the problem of successfully solving $\ACCN$ at least once in a finite sequence of attempts indexed by $\alpha$, with the indexing chosen by the solver. We freely refer to notions and results from Section~\ref{Sec:LocalDiscontinuities}. At the same time, the problems defined in this section are the main examples in Section~\ref{Sec:LocalDiscontinuities}. 
We recommend reading the present section and Section~\ref{Sec:LocalDiscontinuities} in tandem, starting with the definition of the problems $\ACCN^\alpha$ and then reading Section~\ref{Sec:LocalDiscontinuities} before returning to the rest of this section. 

\subsection{Definition and Basic Properties}

To give a formal account of $\ACCN^\alpha$, we must fix a representation of $\alpha$. At first, it may seem natural to use an enumeration of $\alpha$. However, this does not allow the solver of $\ACCN^\alpha$ to choose the next ordinal. Instead, we model the way we represent ordinals after the way ordinals are manifested as a Cantor-Bendixson rank. 

\begin{Definition}
    \label{def:RankWitnessing}
    Fix an ordinal $\alpha$. We say that a function $\ell_\alpha: \N \to \alpha$ is \emph{rank-witnessing} for $\alpha$ if and only if 
    \begin{itemize}
        \item For each $i < j \in \N$ we have $\ell_\alpha(i) \leq \ell_\alpha(j)$. 
        \item For each $\beta < \alpha$ there is $i \in \N$ such that $\ell_\alpha(i) \geq \beta$. 
    \end{itemize}
    Furthermore, we say that a set of functions $\mc{L} = \{\ell_\beta: 0 < \beta \leq \alpha\}$ is a \emph{rank-witnessing representation} of $\alpha$ if and only if for each $\beta \leq \alpha$ we have that $\ell_\beta$ is a rank-witnessing function for $\beta$.
\end{Definition}

A rank-witnessing function can be obtained for any countable ordinal by skipping decreases in its enumeration. On the other hand, a rank-witnessing function for $\alpha$ is cofinal in $\alpha$, so if $\aleph_1$ is a regular cardinal then there is no rank-witnessing representation of $\omega_1$. Hence, if $\aleph_1$ is regular then an ordinal $\alpha$ has a rank-witnessing representation if and only if $\alpha$ is countable. 

We now define the problems $\ACCN^\alpha$ as multi-functions from $\B$ to $\B$. While we could define represented spaces for this purpose, we find it more illustrative to define the realizer version directly. We represent both the problem-poser and problem-solver waiting by the character $\hash$ and we code the ordinal for the next instance of $\ACCN$ by taking the number $j$ from the prefix $\hash^j1$ and interpreting it as the ordinal $\ell_\alpha(j)$. Since $\ell_\alpha$ is rank-witnessing for $\alpha$, the solver may then ensure that they have at least $\beta$-many guesses remaining by searching the instance until finding $\hash^j$ such that $\ell_\alpha(j) \geq \beta$.  

\begin{Definition}
    \label{Def:ACCNAlpha}
    Fix a countable ordinal $\alpha$ and a rank-witnessing representation $\mc{L} = \{\ell_\beta: 0 < \beta \leq \alpha\}$. We define the problem $\ACCN^\alpha: \subset (\N \cup \{\hash\})^\N \tto (\N \cup \{\hash\})^\N$ recursively as follows.
    \begin{itemize}
        \item $\dom(\ACCN^0) = \{\hash^\N\}$
        \item For $\alpha > 0$, 
        \[
            \dom(\ACCN^\alpha) = \{\hash^\N\} \cup  \{\hash^j n w: j,n \in \N \text{ and } w \in \dom\left(\ACCN^{\ell_{\alpha}(j)}\right)\}.
        \]
    \end{itemize}
    \begin{itemize}
        \item $\ACCN^0(\hash^\N) = \emptyset 
        $
        \item and for $\alpha > 0$,
        \[ 
            \ACCN^\alpha(p) = \begin{cases}
               \{m q:m \in \N \text{ and } q \in (\N \cup \{\hash\})^\N\} &\text{if } p = \hash^\N \\
                \{m q:  n \neq m\in \N \text{ and } q \in (\N \cup \{\hash\})^\N\}\cup  &\text{if } p = \hash^j n w \text{ for some } n \in \N.\\
                 \{n\hash^k y: k \in \N \text{ and } y \in \ACCN^{\ell_\alpha(j)}(w)\}  & 
            \end{cases}
        \]
    \end{itemize}
\end{Definition}

Informally, to solve $\ACCN^\alpha$, an agent is allowed $\alpha$ many attempts to solve $\ACCN$. This intuition directly leads to the fact that $\ACCN^\beta$ is discontinuous for each countable $\beta$.

\begin{Theorem}
    For each countable ordinal $\alpha$ and rank-witnessing representation of $\alpha$, the associated problem $\ACCN^\alpha$ is not continuous.  
\end{Theorem}
\begin{proof}
    We prove the claim by transfinite induction on $\alpha$. For $\alpha=0$, the problem $\ACCN^0$ has no realizer, so it is not continuous.

    Fix $\alpha>0$, and assume that $\ACCN^\beta$ is not continuous for every $\beta<\alpha$. Suppose, for contradiction, that $\Phi$ is a continuous realizer of $\ACCN^\alpha$.

    Since $\Phi(\hash^\N)\in \ACCN^\alpha(\hash^\N)$, there is $n\in\N$ with $\Phi(\hash^\N)(0)=n$. By continuity, there is $j\in\N$ such that
    \[
        \Phi(\llb \hash^j\rrb)\subseteq \llb n\rrb.
    \]
    In particular, for every $w\in\dom\left(\ACCN^{\ell_\alpha(j)}\right)$, the input $\hash^j n w$ is in $\dom(\ACCN^\alpha)$ and
    \[
        \Phi(\hash^j n w)\in \ACCN^\alpha(\hash^j n w).
    \]
    Because $\Phi(\llb\hash^j\rrb) \subset \llb n \rrb$, the first branch in the definition of $\ACCN^\alpha(\hash^j n w)$ is impossible. Hence,
    \[
        \Phi(\hash^j n w)=n\hash^{k(w)}y_w
    \]
    for some $k_w \in\N$ and $y_w \in \ACCN^{\ell_\alpha(j)}(w)$.

    Define the partial continuous function $T: n\hash^k y \mapsto y$. Then define
    \[
        \Psi(w):=T\left(\Phi(\hash^j n w)\right).
    \]
    Then $\Psi$ is continuous, and for every $w\in\dom(\ACCN^{\ell_\alpha(j)})$ we have
    \[
        \Psi(w)\in\ACCN^{\ell_\alpha(j)}(w).
    \]
    So $\Psi$ is a continuous realizer of $\ACCN^{\ell_\alpha(j)}$. But $\ell_\alpha(j)<\alpha$, contradicting the induction hypothesis.

    Therefore $\ACCN^\alpha$ is not continuous.
\end{proof}

The discontinuity of $\ACCN^\alpha$ is witnessed by $\hash^\N$, the sole point of discontinuity of $\ACCN^\alpha$. 

\begin{Proposition}\label{Prop:ContinuityOffLine}
    Let $\tau \in (\N \cup \{\hash\})^{<\N}$ and let $\alpha$ be a countable ordinal. Then,  $\ACCN^\alpha|_{\llb \tau \rrb}$ is continuous if and only if $\tau \not\sqsubset \hash^\N$ .
\end{Proposition}
\begin{proof}
    ($\impliedby$): Since $\tau \not\sqsubset \hash^\N$, we have that $\tau \sqsupset \hash^j n$ for some $j \geq 0$ and some $n \in \N$. For $\alpha = 0$, we have $\dom(\ACCN^0) \cap \llb \tau \rrb = \emptyset$, so any partial continuous function is a continuous realizer for $\ACCN^0|_{\llb \tau \rrb}$. For $\alpha > 0$ we have that for each $p\in\dom(\ACCN^\alpha)$ it is the case that $\hash^j n \sqsubset p$ implies that $(n+1)\hash^\N \in\ACCN^\alpha(p)$ by Definition~\ref{Def:ACCNAlpha}. Hence the constant map $p\mapsto (n+1)\hash^\N$ is a continuous realizer of $\ACCN^\alpha|_{\llb\tau\rrb}$.

    ($\implies$): By the previous direction, $\ACCN^\alpha$ is continuous at each element of its domain other than $\hash^\N$. Since $\ACCN^\alpha$ is discontinuous, it has a point of discontinuity by Theorem~\ref{Thm:ContinuityIsContinuityAtAllPoints}, which can only be $\hash^\N$. 
\end{proof}

In fact, we can say something slightly stronger about the subsets on which $\ACCN^\alpha$ is discontinuous. 

\begin{Proposition} \label{Prop:ACCNContOnLowRank}
    Let $\alpha$ be a countable ordinal and let $A \subset \dom(\ACCN^\alpha)$. Then,  $\ACCN^\alpha|_A$ being discontinuous implies that $\hash^\N \in A$ and $\rank_A(\hash^\N) = \alpha$.
\end{Proposition}
\begin{proof}
    We argue by induction on $\alpha$. The base case $\alpha = 0$ is trivial: $\dom(\ACCN^0) = \{\hash^\N\}$, so there are only two cases: either $A = \{\hash^\N\}$, in which case $\rank_A(\hash^\N) = 0$, or $A = \emptyset$, in which case $\ACCN^0|_A$ is continuous.
    
    Now suppose that $\alpha > 0$. We can always build a continuous realizer of $\ACCN^\alpha|_{\dom(\ACCN^\alpha) \setminus \{\hash^\N\}}$ by pasting together the realizers of $\ACCN^\alpha|_{\llb \hash^i n \rrb}$ obtained from Proposition~\ref{Prop:ContinuityOffLine}. Hence, we may assume that $\hash^\N \in A$.
    
    Suppose that $\rank_A(\hash^\N) = \beta < \alpha$. We show that $\ACCN^\alpha|_A$ is continuous. Now, suppose for contradiction that for each $j$ there is $i > j$ and $n \in \N$ such that $\rank_{A}(\hash^i n \hash^\N) \geq \ell_\alpha(j)$. Since $\ell_\alpha$ is rank-witnessing, this implies that $\rank_A(\hash^\N) = \alpha$, a contradiction. Hence, there is $j$ such that for each $i > j$ and $n \in \N$ we have that $\rank_A(\hash^i n \hash^\N) < \ell_\alpha(j)$. For each such $i,n$, let
    \[
        B_{i,n} = \{p \in \dom(\ACCN^{\ell_\alpha(i)}): \hash^i n p \in A\}. 
    \]
    Then, $\rank_{B_{i,n}}(\hash^\N) = \rank_A(\hash^i n \hash^\N) < \ell_\alpha(j)$.
    
    By the inductive hypothesis, $\ACCN^{\ell_\alpha(i)}$ is continuous on $B_{i,n}$ for each $i > j$ and $n \in \N$, say via continuous realizer $\Phi_{i,n}$. Define $\Phi: \subset \B \to \B$ by $\Phi(\hash^\N) = 0 \hash^\N$ and
    \[
        \Phi(\hash^i n w) = \begin{cases}
            (n+1)\hash^\N &\text{if } i \leq j\\
            0\hash^{i - j}\Phi_{i,n}(w) &\text{if } i  > j.
        \end{cases}
    \]
    Then, $\Phi$ is a continuous realizer for $\ACCN^\alpha|_A$. 
\end{proof}


\begin{Proposition}
    	\label{Thm:Seperation}
	Suppose $\beta < \alpha$. Then, 
	\[
		\ACCN^\alpha <_\W^* \ACCN^\beta.  
	\]
\end{Proposition}
The most direct proof of this proposition would rehash many of the details of the proof of Theorem~\ref{Thm:MainTheorem}. Instead, we give a more concise proof by directly invoking Theorem~\ref{Thm:MainTheorem}. We do not use this result to prove either Theorem~\ref{Thm:MainTheorem} or Theorem~\ref{Thm:MainTheoremPointWise}. 
\begin{proof}
    To see that $\ACCN^\alpha \leq_\W^* \ACCN^\beta$, note that $\ACCN^\beta$ is discontinuous and that $\rank_{\dom(\ACCN^\beta)}(\hash^\N) = \beta < \alpha$ and apply Theorem~\ref{Thm:MainTheorem}.

    To see that $\ACCN^\alpha \not\geq_\W^* \ACCN^\beta$, suppose that $\ACCN^\alpha|_A$ is discontinuous for some $A \subset \dom(\ACCN^\alpha)$. Then, \[\rank_A(\hash^\N) = \alpha > \beta\] by Proposition~\ref{Prop:ACCNContOnLowRank}. Hence, $\ACCN^\alpha$ is not discontinuous on a set of pointwise Cantor-Bendixson rank $\leq \beta$, not satisfying Condition~\ref{MainTheoremItem:DiscontinuousOnPointWiseRankedSet} of Theorem~\ref{Thm:MainTheorem}.
\end{proof}

The set-theoretic identity of $\ACCN^\alpha$ depends on our choice of rank-witnessing representation of $\alpha$. However, this is not true for its continuous Weihrauch degree. This is because if the agent has $\beta$ many attempts remaining, they may ensure that they retain at least $\gamma < \beta$ attempts in the next round by checking the instance of $\ACCN$ up to some $j$ with $\ell_\beta(j) \geq \gamma$, so solutions for different rank-witnessing representations of $\alpha$ are continuously equivalent. We may also view this equivalence as a corollary of Theorem~\ref{Thm:MainTheorem}. We only invoke Theorem~\ref{Thm:EllIndependent} establish the robustness of our definition of $\ACCN^\alpha$; we do not use it in any proofs. 

\begin{Theorem} \label{Thm:EllIndependent}
	Suppose that the problems $\{{\ACCN^\beta}\}_{\beta \leq \alpha}$ are defined based on rank-witnessing representation $\mc{L} = \{{\ell_\beta}: \beta \leq \alpha\}$, while the problems  $\{{(\ACCN^\beta)}'\}_{ \beta \leq \alpha}$ are defined based on $\mc{L}' = \{{\ell_\beta}': \beta \leq \alpha\}$. Then,
	\[
	\ACCN^\beta \equiv_\W^* (\ACCN^\beta)'
	\]
	for all $0 \leq \beta \leq \alpha$. 
\end{Theorem}
\begin{proof}
	The point $\hash^\N$ is a rank-$\beta$ discontinuity of both $\ACCN^\beta$ and $(\ACCN^\beta)'$. Therefore, by Theorem~\ref{Thm:MainTheoremPointWise}, 
	\[
		\ACCN^\beta \leq_\W^* (\ACCN^\beta)'.
	\]
	The proof of Theorem~\ref{Thm:MainTheoremPointWise} works for any $\mc{L}'$ as well,  so 
	\[
	\ACCN^\beta \geq_\W^* (\ACCN^\beta)'
	\]
	follows as well. 
\end{proof}

%% file: LocalDiscontinuities.tex
\subsection{Continuity at a Point}
In computable analysis, continuity of multivalued functions between represented spaces is defined in terms of the existence of continuous realizers (see Section~\ref{Sec:Background}). This is a global notion of continuity. We can restrict it to local neighborhoods of points in the domain to obtain a local notion of continuity. 

\begin{Definition}\label{def:continuous_at_a_point}
    Let $P: \subset \N^\N \tto \N^\N$ be a problem on Baire space. 
    For $p \in \dom(P)$, we say that $P$ is \emph{continuous at} $p$ if there is an open neighborhood $U$ of $p$ such that $P|_U$ is continuous. We say that $P$ is \emph{discontinuous} at $p \in \dom(P)$ if $P$ is not continuous at $p$. 
\end{Definition}
\begin{Remark}
For a single valued function $f: \subset \N^\N \to \N^\N$, $f$ is continuous at $p$ in the sense of Definition~\ref{def:continuous_at_a_point} if and only if $f$ is ``locally continuous'' at $p$ in the sense that there is an open neighborhood $U$ of $p$ such that $f|_U$ is continuous. This, of course, is not equivalent to the classical definition of continuity at a point. Indeed, more descriptive terms for Definition~\ref{def:continuous_at_a_point} might be \emph{local continuity} or \emph{local realizability}. However, our notion of rank of discontinuity collapses to only three levels for single valued functions (see Section~\ref{sec:SingleValuedFunctions}). Hence, this is a minor terminological issue and we will stick with the current terminology for the sake of brevity.
\end{Remark}

While Definition~\ref{def:continuous_at_a_point} gives a notion of discontinuity at a point, it is not self-contained in the sense that it makes reference to the global notion of continuity in terms of realizers. In Section~\ref{Sec:RankingDiscontinuities}, we give a purely local definition of continuity at a point for functions which are discontinuous on scattered sets. For now, we will work with Definition~\ref{def:continuous_at_a_point}. 

In this section, we will often reprove results on continuity at a point of multi-valued functions which are classical for single-valued functions. In particular, in Theorem~\ref{Thm:ContinuityIsContinuityAtAllPoints}, we show that $P$ is continuous if and only if $P$ is continuous at each point of its domain. This is true for single valued functions between any pair of topological spaces. However, we only prove it for multivalued functions on Baire space, as our proof relies on every subspace being a disjoint union of open sets.

\begin{Lemma}\label{Lem:DiscontinuityPreservedByOpenRestrictions}
	Let $P: \subset \B \tto \B$ be a problem and let $p \in \dom(P)$. Then $P$ is discontinuous at $p$ if and only if for all open sets $U \subset \B$ with $p \in U$ we have that $P|_U$ is discontinuous at $p$. 
\end{Lemma}
\begin{proof}
	($\implies$): Suppose that $P$ is discontinuous at $p$. Fix open $U$ with $p \in U$. Since $p \in U$, we have that either $P|_U$ is discontinuous at $p$ or that $P|_U$ is continuous at $p$. Suppose that $P|_U$ is  continuous at $p$. Then there is a neighborhood $V$ of $p$ such that $P|_{U \cap V}$ is continuous. Since $U \cap V$ is open, this is a contradiction with the fact that $P$ is  discontinuous at $p$. 

	($\impliedby$): Follows immediately from the fact that $\B$ is an open set. 
\end{proof}

A basic fact about continuous solvability of a problem $P: \subset \mathbf{X} \tto \mathbf{Y}$ on represented spaces is that if the set of representations of each point in $X$ is open and each instance of $P$ has a solution, then $P$ is continuous. Similarly, we obtain a local analogue of this fact.

\begin{Proposition}\label{prop:ContinuousAtOpenPoints}
	Let $P: \mathbf{X} \tto \mathbf{Y}$ be a problem with domain $\mathbf{X} = (X,\delta_{\mathbf{X}})$ and fix $x \in X$. Suppose that $\repr{X}^{-1}(x) \subset \B$ is an open set in the subspace topology of $\dom(\repr{X}) \subset \B$ and that $P(x)$ is nonempty. Then, $P^\mathbf{r}$ is continuous at each $p \in \repr{X}^{-1}(x)$. 
\end{Proposition}
\begin{proof}
	Fix $p\in \B$ such that $\repr{X}(p) = x$. Since $\deltarep{X}^{-1}(x)$ is open, there is $\sigma \sqsubset p$ such that $\deltarep{X}(\llb \sigma \rrb \cap \dom(\deltarep{X})) = \{x\}$. Let $q \in \B$ be such that $\deltarep{Y}(q) \in P(x)$. Then, the constant function $\Phi(r) = q$ is a realizer of $\realizerver{P}|_{\llb \sigma \rrb}$. 
\end{proof}
For example, we can use Proposition~\ref{prop:ContinuousAtOpenPoints} to conclude that all discontinuities of a problem with domain $\overline{\N}$ are names for $\bot$. 

Discontinuity at a point is Weihrauch invariant in the sense that it is preserved by the forward functional of any Weihrauch reduction.

\begin{Theorem}\label{Thm:DiscontinuityToDiscontinuity}
	Let $P: \subset \B \tto \B$ and $Q: \subset \B \tto \B$ be problems on Baire space. Suppose that $a \in \dom(P)$ is a discontinuity of $P$ and that $P \leq_\mathrm{W}^* Q$ via $\Phi$ and $\Psi$. Then, $\Phi(a)$ is a discontinuity of $Q$. 
\end{Theorem}
\begin{proof}
	Suppose that $F$ is a continuous realizer of $Q$ on some open neighborhood $U$ of $\Phi(a)$. Then $\Psi \circ \langle \id ,F \circ \Phi \rangle$ is a realizer of $P$ on $\Phi^{-1}(U)$, which is an open neighborhood of $a$. 
\end{proof}

Theorem~\ref{Thm:DiscontinuityToDiscontinuity} narrows the scope of Weihrauch reductions to those which map discontinuities to discontinuities. 

As is the case for single valued problems, the existence of a global realizer is equivalent to the existence of local realizers. We now show that continuity of $P$ is equivalent to $P^\mathrm{r}$ being continuous at all points of its domain. 

\begin{Theorem}\label{Thm:ContinuityIsContinuityAtAllPoints}
    Let $P: \subset \B \tto \B$ be a problem on Baire space. The following are equivalent.
    
    \begin{enumerate}
        \item \label{item:continuous} $P$ is continuous.
        
        \item \label{item:continuous_at_a_point} For all $p \in \dom(P)$, $P$ is continuous at $p$.
    \end{enumerate}
\end{Theorem}

\begin{proof}
    (\ref{item:continuous} $\Rightarrow$ \ref{item:continuous_at_a_point}) Suppose that $P$ is continuous. Then, there is a continuous realizer $\Phi$ of $P$. $\Phi$ witnesses that $P$ is continuous on the open neighborhood $\dom(P)$ of any $p \in \dom(P)$.
    
    (\ref{item:continuous_at_a_point} $\Rightarrow$ \ref{item:continuous}) Let $R$ be the relation
		
		\[
			R(\sigma) \iff P \text{ is continuous on } \llb \sigma \rrb.
		\]

        Set $A = \dom(P)$. Fix $a \in A$. By assumption, $P$ is continuous at $a$, so there is an open neighborhood $U$ of $a$ such that $P$ restricted to $U$ is continuous. Let $\sigma \sqsubset a$ be such that $\llb \sigma \rrb \subset U$. Then, $P$ is continuous on $\llb \sigma \rrb$, so $R(\sigma)$ holds. 

        By Lemma~\ref{Lemma:CoverByNodes}, there are incomparable strings $\rho_1, \rho_2, \dots$ such that 
        \[
		A = \bigsqcup_{i \in \N} A \cap \llb \rho_i \rrb
        \]
		and such that $P|_{\llb \rho_i \rrb}$ is continuous for each $i \in \N$. For each $i \in \N$, let $\Phi_i: \subset \llb \rho_i \rrb\to \B$ be a continuous realizer of $P|_{\llb \rho_i \rrb}$.
        
        Define $\Phi: \subset \B \to \B$ by
		\[ 
		\Phi = \bigcup_{i \in \N} \Phi_i.
		\]
		
		We claim that $\Phi$ is a continuous realizer of $P$. We have that $\Phi$ realizes $P$ since for each $a \in A$ there is $\rho_i$ such that $a \in A \cap \llb \rho_i \rrb$, so $a \in \dom(\Phi_i)$. Since $\Phi_i$ is a realizer of $P|_{\llb \rho_i \rrb}$, we have that 
		\[
		\Phi(a) = \Phi_i(a) \in P(a).
		\]	
		$\Phi$ is continuous because the union of continuous partial functions with disjoint clopen domains is continuous.  
\end{proof}

Pauly and Sold\`a prove the following lemma in their proof of the main theorem of \cite{SoldaPaulySequentialDiscontinuity}. 

\begin{Lemma}[\cite{SoldaPaulySequentialDiscontinuity}]\label{Lemma:SoldaPaulyLemma}
	Let $P: \subset \B \tto \B$ be a multivalued function whose instances all have solutions. Suppose that there is a convergent sequence of points $a_1, a_2, \dots \rightarrow a \in \B$ such that $P|_{\{a\} \cup \{a_i: i \in \N\}}$ is not continuous. Then, for each $y \in P(a)$ there is $\rho \sqsubset y$ such that for all $\sigma \sqsubset a$ there is $a_i \sqsupset \sigma$ such that $P(a_i) \cap \llb \rho \rrb = \emptyset$. 
\end{Lemma}

For a single valued function $P: \subset \B \to \B$, the conclusion of Lemma~\ref{Lemma:SoldaPaulyLemma} is exactly the $\epsilon$-$\delta$ definition of discontinuity at $a$. Even in the multivalued case, Lemma~\ref{Lemma:SoldaPaulyLemma} gives us an $\epsilon$-$\delta$ characterization of discontinuity at a point which does not require reference to a global notion of continuity. However, this makes essential use of $P$ being discontinuous on a set of Cantor-Bendixson rank $1$. 

Theorem~\ref{Thm:ExtendingSolutions} extends Lemma~\ref{Lemma:SoldaPaulyLemma} to points of discontinuity in general. On one hand, Theorem~\ref{Thm:ExtendingSolutions} circularly references discontinuity in a way which can only be grounded via the global notion of continuity given by realizers. On the other hand, by ranking discontinuities, we shall obtain a self-contained definition of continuity at a point for functions which are discontinuous on scattered sets --- that is, on sets which admit a Cantor-Bendixson rank. 

\begin{Theorem} \label{Thm:ExtendingSolutions}
    Let $P: \subset \B \tto \B$ be a problem and let $p \in \dom(P)$. The following are equivalent. 		
    \begin{enumerate}
        \item \label{item:Discont}$p$ is a discontinuity of $P$. 
        \item \label{item:NodeCharacterization}For each $q \in P(p)$ there is $\rho \sqsubset q$ such that for all $\sigma \sqsubset p$ there is $\tau$ with $\sigma  \sqsubset \tau \not\sqsubset p$ such that $P|\left[\llb \tau \rrb \times \llb\rho\rrb \right]$ is not continuous. In symbols,
        \[
            (\A q \in P(p))(\E \rho \sqsubset q)(\A \sigma \sqsubset p)(\E \tau)(\sigma \sqsubset\tau \not\sqsubset p \wedge P|\left[\llb \tau \rrb \times \llb\rho\rrb \right] \text{\emph{is not continuous}})
        \]
    \end{enumerate} 
\end{Theorem}
\begin{proof}
	(\ref{item:NodeCharacterization}) $\implies$ (\ref{item:Discont}): Suppose for contradiction that $P$ is continuous on a neighborhood $U \subset \dom(P)$ of $p$, witnessed by continuous realizer $\Phi: U \to \B$. By (\ref{item:NodeCharacterization}), there is $\rho \sqsubset \Phi(p)$ such that for all $\sigma \sqsubset p$ there is $\tau$ with $\sigma \sqsubset \tau \not \sqsubset p$ such that $P|\left[{\llb \tau \rrb \times \llb \rho \rrb}\right]$ is not continuous. By continuity of $\Phi$, there is $\sigma' \sqsubset p$ such that $\Phi(\llb \sigma' \rrb) \subset \llb \rho \rrb$. Let $\tau$ witness (\ref{item:NodeCharacterization}) for $\sigma = \sigma'$. Since $\sigma' \sqsubset \tau$, we have that  
	\[
		\Phi(\llb \tau \rrb) \subset \llb \rho \rrb. 
	\]
	Hence, $\Phi$ is a realizer for $P|\left[{\llb \tau \rrb \times \llb \rho \rrb}\right]$, contradicting that $P|\left[{\llb \tau \rrb \times \llb \rho \rrb}\right]$ is not continuous. 
	
	(\ref{item:Discont}) $\implies$ (\ref{item:NodeCharacterization}): Let $p$ be a discontinuity of $P$ and suppose that (\ref{item:NodeCharacterization}) is false. Then, there is $q \in P(p)$ such that for all $\rho \sqsubset q$ there is $\sigma \sqsubset p$ such that for all $\tau$ with $\sigma \sqsubset \tau \not\sqsubset p$ we have that $P|\left[{\llb \tau \rrb \times \llb \rho \rrb}\right]$ is continuous. In symbols, 
	\begin{equation}
		\label{cond:star}
		\tag{$\ast$}
		(\exists q \in P(p))(\forall \rho \sqsubset q)(\exists \sigma \sqsubset p)(\forall\tau\sqsupset\sigma)(\tau \not \sqsubset p \rightarrow P|\left[{\llb \tau \rrb \times \llb \rho \rrb}\right] \text{ is continuous}).
	\end{equation}
	We will construct a continuous realizer for $P$ on a neighborhood of $p$. 
	
	Let $q\in P(p)$ witness (\ref{cond:star}). For each $i \in \N$, let $\rho_i = q|_{i}$. Let $\sigma_i$ be a witness to (\ref{cond:star}) for $\rho_i$. Since $\sigma$ being a witness to (\ref{cond:star}) for $\rho$ implies that each $\sigma' \sqsupset \sigma$ is a witness to (\ref{cond:star}) for $\rho$, we can assume that 
	\[
	\sigma_0 \sqsubset \sigma_1 \sqsubset \dots
	\]
	
	is a strictly increasing sequence of strings.
	For each $k \in \N$, let $\ell(k) \in \N$ be greatest such that $\sigma_{\ell(k)} \sqsubset  p|_k$. 
	We deduce from ($\ast$)  that for each $i,k \in \N$ with $i \neq p(k)$, we have that $P|\left[ \llb p|_k \cat i \rrb \times \llb \rho_{\ell(k)}\rrb\right]$ is continuous. Let $\Phi_{i,k}$ be a continuous realizer for $P|\left[\llb{p|_{k}}{^\frown} i \rrb \times \llb  \rho_{\ell(k)} \rrb \right]$. Then $\Phi_{i,k}(x) \in \llb \rho_{\ell(k)} \rrb$ for all $x \in \dom(P) \cap \llb p|_k \cat i \rrb$. In fact, we assume without loss of generality that $\Phi_{i,k}(x) \in \llb \rho_{\ell(k)} \rrb$ for all $x \in \dom(\Phi_{i,k})$. 
	
	We now define continuous realizer $\Phi$ of $P|_{\llb \sigma_0 \rrb}$ by
	\[
	\Phi(x) = \begin{cases}
		q &\text{if } x = p\\
		\Phi_{i,k}(x) &\text{if } x = {p|_{k}}{^\frown} i {^\frown} r \text{ for some } i,k \in \N \text{ and } r \in \B.
	\end{cases}
	\]
	
	Since sets of the form $\llb{p|_{k}}{^\frown} i\rrb$ form a disjoint cover of $\llb \sigma_0 \rrb \setminus \{p\}$, we have that $\Phi$ is well defined and that $\dom(\Phi) \supset \dom(P)$. 
	The function $\Phi$ realizes $P_{\llb \sigma_0\rrb}$ because $q \in P(p)$ and the $\Phi_{i,k}$ each realize $P$ on $\llb p|_k{^\frown} i \rrb$. 
	The function $\Phi$ is continuous because its value on $p$ can be represented by sending $p|_{k}$ to $\rho_{\ell(k)}$. 
	This is consistent with each $\Phi_{i,k}$ because the range of $\Phi_{i,k}$ is contained in the cone above $\rho_{\ell(k)}$. Then, $\llb \sigma_0 \rrb$ is a neighborhood of $p$ on which $P$ has a continuous realizer, contradicting that $p$ is a discontinuity of $P$.
\end{proof}

To show that (\ref{MainTheoremPointwiseItem:AboveACCN}) implies (\ref{MainTheoremPointwiseItem:DiscontinuousOnPointWiseRankedSet}) in Theorem~\ref{Thm:MainTheoremPointWise}, we construct a compact subset of $\dom(\ACCN^\alpha)$ on which $\ACCN^\alpha$ is discontinuous. Conveniently, showing that $\ACCN^\alpha$ is discontinuous on this set is well suited to illustrating Theorem~\ref{Thm:ExtendingSolutions}. 

\begin{Theorem}\label{Thm:FollowPhiOnACompactSubsetOfDiscontinuity}
    Let $P: \subset \B \tto \B$ be a problem and let $\alpha$ be a countable ordinal. Then, $P \geq_\W^* \ACCN^\alpha$ implies that there is a set $A \subset \dom(P)$ such that $P|_A$ is discontinuous and $\rankp(A) \leq \alpha$. Furthermore, if $P \geq_\W^* \ACCN^\alpha$ is witnessed by forward function $\Phi$ and backward function $\Psi$, then $\Phi(\hash^\N)$ is a point of discontinuity of $P|_A$. 
\end{Theorem}
\begin{proof}
    First we show that the domain of $\ACCN^\alpha$ contains a compact subset $A_\alpha \subset \dom(\ACCN^\alpha)$ such that $\ACCN^\alpha|_{A_\alpha}$ is discontinuous. Define $A_\alpha$ recursively by
    \begin{itemize}
        \item $A_0 = \{\hash^\N\}$
        \item For $\alpha > 0$, 
        \[
            A_\alpha = \{\hash^\N\} \cup  \{\hash^j n w: n < j \in \N \text{ and } w \in A_{\ell_\alpha(j)}\}.
        \]
    \end{itemize}
    The sets $A_\alpha$ are compact because they are each the paths through a finitely branching tree. We show that $\ACCN^\alpha|_{A_\alpha}$ is discontinuous by showing that $\hash^\N$ satisfies Condition~\ref{item:NodeCharacterization} of Theorem~\ref{Thm:ExtendingSolutions}.

	We argue by induction. In the base case, $\hash^\N$ is a discontinuity of $\ACCN^0|_{A_0}$ because $\ACCN^0(\hash^\N) = \emptyset$. In the inductive step, fix $\alpha > 0$ and suppose that $\hash^\N$ is a discontinuity of $\ACCN^\beta|_{A_\beta}$ for each $\beta < \alpha$. 

	Fix a solution $y \in \ACCN^\alpha(\hash^\N)$. Then, $y = nq$ for some $n \in \N$ and $q \in (\N \cup \{\hash\})^\N$. Let $\rho = y|_1 = n$ and fix a $\sigma = \hash^j \subset \hash^\N$. Fix a $j' > j,n$ and let $\tau = \hash^{j'}n$. Then, $\tau \not\sqsubset \hash^\N$ and $\tau \sqsupset \sigma$. Furthermore, $\llb \tau \rrb \cap A_\alpha = \{\tau q : q \in A_{\ell_\alpha(j')}\}$. Thus, 
	\[
	(\ACCN^\alpha|_{A_\alpha})|\left[{\llb \tau \rrb \times \llb \rho \rrb}\right] \equiv_{\sW} \ACCN^{\ell_\alpha(j')}|_{A_{\ell_\alpha(j')}},
	\]
	witnessed by forward function $q \mapsto \tau q$ and backward function $\rho \hash^m q \mapsto q$ in the $\geq_{\sW}$ reduction and witnessed by forward $\tau q \mapsto q$ and backward $w \mapsto \rho w$ for the $\leq_{\sW}$ reduction.
		
	By the inductive hypothesis, the point $\hash^\N$ is a discontinuity of $\ACCN^{\ell_\alpha(j')}|_{A_{\ell_\alpha(j')}}$, so, in particular, $(\ACCN^\alpha|_{A_\alpha})|\left[{\llb \tau \rrb \times \llb \rho \rrb}\right]$ is discontinuous. Hence, $\hash^\N$ satisfies Condition~\ref{item:NodeCharacterization} of Theorem~\ref{Thm:ExtendingSolutions} and is a discontinuity of $\ACCN^\alpha|_{A_\alpha}$.

	Now suppose that $P \geq_\W^* \ACCN^\alpha$ via continuous functionals $\Phi$ and $\Psi$. Then, the same functionals witness that $P \geq_\W^* \ACCN^\alpha|_{A_\alpha}$. By Lemma~\ref{Lem:ContinuousImageMonotone} the set $A = \Phi(A_{\alpha})$ is compact and has $\rankp(A) \leq \rankp(A_\alpha) = \alpha$. By Theorem~\ref{Thm:DiscontinuityToDiscontinuity}, $p = \Phi(\hash^\N)$ is a discontinuity of $P|_A$, as desired.
\end{proof}

The image of $\ACCN^{\alpha}|_{A_\alpha}$ under $\Phi$ in the above proof is also compact. Hence, we obtain the following corollary from Theorem~\ref{Thm:MainTheoremPointWise}.

\begin{Corollary}
	\label{Cor:CompactnessEquivalence}
	Let $P: \subset \B \tto \B$ be a problem and let $\alpha$ be a countable ordinal. The following are equivalent.
	\begin{enumerate}
		\item $P$ is discontinuous on a set of pointwise rank at most $\alpha$.
		\item $P$ is discontinuous on a compact set of pointwise rank at most $\alpha$.
	\end{enumerate}
\end{Corollary}

\subsection{Ranking Discontinuities and the Discontinuity Game}
\label{Sec:RankingDiscontinuities}
The phenomenon observed in Lemma~\ref{Lemma:SoldaPaulyLemma} is a special case of Theorem~\ref{Thm:ExtendingSolutions} in the following sense. Suppose that $P|[\llb \tau_0 \rrb \times \llb \rho_0 \rrb]$ is one of the witnesses that a point $a_0$ is a discontinuity of $P$ as in Condition~\ref{item:NodeCharacterization} of Theorem~\ref{Thm:ExtendingSolutions}. By Theorem~\ref{Thm:ContinuityIsContinuityAtAllPoints}, we have that $P|[\llb \tau_0  \rrb \times \llb \rho_0  \rrb]$ also has a point of discontinuity $a_1$. It is possible that $P(a_1) \cap \llb \rho_0  \rrb$ is empty, as in Lemma~\ref{Lemma:SoldaPaulyLemma}. If $P(a_1) \cap \llb \rho_0  \rrb$ is nonempty, then we may iteratively apply Theorem~\ref{Thm:ExtendingSolutions} to obtain sequences $a_2, a_3, \dots$ and $\rho_2,\rho_3, \dots$ as long as each $a_i$ has a solution which is contained in $\llb \rho_{i-1} \rrb$. If there is no such solution and the pair of sequences cannot be extended, then we are in the case of Lemma~\ref{Lemma:SoldaPaulyLemma}. We assign a rank to a point of discontinuity based on how long such a sequence can be extended. 

\begin{Definition}\label{Def:RankOfDiscontinuity}
	Let $P: \subset \B \tto \B$ be a problem and let $p \in \dom(P)$ be a discontinuity of $P$. We assign a countable ordinal rank of discontinuity to $p$ as follows.
	\begin{itemize}
		\item If $P(p)$ is empty then we say that $p$ is a rank-$0$ discontinuity of $P$.
		\item For countable $\alpha > 0$, $p$ is a rank-$\alpha$ discontinuity of $P$ if $\alpha$ is the least ordinal such that for all $q \in P(p)$ there is $\rho \sqsubset q$ such that for all $\sigma \sqsubset p$ there exists $p' \in \dom(P) \cap \llb \sigma \rrb$ with $p' \neq p$ such that $p'$ is a rank-$\beta$ discontinuity of $P|[\llb \sigma \rrb \times \llb \rho \rrb]$ for some $\beta < \alpha$.
		\item If there is no such ordinal $\alpha$ then we say that $p$ is a rank-$\infty$ discontinuity of $P$.
	\end{itemize}
\end{Definition}
For single valued functions $f \subset \B \to \B$, Definition~\ref{Def:RankOfDiscontinuity} collapses to only three levels (see Section~\ref{sec:SingleValuedFunctions}), but for multivalued functions, there can be discontinuities of arbitrarily high countable rank.

For each countable ordinal $\alpha$, the property of being a rank-$\alpha$ discontinuity can be characterized by a two player game, which we call the rank-$\alpha$ discontinuity game. 

\begin{Definition}[Rank-$\alpha$ Discontinuity Game] \label{def:DiscontinuityGame}
	Let $P: \subset \B \tto \B$ be a problem and let $\alpha$ be a countable ordinal. We recursively define the rank-$\alpha$ discontinuity game for $P$ (or the ``$\alpha$-$P$ game'' for short) as the two player game with the following steps.
	\begin{enumerate}[label=Step~\arabic*., ref=\arabic*]
		\item \label{GameStep:Player1Plays} Player~1 first chooses some $p \in \dom(P)$.
		\item \label{GameStep:Player2Playsq} If $P(p)$ is empty, then the game is concluded and Player~1 wins. Otherwise, Player 2 plays some $q \in P(p)$. 
		\item \label{GameStep:CheckAlphaGreaterThanZero} At this time, if $\alpha = 0$ and Player~1 has not already won then the game is concluded and Player~2 wins. Otherwise, Player 1 next plays some initial segment $\rho \sqsubset q$.
		\item \label{GameStep:Player2PlaysSigma} In response, Player 2 plays some initial segment $\sigma \sqsubset p$. 
		\item \label{GameStep:Player1ChoosesOrdinal} Player~1 then chooses a $\gamma < \alpha$ (which is possible since the prior Step~\ref{GameStep:CheckAlphaGreaterThanZero} ensures that $\alpha > 0$) and play proceeds by starting the rank-$\gamma$ discontinuity game for $P|[\llb\sigma\rrb \times \llb\rho\rrb]$. The rank-$\alpha$ game ends when the rank-$\gamma$ game ends, with the same winner. 
	\end{enumerate}
	We refer to Steps~\ref{GameStep:Player1Plays}--\ref{GameStep:Player1ChoosesOrdinal} as a \emph{round} of the game. When the ordinal parameter of the round (tested at Step~\ref{GameStep:CheckAlphaGreaterThanZero}) is $\beta$, we call it the \emph{$\beta$'th round}. Thus the initial round is the $\alpha$'th round; after Player~1 chooses $\gamma < \alpha$ at Step~\ref{GameStep:Player1ChoosesOrdinal}, the game continues to the $\gamma$'th round.
\end{Definition}

At this point, a reader who wants examples may skip to Section~\ref{sec:SingleValuedFunctions} to see how the rank-$\alpha$ discontinuity game simplifies in the case of single valued functions.

\begin{Remark}
	We pause to note some equivalent alternative formulations of the discontinuity game. Firstly, we could have defined the discontinuity game without reference to an ordinal $\alpha$. In this alternative version of the game, play proceeds as in the ranked discontinuity game except that Player~1 does not choose an ordinal $\gamma$ at Step~5. Instead, play continues forever until Player~1 wins by playing an instance of $P$ with no solutions above the current value of $\rho$; Player~2 wins a run of this game if the run is infinitely long. 

	We could then apply the theory of transfinite game values developed by Evans and Hamkins \cite{evansHamkinsTransfiniteGameValues2014} to infinite chess (which has recently been shown to be universal for Gale-Stewart games by Bolan and Tsevas \cite{bolanTsevasUniversalityInfiniteChess2026}) to assign ranks to winning positions for Player~1. However, since we only want to assign ranks at Step~\ref{GameStep:Player1Plays} of the game (rather than also assigning ranks at Step~\ref{GameStep:CheckAlphaGreaterThanZero}), we opt for the more ad-hoc approach of making the game-rank part of Player~1's moves. 
	
	Secondly, it may be desirable for the set of possible moves by one or both players to be countable. We may restrict the set of possible Player~2 moves to a countable set by changing the rules to have Player~1 preemptively play a set of initial segments $\rho_0, \rho_1, \dots$ that cover $P(p)$ in addition to playing $p$ at Step~\ref{GameStep:Player1Plays}. Then, after confirming that $P(p)$ has solutions, we can skip straight to Step~\ref{GameStep:Player2PlaysSigma}, having Player~2 choose a $\rho_i$ as well as an initial segment $\sigma \sqsubset p$. This version of the game is equivalent to the version given in Definition~\ref{def:DiscontinuityGame}.
\end{Remark}

We will show that Player~1 has a winning strategy in the rank-$\alpha$ discontinuity game for $P$ whose first move is $p$ if and only if $p$ is a rank-$\beta$ discontinuity of $P$ for some $\beta \leq \alpha$. To show this, we will need to use the fact that winning strategies do not need to play any instance of $P$ twice in a row. Essentially, this justifies the absence of the $\tau$ of Theorem~\ref{Thm:ExtendingSolutions} in Definition~\ref{Def:RankOfDiscontinuity}. 

\begin{Proposition}	\label{Prop:WinningStratNotRepeating}
	Fix $P: \subset \N^\N \tto \N^\N$ and a countable ordinal $\alpha$. Suppose that Player~1 has a winning strategy  $\Gamma$ in the rank-$\alpha$ discontinuity game for $P$. Then, Player~1 has a winning strategy $\widehat{\Gamma}$ whose first move is the same as $\Gamma$'s and whose first two instances of $P$ are always different.  
\end{Proposition}
\begin{proof}
	The case $\alpha = 0$ is trivial since Player~1 wins immediately after the first move. For $\alpha > 0$, let $\Gamma$ be a winning strategy for Player~1 in the rank-$\alpha$ discontinuity game for $P$. 
	
	Let $p$ be $\Gamma$'s first move. Fix a Player~2 move $q \in P(p)$. The key idea is that, if Player~2 always plays $q$ at Step~\ref{GameStep:Player2Playsq} then there must be a gamestate in which $\Gamma$'s next Step~\ref{GameStep:Player1Plays} move is something different from $p$ no matter what Player~2 plays at the next Step~\ref{GameStep:Player2PlaysSigma}. 
	
	Consider all partial runs of the game in which Player~1, while playing according to $\Gamma$, plays $p$ at every Step~\ref{GameStep:Player1Plays} while Player~2 plays $q$ at every Step~\ref{GameStep:Player2Playsq}, ending the partial run after Player~2's Step~\ref{GameStep:Player2PlaysSigma} move before $\Gamma$ plays something other than $p$ at Step~\ref{GameStep:Player1Plays}. This set of partial runs has an element $\mathcal{R}$ with minimal most-recent ordinal $\beta$ played by Player~1 at Step~\ref{GameStep:Player1ChoosesOrdinal}. We have $\beta > 0$ since $\beta = 0$ implies then Player~2 has won. Since $\mathcal{R}$ has minimal $\beta$, we know that no matter what Player~2 plays at Step~\ref{GameStep:Player2PlaysSigma}, $\Gamma$'s next Step~\ref{GameStep:Player1Plays} move will be different from $p$ (see Figure~\ref{run:OriginalGammaTurns}). Let $\rho_q \sqsubset q$ be $\Gamma$'s longest Step~\ref{GameStep:CheckAlphaGreaterThanZero} move in $\mathcal{R}$. For each $\sigma \sqsubset p$, let $\gamma_{q,\sigma}$ and $r_{q,\sigma}$ be $\Gamma$'s next moves when we replace Player~2's most recent Step~\ref{GameStep:Player2PlaysSigma} move in $\mathcal{R}$ with $\sigma$. Then, continue playing as $\Gamma$ does. 

	\begin{figure}[!htbp]
            \[
            \begin{array}[t]{@{}c@{\qquad\qquad}c@{\qquad\qquad}c@{}}
                \mathcal{R} &
                \mathcal{R} \text{ with } \eta \text{ replaced by } \sigma &
                \phantom{\text{Player 1 }(\widehat\Gamma)}\\
                \begin{array}[t]{c|c}
                    \text{Player 1 }(\Gamma) & \text{Player 2}\\
					p & q\\
                    \vdots & \vdots\\
                    \beta, p & q\\
                    \rho =: \rho_q & \eta
                \end{array}
                &
                \begin{array}[t]{c|c}
                    \text{Player 1 }(\Gamma) & \text{Player 2}\\
                    p & q\\
                    \vdots & \vdots\\
                    \beta, p & q\\
                    \rho_q & \sigma \\
                    \gamma =: \gamma_{q,\sigma},r =: r_{q,\sigma} &
                \end{array}
                &
                \begin{array}[t]{c|c}
                    \text{Player 1 }(\widehat\Gamma) & \text{Player 2}\\
                    p & q\\
                    \rho_q & \sigma\\
                    \gamma_{q,\sigma},r_{q,\sigma} &
                \end{array}
            \end{array}
            \]
        \caption{Tables tracking the three game runs in the proof of Proposition~\ref{Prop:WinningStratNotRepeating}. On the left is $\mathcal{R}$. In the center is $\mathcal{R}$ with its last Step~\ref{GameStep:Player2PlaysSigma} move replaced by $\sigma$. On the right is the $\alpha$'th round played according to $\widehat{\Gamma}$.}\label{run:OriginalGammaTurns}
    \end{figure}
\end{proof}

We are now ready to show that the rank-$\alpha$ discontinuity game characterizes rank-$\alpha$ discontinuity.

\begin{Theorem}\label{Thm:DiscontinuityGameCharacterization}
	Let $P: \subset \B \tto \B$ be a problem, let $p \in \dom(P)$ and let $\alpha$ be a countable ordinal. Player~1 has a winning strategy with first move $p$ in the rank-$\alpha$ discontinuity game for $P$ if and only if $p$ is a rank-$\beta$ discontinuity of $P$ for some $\beta \leq \alpha$.
\end{Theorem}
\begin{proof}
	We argue by induction on $\alpha$. For $\alpha = 0$, Player~1 has a winning strategy in the rank-$0$ discontinuity game for $P$ if and only if $P$ has an instance with no solution, which is equivalent to $P$ having a rank-$0$ discontinuity. 
	
	($\implies$): Suppose that the theorem is true for all $\gamma < \alpha$. Let $\Gamma$ be a winning strategy for Player~1 in the $\alpha$-$P$ game with first move $\Gamma(\lambda) = p$. If $P(p) = \emptyset$, then $p$ is a rank-$0$ discontinuity of $P$, so we are done. 
	
	Otherwise, fix $q \in P(p)$ and $\sigma \sqsubset p$. Let $\rho \sqsubset q$, $\gamma < \alpha$, and $p' \sqsupset \sigma$ be the moves chosen by $\Gamma$ in the following partial run of the $\alpha$-$P$ game in which Player~1 follows $\Gamma$.
	\[
		\begin{array}{c|c}
			\text{Player~1 } (\Gamma)& \text{Player~2} \\
			p & q\\
			\rho & \sigma\\
			\gamma, p' &
		\end{array}
	\] 

	By Proposition~\ref{Prop:WinningStratNotRepeating}, we may assume that $p' \neq p$.  Since $\Gamma$ is a winning strategy, its continuation on $P|[\llb \sigma \rrb \times \llb \rho \rrb]$ is a winning strategy for Player~1 in the $\gamma$-$P|[\llb \sigma \rrb \times \llb \rho \rrb]$ game. By the induction hypothesis, $p'$ is a rank-$\gamma'$ discontinuity of $P|[\llb \sigma \rrb \times \llb \rho \rrb]$ for some $\gamma' \leq \gamma < \alpha$.  Since $q$ and $\sigma$ were arbitrary, $\alpha$ satisfies the condition in the definition of $p$ being at most a rank-$\alpha$ discontinuity of $P$ in Definition~\ref{Def:RankOfDiscontinuity}, so $p$ is a rank-$\beta$ discontinuity of $P$ for some $\beta \leq \alpha$, as desired.  

	($\impliedby$): Suppose that the theorem is true for all $\gamma < \alpha$. Let $p$ be a rank-$\beta$ discontinuity of $P$ for some $\beta \leq \alpha$. We show that Player~1 has a winning strategy in the rank-$\beta$ discontinuity game for $P$. Since $p$ is a rank-$\beta$ discontinuity of $P$, we have that for all $q \in P(p)$ there is $\rho_q \sqsubset q$ such that for all $\sigma \sqsubset p$ there is $\gamma_{q,\sigma} < \beta$ and $p_{q,\sigma} \in \dom(P) \cap \llb \sigma \rrb$ with $p_{q,\sigma} \neq p$ such that $p_{q,\sigma}$ is a rank-$\gamma_{q,\sigma}$ discontinuity of $P|[\llb \sigma \rrb \times \llb \rho_q \rrb]$.
	
	We define a winning strategy $\Gamma$ for Player~1 in the rank-$\beta$ discontinuity game for $P$ with the following first turn for any $q \in P(p)$ and $\sigma \sqsubset p$. 
	\[
		\begin{array}{c|c}
		\text{Player~1 } (\Gamma)& \text{Player~2} \\
		p & q\\
		\rho_q & \sigma\\
		\gamma_{q,\sigma}, p_{q,\sigma} &
		\end{array}
	\]
	Since $p_{q,\sigma}$ is a rank-$\gamma_{q,\sigma}$ discontinuity of $P|[\llb \sigma \rrb \times \llb \rho_q \rrb]$, by the induction hypothesis there is a winning strategy for Player~1 in the rank-$\gamma_{q,\sigma}$ discontinuity game for $P|[\llb \sigma \rrb \times \llb \rho_q \rrb]$. We extend $\Gamma$ by following such a winning strategy. Hence, $\Gamma$ is a winning strategy for Player~1 in the rank-$\beta$ discontinuity game for $P$. It is therefore also a winning strategy for Player~1 in the rank-$\alpha$ discontinuity game for $P$, and its first move is $p$, as desired.
\end{proof}

Having a discontinuity of rank $\leq\alpha$ is invariant under Weihrauch reducibility. This fact extends to points; if $p$ is a rank-$\alpha$ discontinuity of $P$ and $P \leq_\W^* Q$ via forward function $\Phi$ and backward function $\Psi$, then $\Phi(p)$ is a rank-$\beta$ discontinuity of $Q$ for some $\beta \leq \alpha$. 

\begin{Theorem}\label{Thm:DiscontinuityGameInvariant}
	Let $P: \subset \B \tto \B$ and $Q: \subset \B \tto \B$ be problems and suppose that $P \leq_\W^* Q$ via forward function $\Phi$ and backward function $\Psi$. Then,
	\begin{enumerate}
		\item  Player~1 having a winning strategy in $\alpha$-$P$-game with first move $p$ implies that Player~1 has a winning strategy in the $\alpha$-$Q$ game with first move $\Phi(p)$.
		\item  Player~2 having a winning strategy in the $\alpha$-$Q$ game implies that Player~2 has a winning strategy in the $\alpha$-$P$ game. 
	\end{enumerate} 
\end{Theorem}
\begin{proof}
	(1) We argue by induction. For $\alpha = 0$, Player~1 wins the rank-$0$ discontinuity game for $P$ if and only if $P$ has an instance with no solution. If $P \leq_\W^* Q$ then $Q$ must also have an instance with no solution, so Player~1 has a winning strategy in the rank-$0$ discontinuity game for $Q$. 

	Suppose that the claim is true for all $\beta < \alpha$. Let $\Gamma$ be a winning strategy for Player~1 in the rank-$\alpha$ discontinuity game for $P$. Suppose that $P \leq_\W^* Q$ via forward function $\Phi$ and backwards function $\Psi$. Let $p_0 = \Gamma(\lambda)$ be $\Gamma$'s first move. Let $\widehat{\Gamma}$ be a strategy for the first three Player~1 moves of the rank-$\alpha$ game for $Q$ given by
	\begin{itemize}
		\item $\widehat{\Gamma}(\lambda) = \hat{p}_0 := \Phi(p_0)$.
		\item $\widehat{\Gamma}(\hat{q}) = \hat{\rho}_q$ such that there exists $n$ such that $\Psi(\llb p_0|_n \rrb ,\llb \hat{\rho}_q \rrb) \subset \llb \rho \rrb$ for $\rho := \Gamma(\Psi(p_0,\hat{q}))$. Such a $\rho$ and $n$ exist by the fact that $\Gamma(\Psi(p_0, \hat{q})) \sqsubset \Psi(p_0, \hat{q})$ and by continuity of $\Psi$.  
		\item $\widehat{\Gamma}(\hat{q},\hat{\sigma}) = \gamma := \Gamma(\Psi(p_0, \hat{q}), \sigma)$ for some $\sigma$ such that $\Phi(\llb \sigma \rrb) \subset \llb \hat{\sigma }\rrb$ and $p_0|_n \sqsubseteq \sigma $. Such a $\sigma$ exists by the fact that $\hat{\sigma} \sqsubset \hat{p_0}$ and by continuity of $\Phi$. 
	\end{itemize}

	\begin{figure}[h!t]
		\begin{equation*}
			\begin{array}{c|c}
				\text{Player~1 } (\Gamma)& \text{Player~2} \\
				p_0 & q\\
				\rho & \sigma \\
				\gamma &
		\end{array} \quad
			\begin{array}{c|c}
				\text{Player~1 } (\widehat{\Gamma})& \text{Player~2} \\
				\hat{p}_0 & \hat{q}\\
				\hat{\rho}_q & \hat{\sigma} \\
				\gamma & 
			\end{array} 
		\end{equation*}
		\caption{The first three moves of $\Gamma$ and $\widehat{\Gamma}$ in their respective games.}
		\label{fig:FirstThreeMovesOfDelta}
	\end{figure}

	Suppose that Player~1 plays the first three moves of the $\alpha$-$Q$ game according to $\widehat{\Gamma}$, as in Figure~\ref{fig:FirstThreeMovesOfDelta}. 
	Then, play proceeds by beginning the rank-$\gamma$ game for $Q|[\llb\hat{\sigma}\rrb \times \llb \hat{\rho}_q \rrb]$. 
	Then there is a corresponding partial run of the $\alpha$-$P$ game depicted in Figure~\ref{fig:FirstThreeMovesOfDelta} such that the following statements hold.
	\begin{enumerate}
		\item $\hat{p}_0 = \Phi(p_0)$.
		\item $q = \Psi(p_0, \hat{q})$.
		\item \label{DiscontinuityGameInvariantItem:PsiCorrect}$\Psi(\llb p_0|_n \rrb ,\llb \hat{\rho}_q \rrb) \subset \llb \rho \rrb$ for some $n \in \N$.
		\item \label{DiscontinuityGameInvariantItem:SigmaLongEnough}$p_0|_n \sqsubseteq \sigma$
		\item \label{DiscontinuityGameInvariantItem:PhiCorrect}$\Phi(\llb \sigma \rrb) \subset \llb \hat{\sigma} \rrb$.
	\end{enumerate}
	We claim that
	\[
		P|[\llb \sigma \rrb \times \llb\rho\rrb] \leq_\W^* Q|[\llb\hat{\sigma}\rrb \times \llb \hat{\rho}_q \rrb].
	\]
	Continuing to follow $\Gamma$ provides a winning Player~1 strategy for the rank-$\gamma$ game for $P|[\llb \sigma \rrb \times \llb \rho \rrb]$. Hence, if the claim is true then we can extend $\widehat{\Gamma}$ by the winning strategy promised by the induction hypothesis, which is enough to finish the current proof.  

	Statement~\ref{DiscontinuityGameInvariantItem:PhiCorrect} implies that $\Phi$ sends instances of $P|_{\llb\sigma\rrb}$ to instances of $Q|_{\llb \hat{\sigma}\rrb}$.  Furthermore, Statements~\ref{DiscontinuityGameInvariantItem:PsiCorrect} and \ref{DiscontinuityGameInvariantItem:SigmaLongEnough} imply that $\Psi$ is a partial map from $\llb\sigma\rrb \times \llb \hat{\rho}_q \rrb$ to $\llb \rho \rrb$. Since $\Phi$ and $\Psi$ give a reduction from $P$ to $Q$ this implies that $\Phi$ and $\Psi$ also give a reduction of $P|[\llb \sigma \rrb \times \llb\rho\rrb]$ to $Q|[\llb\hat{\sigma}\rrb \times \llb \hat{\rho}_q \rrb]$, as required. 

	(2) If Player~2 has a winning strategy in the $\alpha$-$Q$ game then Player~1 does not have a winning strategy in the $\alpha$-$Q$ game. By (1), we can conclude that Player~1 also does not have a winning strategy in the $\alpha$-$P$ game. Since the $\alpha$-$P$ game is determined, we may conclude that Player~2 has a winning strategy.
\end{proof}

Unlike the Wadge game developed by Nobrega and Pauly \cite{nobregaGameCharacterizationsLower2019} to characterize Weihrauch reducibility, Player~2 winning the $\alpha$-$P$ discontinuity game does not imply that $P$ is continuous (even if this is true for all $\alpha$, or even in a version of the game with no ordinal constraining Player~1 at all). This can be seen by the fact that every instance of $\DIS$ has solutions above every string $\rho$ but is still discontinuous. However, Player~2 winning the $\alpha$-$P$ does imply that $P$ is continuous on each set of pointwise rank $\alpha$. We prove this as a corollary of Theorem~\ref{Thm:MainTheoremPointWise} and do not use it in the proof. 

\begin{Corollary}\label{Cor:Player2WinConsequence}
	Let $P: \subset \B \tto \B$ be a problem and suppose that Player~2 wins the $\alpha$-$P$ game.  Then, for all sets $A$ with $\rankp(A) \leq \alpha$, we have that $P|_A$ is continuous.
\end{Corollary}
\begin{proof}
	Suppose that there is an $A$ with $\rankp(A) \leq \alpha$ such that $P|_{A}$ is discontinuous. Then, $P|_A$ has a point of discontinuity $p$. Since $\rankp(A) \leq \alpha$, we have that $p$ is a rank-$\beta$ discontinuity for some $\beta \leq \alpha$. By Theorem~\ref{Thm:MainTheoremPointWise}, Player~1 has a winning strategy in the $\alpha$-$P$ game, a contradiction. 
\end{proof}

Next we show that Condition~\ref{MainTheoremPointwiseItem:WinningDiscontinuityGameInAlphaTurns} of Theorem~\ref{Thm:MainTheoremPointWise} implies Condition~\ref{MainTheoremPointwiseItem:AboveACCN}.

\begin{Theorem} \label{Thm:WinningStrategyImpliesAboveACCN}
	Let $P: \subset \B \tto \B$ be a problem and let $\alpha$ be a countable ordinal. Fix $p \in \dom(P)$. Suppose that Player~1 has a winning strategy for the rank-$\alpha$ discontinuity game for $P$ whose first move is $p$. Then, $P \geq_\W^* \ACCN^\alpha$ via forward function $\Phi$ such that $\Phi(\hash^\N) = p$. 
\end{Theorem}
\begin{proof}
	We argue by induction. For $\alpha = 0$, Player~1 has a winning strategy in the $0$-$P$ game whose first move is $p$ $\implies$ $P(p) = \emptyset$ $\implies$ $P \geq_\W^* \ACCN^0$ via some function $\Phi$ such that $\Phi(\hash^\N) = p$.

	Suppose that the theorem is true for all $\beta < \alpha$. Let $\Gamma$ be a winning strategy for Player~1 in the $\alpha$-$P$ game. Let $p := \Gamma(\lambda)$ be the first move made by $\Gamma$. Let $\rho_1, \rho_2, \dots$ be an enumeration of all moves $\Gamma(q)$ made by $\Gamma$ when Player~2 plays some $q \in P(p)$. In other words, let $\{\rho_i\}_{i \in \N}$ be an enumeration of the set $\{\Gamma(q): q \in P(p)\}$, which is countable since it is a subset of $\N^{<\N}$. Furthermore, for each $i \in \N$, fix a $q_i$ such that $\Gamma(q_i) = \rho_i$. For each $i \in \N$ and $\sigma \sqsubset p$, define $r(i, \sigma)$ and $\beta(i,\sigma)$ so $\Gamma(q_i, \sigma) = (\beta(i,\sigma), r(i,\sigma))$. By Proposition~\ref{Prop:WinningStratNotRepeating}, we may assume that $r(i,\sigma) \neq p$ for all $i$ and $\sigma.$ We summarize this discussion in Figure~\ref{fig:FirstTwoRoundsOfGamma}.

	\begin{figure}[h]
		\begin{equation*}
			\begin{array}{c|c}
				\text{Player~1 } (\Gamma)& \text{Player~2} \\
				p & q_i\\
				\rho_i & \sigma \\
				\beta(i,\sigma), r(i,\sigma) &
			\end{array}
		\end{equation*}
		\caption{Our definition of $\rho_i$, $r(i,\sigma)$ and $\beta(i,\sigma)$ in terms of $\Gamma$.}
		\label{fig:FirstTwoRoundsOfGamma}
	\end{figure}

	Since $\Gamma$ is a winning strategy for Player~1 in the $\alpha$-$P$ game, we have that the continuation of $\Gamma$ after $\Gamma$ has played $\Gamma(q_i, \sigma)$ is a winning strategy for Player~1 in the rank-$(\beta(i,\sigma))$ discontinuity game for $P|[\llb \sigma \rrb \times \llb \rho_i \rrb]$. By the induction hypothesis, for all $\gamma \geq \beta(i,\sigma)$, we have that
	\[
		P|[\llb \sigma \rrb \times \llb \rho_i \rrb] \geq_\W^* \ACCN^{\gamma}.
	\]
	Let $\Phi_{i,\sigma,\gamma}$ and $\Psi_{i,\sigma,\gamma}$ be functionals witnessing this reduction. Without loss of generality, we may assume that for all $i$ and $\gamma$ we have that $\sigma \sqsubset \Phi_{i,\sigma,\gamma}(x)$ for all $x \in \dom(\Phi_{i,\sigma,\gamma})$. This can be achieved by removing any points not lying above $\sigma$ from the domain of $\Phi_{i,\sigma,\gamma}$ which does not affect the fact that $\Phi$ witnesses the reduction.

	We construct increasing sequences $\{n_i \in \N\}_{i \in \N}$ and $\{\sigma_i \sqsubset p \}_{i \in \N}$ of parameters used to define $\Phi$ and $\Psi$ witnessing that $P \geq_\W^* \ACCN^\alpha$. Let $\sigma_0 = \lambda$, the empty string. Once $\sigma_s$ has been defined, for each $i \leq s$ we define $\tau_{i,s + 1}$ to be the longest shared initial segment of $r(i, \sigma_s)$ and $p$. The string $\tau_{i,s+1}$ exists by our assumption that $r(i,\sigma_s) \neq p$. Note that we also have $\sigma_s \sqsubseteq \tau_{i,s+1}$ for all $i \leq s$. 

	Then, we may select any $\sigma_{s+1} \sqsubset p$ such that $\sigma_{s} \sqsubset \sigma_{s+1}$ and for all $i \leq s$,
	\[
		\tau_{i,s + 1} \sqsubset \sigma_{s+1}. 
	\]
	We illustrate the positions of $\sigma_i$, $\tau_{i,j}$ and $r(i,\sigma_i)$ along $p$ in Figure~\ref{fig:tau-sigma-tree}.
	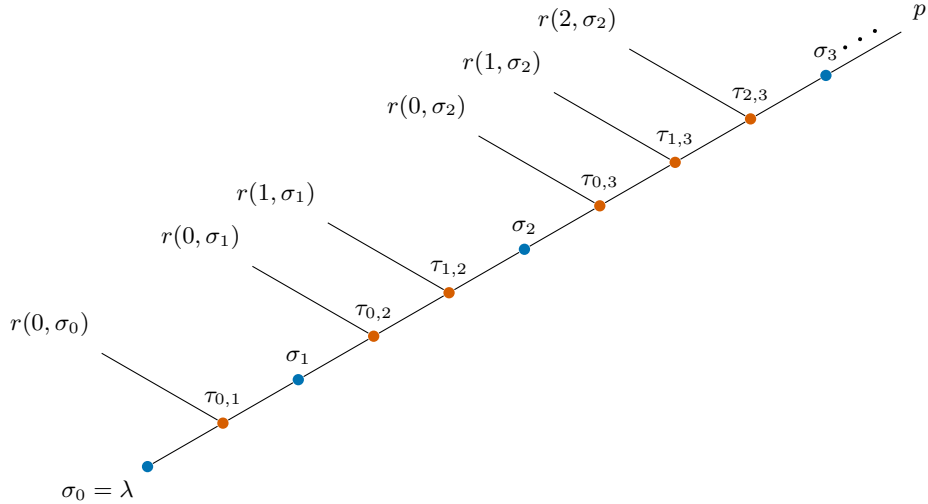
\begin{figure}[ht]
	\centering
	\begin{tikzpicture}[
		every node/.style={font=\small},
		sigmadot/.style={circle,fill={rgb,1:red,0.00;green,0.45;blue,0.70},inner sep=1.5pt},
		taudot/.style={circle,fill={rgb,1:red,0.84;green,0.37;blue,0.00},inner sep=1.5pt},
		main/.style={grow'=30,level distance=1.15cm,sibling distance=8mm},
		leftbranch/.style={grow'=150,level distance=1.85cm,sibling distance=0mm}
	]
		\node[sigmadot, label={below left:{$\sigma_0 = \lambda$}}] (s0) {}
			child[main] { node[taudot, label={above:{$\tau_{0,1}$}}] (t01) {}
				child[leftbranch] { node[coordinate] (b01) {} edge from parent node[pos=1, above left, yshift=3pt, xshift=-1pt] {$r(0,\sigma_0)$} }
				child[main] { node[sigmadot, label={above:{$\sigma_1$}}] (s1) {}
					child[main] { node[taudot, label={above:{$\tau_{0,2}$}}] (t02) {}
						child[leftbranch] { node[coordinate] (b02) {} edge from parent node[pos=1, above left, yshift=3pt, xshift=-1pt] {$r(0,\sigma_1)$} }
						child[main] { node[taudot, label={above:{$\tau_{1,2}$}}] (t12) {}
							child[leftbranch] { node[coordinate] (b12) {} edge from parent node[pos=1, above left, yshift=3pt, xshift=-1pt] {$r(1,\sigma_1)$} }
							child[main] { node[sigmadot, label={above:{$\sigma_2$}}] (s2) {}
								child[main] { node[taudot, label={above:{$\tau_{0,3}$}}] (t03) {}
									child[leftbranch] { node[coordinate] (b03) {} edge from parent node[pos=1, above left, yshift=3pt, xshift=-1pt] {$r(0,\sigma_2)$} }
									child[main] { node[taudot, label={above:{$\tau_{1,3}$}}] (t13) {}
										child[leftbranch] { node[coordinate] (b13) {} edge from parent node[pos=1, above left, yshift=3pt, xshift=-1pt] {$r(1,\sigma_2)$} }
										child[main] { node[taudot, label={above:{$\tau_{2,3}$}}] (t23) {}
											child[leftbranch] { node[coordinate] (b23) {} edge from parent node[pos=1, above left, yshift=3pt, xshift=-1pt] {$r(2,\sigma_2)$} }
											child[main] { node[sigmadot, label={above:{$\sigma_3$}}] (s3) {}
												child[main] { node[coordinate] (endp) {} }
											}
										}
									}
								}
							}
						}
					}
				}
			};

		\node[font=\Large, above left=-2pt and -0.5pt of endp, rotate=29] {$\cdots$};
		\node[above right=1pt and 1pt of endp] {$p$};
	\end{tikzpicture}
	\caption{The positions of $\sigma_i$, $\tau_{i,j}$, and $r(i,\sigma_i)$ along $p$.}
	\label{fig:tau-sigma-tree}
\end{figure}

	Let $n_0 = 0$. To define $n_{s+1}$ from $n_s$, let 
	\[
	\gamma_{s} = \max\{\beta(i,\sigma_{s}): 0 \leq i \leq s \}.
	\]
	Since each $\beta(i,\sigma_s) < \alpha$, we have that $\gamma_s < \alpha$. Hence, we may pick the least $n_{s+1} > n_s$ such that
	\[
		\ell_\alpha(n_{s+1}) \geq \gamma_s.
	\]
	For each $k \geq n_1$, let $s(k)$ be the least stage number such that $n_{(s(k) + 1)} \leq k < n_{(s(k)+2)}$, so that we have a guarantee that $\ell_\alpha(k) \geq \gamma_{s(k)}$. Then, for all $k \geq n_1$ and $i \leq s(k)$ we have that $\ell_\alpha(k) \geq \beta(i,\sigma_{s(k)})$, so
	\begin{equation}
		\label{eq:InductiveStepAppliedToCorrectOnes}
		P|[\llb\sigma_{s(k)}\rrb \times \llb \rho_i \rrb] \geq_\W^* \ACCN^{\ell_\alpha(k)},
	\end{equation}
	witnessed by $\Phi_{i, \sigma_{s(k)}, \ell_\alpha(k)}$ and $\Psi_{i,\sigma_{s(k)},\ell_\alpha(k)}$.
	We now define $\Phi$ and $\Psi$.  
	
	We compute $\Phi$ as follows. The computation for $\Phi(x)$ operates in stages $t \geq 0$.
	At stage $t$, the computation checks $x|_{n_{(t + 2)}}$ for a non-$\hash$ entry. 
	If it does not find one, then the computation commits to outputting $\sigma_{s(n_{(t + 2)})} = \sigma_{t + 1}$. 
	If the computation does find a non-$\hash$ entry $i$, then $x = \hash^k i w$ for some $k < n_{(t + 2)}$ and $w \in (\N \cup \{\hash\})^\N$. Then, $\Phi$ checks whether $k \geq n_1$ and $i \leq s(k)$. If either of these is false, then $\Phi$ outputs $p$ (knowing that $\Psi$ will find the $i$ and solve it without the help of a $P$-solution.) If both are true, then $\Phi$ extends the current output $\sigma_{t}$ to  $\Phi_{i, \sigma_{t}, \ell_\alpha(k)}(w)$.

	Since $n_{t + 1} \leq k < n_{(t+2)}$, we have that $s(k) = t$, so $\sigma_t = \sigma_{s(k)}$. Hence, we obtain the formula
	\[
		\Phi(x) = \begin{cases}
			p	&\text{if }  x = \hash^\N\\
			p   &\text{if } x = \hash^k i w \text{, } k < n_1 \text{ or } i > s(k)\\
			\Phi_{i, \sigma_{s(k)}, \ell_\alpha(k)}(w)	&\text{if } x = \hash^k i w \text{, } k \geq n_1 \text{ and } i \leq s(k).
		\end{cases}
	\]
	Since $\sigma_t \sqsubset \Phi_{i, \sigma_{t}, \ell_\alpha(k)}(w)$ for all $w \in \dom(\Phi_{i, \sigma_{t}, \ell_\alpha(k)})$, we have that $\Phi$ is continuous. 
	
	Next, we define the backward function $\Psi$. We compute $\Psi(x,q)$ by first simultaneously searching for a non-$\hash$ entry of $x$ or a $\rho_i \sqsubset q$. This search must terminate because if $x = \hash^\N$ then $\Gamma(q) \in \{\rho_i: i \in \N\}$, so there exists an $i$ such that $\rho_i = \Gamma(q) \sqsubset q$. If $\Psi$ first finds $\rho_i \sqsubset q$ then it continues searching $x|_{n_{(i+1)}}$ for a non-$\hash$ entry.  
	
	In either search mode, if $\Psi$ finds that $x = \hash^k i w$ for some $i \in \N$ and $w \in \extendedBaireSpace$ then setting $\Psi(x,q) = (i+1)\hash^\N$ successfully solves $\ACCN^\alpha(x)$.

	Otherwise, we have that $\rho_i \sqsubset q$ and $\hash^{n_{(i+1)}} \sqsubset x$. Then $\Psi$ outputs $i$ and then starts outputting $\hash$. It continues doing so forever unless it finds a $k \geq n_{(i+1)}$ such that $\hash^k i \sqsubset x$, and hence that $x = \hash^k i w$ for some $w \in \extendedBaireSpace$. Since $k \geq n_{(i+1)}$ we have that $s(k) \geq i$ and $k \geq n_1$. Hence, the conditions for the reduction in Line~\ref{eq:InductiveStepAppliedToCorrectOnes} are satisfied, so we know that $\Phi_{i, \sigma_{s(k)}, \ell_\alpha(k)}$ and $\Psi_{i, \sigma_{s(k)}, \ell_\alpha(k)}$ witness that
	\[
		P|[\llb\sigma_{s(k)} \rrb \times \llb \rho_i \rrb] \geq_\W^* \ACCN^{\ell_\alpha(k)}.
	\] 
	By our definition of $\Phi$, we have that $\Phi(\hash^k i w) = \Phi_{i, \sigma_{s(k)}, \ell_\alpha(k)}(w)$. Hence, having $\Psi$ begin outputting $\Psi_{i, \sigma_{s(k)}, \ell_\alpha(k)}(w,q)$ successfully solves $\ACCN^\alpha(x)$. Thus, $\Phi$ and $\Psi$ witness that $P \geq_\W^* \ACCN^\alpha$. 

	\begin{Remark} By using Lemma~\ref{Lemma:CoverByNodes} to assume that the $\rho_i$ are incomparable, we obtain an explicit formula for $\Psi$,
	\[
		\Psi(\hash^\N, q) = \begin{cases}
			i\hash^\N &\text{if } \rho_i\sqsubset q
		\end{cases}
	\]
	and 
	\[
		\Psi(\hash^k i w, q) = \begin{cases}
			(i + 1) \hash^\N &\text{if } (\A j < s(k))(\rho_j \not\sqsubset q)\\
			j \hash^\N &\text{if } (\E j < s(k))((\rho_j \sqsubset q) \wedge (j \neq i))\\
			i \hash^k \Psi_{i, \sigma_{s(k)}, \ell_\alpha(k)} (w,q) &\text{if }( \rho_i \sqsubset q )\wedge (i \leq s(k))
		\end{cases}
	\]
	The first case of the $\Psi(\hash^k i w, q)$ formula occurs when $\Psi$ finds the $i$ in $\hash^k i w$ before it finds any $\rho_j \sqsubset q$. The second case occurs when $\Psi$ finds some $\rho_j \sqsubset q$ before it finds the $i$ in $\hash^k i w$ and $j \neq i$. The third case occurs when $\Psi$ finds the $\rho_j \sqsubset q$ before it finds the $i$ in $\hash^k i w$ and $j = i$.
	\end{Remark}
\end{proof}

In a presentation at CCR 2022 \cite{SoldaMultifunctionsTalk}, Sold\`a announced joint work with Pauly which extends the main theorem of \cite{SoldaPaulySequentialDiscontinuity} to finite ordinals. They define represented spaces $\mathbf{Z}_n$ and prove the following theorem.
\begin{Theorem}[\cite{SoldaMultifunctionsTalk}]
	The following are equivalent for a problem $P: \subset \B \tto \B$ and $n \in \N$.
	\begin{enumerate}
		\item $\ACCN^{+n} \leq_\W^* P$.
		\item There is discontinuous $Q: Z_n \tto \B$ such that $Q \leq_\W^* P$.
	\end{enumerate}
\end{Theorem}
Here, we write $\ACCN^{+n} = \ACCN + \ACCN + \dots + \ACCN$ for the problem which takes $n$ instances of $\ACCN$ and outputs a solution to at least one of them. See \cite{brattkaGherardiPaulyWeihrauchComplexityComputable2021} for a more precise definition. 

For $n \in \N$ with $n \geq 1$, the set $\dom(\delta_{\mathbf{Z}_n})$ has pointwise Cantor-Bendixson rank $\rankp(\dom(\delta_{\mathbf{Z}_n})) = n$, so together with Corollary~\ref{Cor:CompactnessEquivalence} and Theorem~\ref{Thm:MainTheorem}, the above theorem implies that $\ACCN^{+n} \equiv_\W^* \ACCN^n$ for all such $n$. In fact, they are equivalent under computable Weihrauch reducibility. We sketch a proof of this fact. The reduction from $\ACCN^n$ to $\ACCN^{+n}$ is nontrivial and requires much of the power of our proof of Theorem~\ref{Thm:WinningStrategyImpliesAboveACCN}, although it can be slightly simplified because no limit ordinals are involved. However, we apply the argument as-is.
\begin{Corollary}
	For each natural number $n$, $\ACCN^n \equiv_\W \ACCN^{+n}$. 
\end{Corollary}
\begin{proof}
	First note that the statement of this corollary makes sense since, for each finite ordinal $n$, the set of rank-witnessing representations of $n$ is uniformly computable. 
	
	We have that $\ACCN^{n} \geq_\W \ACCN^{+n}$ by sending $\ACCN^{+n}$'s $i$'th instance of $\ACCN$ to $\ACCN^{n}$'s $i$'th instance of $\ACCN$. Then, $\ACCN^{+n}$ can guess each of $\ACCN^{n}$'s answers in order and at least one of them is correct. The same approach does not work in the other direction because $\ACCN^{+n}$ may never provide a guess for its first instance, while $\ACCN^n$ must make its guesses in the order the instances appear.

	Instead, we argue via Theorem~\ref{Thm:MainTheorem} and the proof of Theorem~\ref{Thm:WinningStrategyImpliesAboveACCN}. Since $\ACCN^{+n}$ has \[\tup{\hash^\N, \hash^\N, \dots, \hash^\N}\] as its sole point of discontinuity and \[\rank_{\dom(\ACCN^{+n})}(\tup{\hash^\N, \hash^\N, \dots, \hash^\N}) = n,\] we immediately obtain $\ACCN^n \leq_\W^* \ACCN^{+n}$ by Theorem~\ref{Thm:MainTheorem}. However, we want a computable reduction.  
	
	The $\Phi$ and $\Psi$ given in the proof of Theorem~\ref{Thm:WinningStrategyImpliesAboveACCN} are computable if the $\rho_i$ are computably enumerable and the functions $\beta(i,\sigma)$, $r(i,\sigma)$, $\Phi_{i,\sigma,\gamma}$, and $\Psi_{i,\sigma,\gamma}$ are computable. In the case of showing that $\ACCN^{+n} \geq_\W \ACCN^{n}$, these functions are computable from the following computable winning Player~1 strategy for the $n$-$\ACCN^{+n}$ game. 
	\begin{itemize}
		\item First play $\tup{\hash^\N,\hash^\N, \dots, \hash^\N}$.
		\item Player~2's move $q$ must include some guess $i \in \N$ at some coordinate. Search $q$ until such a guess is found and then play some $\rho$ which locks in this guess. 
		\item In response to any $\sigma \sqsubset \tup{\hash^\N,\hash^\N, \dots, \hash^\N}$, play $n-1$ and some \[p' = \tup{\hash^\N,\hash^\N, \dots, \hash^n i \hash^\N, \hash^\N, \dots, \hash^\N},\] with the $\hash^n i \hash^\N$ at the coordinate which $q$ guessed $i$ for and $n$ so that $\sigma \sqsubset p'$. 
		\item Continue with a similar strategy for $\ACCN^{+{(n-1)}}$. 
	\end{itemize}
\end{proof}

Next, we show that undiagonalizable problems have the weakest form of discontinuity in our hierarchy. 

\begin{Definition}
	Let $P: \subset \B \tto \B$ be a problem. We say that $P$ is  \emph{undiagonalizable} if there is a computable procedure which, for all $\rho \in \N^{<\N}$ and $p \in \dom(P)$, decides whether $P(p) \cap \llb \rho \rrb = \emptyset$.
\end{Definition}

\begin{Theorem}\label{Thm:undiagonalizableImpliesPlayer2Win}
	Let $P: \subset \B \tto \B$ be undiagonalizable and suppose that every instance of $P$ has a solution. Then, Player~2 has a winning strategy in the infinite discontinuity game for $P$. 
\end{Theorem}
\begin{proof}
	Let $\Phi: \dom(P) \times \N^{<\N} \to \{0,1\}$ be a computable functional such that $\Phi(p,\sigma) = 1$ if and only if $\llb \sigma \rrb \cap P(p) \neq \emptyset$. We define a winning strategy $\Delta$ for Player~2 in the infinite discontinuity game for $P$ as follows. In round one, when Player~1 plays $p \in \dom(P)$, Player~2 responds with some arbitrary $q \in P(p)$. When Player~1 plays $\rho \sqsubset q$, we have that $p$ has $P$-solutions above $\rho$, so $\Phi(p,\rho) = 1$. Then, we have $\Delta$ respond with $\sigma \sqsubset p$ longer than the use of $\Phi(p,\rho) = 1$. 
	\begin{equation*}
		\begin{array}{c|c}
			\text{Player~1 }& \text{Player~2} (\Delta) \\
			p & q\\
			\rho & \sigma \text{ such that } \Phi(\llb \sigma \rrb, \rho) = \{1\} \\
		\end{array}
	\end{equation*}
	Since $\Phi(\llb \sigma \rrb, \rho) = \{1\}$, we have that $P(p') \cap \llb \rho' \rrb \neq \emptyset$ for all $p' \in \llb \sigma \rrb$ and $\rho' \sqsupset \rho$. Hence, from now on, $\Delta$ always has access to a legal move and wins the game. \footnote{The proof of Theorem~\ref{Thm:undiagonalizableImpliesPlayer2Win} was produced by generative AI. See the acknowledgements section.}

\end{proof}

Finally, we finish the proof of Theorems~\ref{Thm:MainTheorem} and \ref{Thm:MainTheoremPointWise} by proving that (\ref{MainTheoremPointwiseItem:DiscontinuousOnPointWiseRankedSet}) implies (\ref{MainTheoremPointwiseItem:WinningDiscontinuityGameInAlphaTurns}) in the statement of Theorem~\ref{Thm:MainTheoremPointWise}. We will use the following characterizations of discontinuity on some set of a certain rank.

\begin{Proposition}	\label{Prop:NbhdRankingOfDiscontinuities}
	Let $P: \subset \B \tto \B$ be a problem, let $p \in \dom(P)$ be an instance of $P$ and let $\alpha$ be a countable ordinal. The following are equivalent. 
	\begin{enumerate}
		\item \label{RankDiscontinuityDef:Local}There exists a set $A \subset \dom(P)$ with $\rank_A(p) \leq \alpha$ and such that $p$ is a discontinuity of $P|_A$.  
		\item \label{RankDiscontinuityDef:Neighborhood}There exists a set $A \subset \dom(P)$ and an open neighborhood $U$ of $p$ such that $p$ is a discontinuity of $P|_{A \cap U}$ and $U$ is a $\leq\!\!\alpha$-isolating neighborhood of $p$ in $A$.
		\item \label{RankDiscontinuityDef:DiscontinuousAndIsolating} There exists a set $A \subset \dom(P)$ such that $p$ is a discontinuity of $P|_A$, and such that $\rankp(A) \leq \alpha$.
	\end{enumerate}
\end{Proposition}
\begin{proof}
	(\ref{RankDiscontinuityDef:DiscontinuousAndIsolating}) $\implies$ (\ref{RankDiscontinuityDef:Local}): The set $A$ which witnesses (\ref{RankDiscontinuityDef:DiscontinuousAndIsolating}) also witnesses (\ref{RankDiscontinuityDef:Local}). 

	(\ref{RankDiscontinuityDef:Local}) $\implies$ (\ref{RankDiscontinuityDef:Neighborhood}): Let $A$ be such that $p$ is a discontinuity of $P|_A$ and such that $\rank_A(p) \leq \alpha$. By definition of rank, there is an open neighborhood $U$ of $p$ which is a $\leq\!\!\alpha$-isolating neighborhood of $p$ in $A$. Since $p \in U$ we have that $P|_{A \cap U}$ is discontinuous at $p$ by Lemma~\ref{Lem:DiscontinuityPreservedByOpenRestrictions}. Hence, $A$ and $U$ witness (\ref{RankDiscontinuityDef:Neighborhood}).

	(\ref{RankDiscontinuityDef:Neighborhood}) $\implies$ (\ref{RankDiscontinuityDef:DiscontinuousAndIsolating}): Let $A$ and $U$ witness (\ref{RankDiscontinuityDef:Neighborhood}) and set $B := A \cap U$. Then $p$ is a discontinuity of $P|_B$. 
	If $\alpha = 0$, then $U$ isolates $p$ in $A$, so $B = \{p\}$ and therefore $\rankp(B) = 0$. 
	Now suppose that $\alpha > 0$. Since $U$ is a $\leq\!\!\alpha$-isolating neighborhood of $p$ in $A$, Lemma~\ref{Lem:SurroundedByLowerRanks} gives $\rank_A(q) < \alpha$ for all $q \in B$ with $q \neq p$. By Lemma~\ref{Lem:OpenSetsPreserveRank}, for all such $q$, we have $\rank_B(q) = \rank_A(q)$. Hence $\rank_B(p) \leq \alpha$, and for every $q \in B$ with $q \neq p$ we have $\rank_B(q) < \alpha$. Therefore $\rankp(B) \leq \alpha$, so $B$ witnesses (\ref{RankDiscontinuityDef:DiscontinuousAndIsolating}).
\end{proof}

\begin{Theorem}	\label{Thm:RankDiscontinuityEquivalence}
	Let $P: \subset \B \tto \B$ be a problem, let $p \in \dom(P)$, and let $\alpha$ be a countable ordinal. Then, (\ref{RankDiscontinuityEquivalence:Item1}) $\implies$ (\ref{RankDiscontinuityEquivalence:Item2}) in the following list of statements.
	\begin{enumerate}
		\item \label{RankDiscontinuityEquivalence:Item1} There is a set $A \subset \dom(P)$ such that  $P|_A$ is discontinuous, $p$ is a discontinuity of $P|_A$ and $\rank_A(p) \leq \alpha$.
		\item \label{RankDiscontinuityEquivalence:Item2} Player~1 has a winning strategy for the rank-$\alpha$ discontinuity game for $P$ with initial move $p$.
	\end{enumerate}
\end{Theorem}
\begin{proof}
	We argue by induction. For $\alpha = 0$, we have that Player~1 wins the $0$-$P$ game with initial move $p$ if and only if $P(p)$ is empty, which is also equivalent to there being a set $A \subset \dom(P)$ (not necessarily an open subset of $\dom(P)$) such that $p$ is isolated on $A$ and $p$ is a point of discontinuity of $P|_{A}$. 

	Now suppose that $\alpha > 0$ and assume the theorem for all $\gamma < \alpha$. 

	Assume that (\ref{RankDiscontinuityEquivalence:Item1})  holds for some $\beta \leq \alpha$. By Proposition~\ref{Prop:NbhdRankingOfDiscontinuities}, there are a set $A \subset \dom(P)$ and an open neighborhood $U$ of $p$ such that $p$ is a discontinuity of $P|_{A \cap U}$ and $U$ is a $\leq\!\!\beta$-isolating neighborhood of $p$ in $A$. We construct a Player~1 strategy $\Gamma$. 

	Set Player~1's first move $\Gamma(\lambda)$ equal to $p$. Let $q \in P(p)$. Let $\rho \sqsubset q$ be a witness to Theorem~\ref{Thm:ExtendingSolutions} for $q$ and set $\Gamma(q) = \rho$. Fix $\sigma \sqsubset p$. Let $\sigma' \sqsubset p$ be an extension of $\sigma$ such that $\llb \sigma'\rrb \subset U$. By our assumption on $\rho$, there is $\tau \sqsupset \sigma'$ with $\tau \not \sqsubset p$ such that $P|[\llb \tau \rrb \times \llb \rho \rrb]$ is not continuous. 
	By Theorem~\ref{Thm:ContinuityIsContinuityAtAllPoints}, $P|[\llb \tau \rrb \times \llb \rho \rrb]$ has a point of discontinuity $p'$ on $A \cap \llb \tau \rrb$. Since $\llb \tau \rrb \subset U$ and $\tau \not\sqsubset p$, we have $p' \in A \cap U$ and $p' \neq p$. Thus Lemma~\ref{Lem:SurroundedByLowerRanks} gives $\rank_A(p') < \beta$. Let $\gamma = \rank_A(p')$. Then $\gamma < \beta \leq \alpha$. Set $\Gamma(q,\sigma) := \gamma$. 
	
	By Lemma~\ref{Lem:OpenSetsPreserveRank}, we have that $\rank_{A \cap \llb \tau \rrb}(p') = \rank_A(p') = \gamma$. Hence, (\ref{RankDiscontinuityEquivalence:Item1}) holds for $A$, $p'$, $\gamma$, and  $P|[\llb \tau \rrb \times \llb \rho \rrb]$. By induction, Player~1 has a winning strategy in the $\gamma$-$P|[\llb \tau  \rrb\times \llb \rho \rrb]$ game, which we set $\Gamma$ to follow. 
	
	Since this works for any Player~2 choices of $q$ and $\sigma$, we may conclude that $\Gamma$ is a winning strategy for Player~1 in the $\alpha$-$P$ game. 
\end{proof}
This completes the proof of Theorems~\ref{Thm:MainTheorem} and \ref{Thm:MainTheoremPointWise}.  As a consequence, we also get an extension of Theorem~\ref{Thm:DiscontinuityToDiscontinuity} to ranked discontinuities. 

\begin{Corollary}\label{Cor:RanksAreWeihrauchInvariant}
	Suppose that $P \leq_\W^* Q$ via $\Phi$ and $\Psi$ and suppose that $p \in \dom(P)$ is a rank-$\alpha$ discontinuity of $P$. Then $\Phi(p)$ is a rank-$\beta$ discontinuity of $Q$ for some $\beta \leq \alpha$.
\end{Corollary}
\begin{proof}
	By Theorem~\ref{Thm:MainTheoremPointWise}, we have that Player 1 wins the $\alpha$-$P$ game with first move $p$. By Theorem~\ref{Thm:DiscontinuityGameInvariant}, we have that Player~1 wins the $\alpha$-$Q$ game with first move $\Phi(p)$. By Theorem~\ref{Thm:DiscontinuityGameCharacterization} we have that $\Phi(p)$ is a rank-$\beta$ discontinuity of $Q$ for some $\beta \leq \alpha$.
\end{proof}

\subsection{A Note  on Single Valued Functions}
\label{sec:SingleValuedFunctions}
\input{SingleValued.tex}

%% file: SingleValued.tex
We briefly discuss the ranks of local discontinuities of single valued functions. When referring to discontinuities of single valued functions in this subsection, we mean the classical $\epsilon$-$\delta$ definition of discontinuity at a point: $p$ is a discontinuity of $f: \B \to \B$ if there exists $\rho \sqsubset f(p)$ such that for all $\sigma \sqsubset p$ we have $f(\llb \sigma \rrb) \not\subset \llb \rho \rrb$. 

We say that a single-valued function $f: \B \to \B$ is \emph{locally continuous at} $p \in \B$ if there is an open neighborhood $U$ of $p$ such that $f|_U$ is continuous. We say that $f$ is \emph{locally discontinuous} at $p \in \B$ if $f$ is not locally continuous at $p$. 

It is well known that $f$ is continuous if and only if $f$ is locally continuous at each point of its domain. It is also well known that this equivalence does not extend to points.  Indeed, the common pedagogical example of a function which is continuous at exactly one point
\[
    f(p) = \begin{cases}
        0^\omega & \text{if } p \in \B \setminus \Q\\
        p & \text{if } p \in \Q,
    \end{cases}
\]
is continuous at $0^\omega$ but also locally discontinuous at $0^\omega$. 

Local continuity at $p$ implies continuity at $p$. On the other hand, the converse is partially true in the sense that $f$ is locally discontinuous at $p$ if and only if $p$ is in the closure of the set of discontinuities of $f$. This implies that local discontinuities of single valued functions are all of rank either $1$ or $2$, which we prove below. 

\begin{Proposition}
    \label{Prop:DiscontinuitiesOfSingleValuedFunctions}
    Let $f: \B \to \B$ be a single valued function and let $p \in \dom(f)$ be a discontinuity of $f$. Then, $p$ is a rank-$1$ discontinuity of $f$. 
\end{Proposition}
\begin{proof}
    We show that Player~1 has a winning strategy in the rank-$1$ discontinuity game for $f$ with initial move $p$. Suppose that Player~1 plays $p$ on Step~\ref{GameStep:Player1Plays}. Since $f$ is single-valued, Player~2 must play $f(p)$ on Step~\ref{GameStep:Player2Playsq}. Since $p$ is a discontinuity of $f$, there is an open set $V \ni f(p)$ such that for all neighborhoods $U \ni p$, $f(U) \not \subset V$. Then, Player~1 can play some $\rho$ at Step~\ref{GameStep:CheckAlphaGreaterThanZero} such that $f(p) \in \llb \rho \rrb \subset V$. Then, for any $\sigma \sqsubset p$ played by Player~2 at Step~\ref{GameStep:Player2PlaysSigma}, we have that $f(\llb \sigma \rrb) \not\subset V$. In particular, there is some $p' \in \llb \sigma \rrb$ such that $f(p') \not\in V$. Then, Player~1 can play $p'$ at Step~\ref{GameStep:Player1Plays} and win the game. 
\end{proof}

\begin{Proposition}
    Let $f: \B \to \B$ be a single valued function and let $p \in \dom(f)$ be a local discontinuity of $f$ but not a discontinuity of $f$. Then, $p$ is a rank-$2$ discontinuity of $f$.
\end{Proposition}
\begin{proof}
    First, we show that Player~1 has a winning strategy in the rank-$2$ discontinuity game for $f$ with initial move $p$. Suppose that Player~1 plays $p$ on Step~\ref{GameStep:Player1Plays}. Since $f$ is single-valued, Player~2 must play $f(p)$ on Step~\ref{GameStep:Player2Playsq}. Player~1's Step~\ref{GameStep:CheckAlphaGreaterThanZero} move does not matter; we can have them play $\lambda$, the empty string. For any Player~2 move $\sigma \sqsubset p$, there is a discontinuity $p'$ of $f$ with $p' \in \llb \sigma \rrb$  since $p$ is a local discontinuity of $f$. Player~1 can then play $1$ at Step~\ref{GameStep:Player1ChoosesOrdinal} and then follow the winning strategy for $p'$ in the rank-$1$ discontinuity game for $f$ with initial move $p'$.

    Next, we show that Player~2 has a winning strategy in the rank-$1$ discontinuity game for $f$ with initial move $p$. Suppose that Player~1 plays $p$ on Step~\ref{GameStep:Player1Plays}. Since $f$ is single-valued, Player~2 must play $f(p)$ on Step~\ref{GameStep:Player2Playsq}. For any Player~1 move $\rho$ at Step~\ref{GameStep:CheckAlphaGreaterThanZero}, there is some $\sigma \sqsubset p$ such that $f(\llb \sigma \rrb) \subset \llb \rho \rrb$ by continuity of $f$ at $p$. We have Player~2 play $\sigma$ at Step~\ref{GameStep:Player2PlaysSigma}. Then, Player~1 must play $0 < 1$ at Step~\ref{GameStep:Player1ChoosesOrdinal}. For any Player~1 move $p'$ at Step~\ref{GameStep:Player1Plays}, we have that $f(p') \in \llb \rho \rrb$. Then, Player~2 can play $f(p')$ at Step~\ref{GameStep:Player2Playsq}. Since the ordinal is $0$, Player~2 wins at Step~\ref{GameStep:CheckAlphaGreaterThanZero}. 
\end{proof}

%% file: AchromaticRamseyTheorems.tex
For a coloring $c: \N \to \N$, recall that we defined $\TS^1_\N(c) = \{A \subset \N: A \text{ is infinite and } c(A) \neq \N\}$. Consider the following proof that $\TS^1_\N$ is not continuous. 

\begin{proof}
    Suppose that $\Phi$ is a continuous realizer for $\TS^1_\N.$ We construct a $c: \N \to \N$ such that $c(\Phi(c)) = \N$. Let $c_0 \equiv 0$ be the constant coloring which colors every element $0$. Then, $\Phi(c_0)$ is some infinite set $A_0 \in 2^\N$. Since $A_0$ is infinite, there is $\phi_0 \sqsubset A_0$ such that $|\{i: \phi_0(i) = 1\}| \geq 1$. Since $\Phi$ is continuous, there is $\theta_0 \sqsubset c_0$ such that $\Phi(\llb \theta_0 \rrb) \subset \llb \phi_0 \rrb$. Then, $0 \in c'(\Phi(c'))$ for any $c'$ extending $\theta_0$.
    
    We can repeat this process above $\theta_0$ and $\phi_0$ to obtain $\theta_1 \sqsupset \theta_0$ and $\phi_1 \sqsupset \phi_0$ such that $\{0,1\} \subset c'(\Phi(c'))$ for any $c'$ extending $\theta_1$. Continuing in this way, we obtain a sequence of strings $\theta_0 \sqsubset \theta_1 \sqsubset \dots$ such that $\{0, 1, \dots, k\} \subset c'(\Phi(c'))$ for any $c'$ extending $\theta_k$. Let $c = \bigcup_{k} \theta_k$. Then, $c$ is a coloring such that $c(\Phi(c)) = \N$, contradicting the assumption that $\Phi$ is a realizer for $\TS^1_\N$.
\end{proof}

We can view the proof above as a run of an infinite version of the discontinuity game in which, at the $i$'th stage, Player~1 plays $c_i$, Player~2 responds with $\Phi(c_i)$, Player~1 plays $\phi_i$, and Player~2 responds with $\theta_i$. At each stage, Player~1 forces Player~2 to add a specific color to their solution. On the other hand, Player~2 can ensure that Player~1 must adopt such a strategy by always playing a set with as few new colors as possible. 

A natural question is whether each stage of the above proof corresponds to a change in uniform strength. Since Player~1 can force Player~2 to add a new color at each stage, in $k$ stages of play they can build a partial coloring so that $|c(\Phi(c))| \geq k$. Furthermore, they can build such a coloring with exactly $k+1$ many colors. The need to use more stages materializes in the strictness of the reductions 
\[
    \TS^1_2 >_\W \TS^1_3 >_\W \TS^1_4 >_\W \dots >_\W \TS^1_\N.
\] 

One can go one step further: what if the number of times Player~2 can add a new color to their solution is determined dynamically? Instead of fixing the limit to the number of colors in the solution, we could allow the solver to increase the number of colors a number of times controlled by an ordinal $\beta$. At the same time, we could also allow the instance to have a number of colors controlled by an ordinal $\alpha$. This leads to a family of problems $\RT^n_{\alpha,\beta}$, the thin set and achromatic Ramsey theorems for transfinitely-valued number of colors. 

\subsection{Definition of \texorpdfstring{$\RT^n_{\alpha,\beta}$}{RT\^{}n\_\{alpha,beta\}}}

We now define a class of thin set Ramsey theorems with ordinal-valued indices controlling both the number of colors in the coloring and the number of colors allowed in the thin set. 
On both the instance and solution side, the ordinal controls the number of colors by decreasing each time a new color is added. Once the ordinal reaches zero, no new colors may be added. 

Before we can define $\RT^n_{\alpha,\beta}$, we define represented spaces $\mathbf{C}^n$ of colorings from $[\N]^n$ to $\N$; $\mathbf{S(\alpha)}$ of decreasing sequences of ordinals strictly less than $\alpha$; $\mathbf{C^n_\alpha}$ of pairs $(c,\vec{\eta})$ of colorings $c$ and decreasing sequences of ordinals $\vec{\eta}$ such that $\vec{\eta}$ decreases each time a new color is added to $c$; and $\mathbf{R}^c_\beta$ of pairs $(A,\vec{\zeta})$ of sets $A$ and decreasing sequences of ordinals $\vec{\zeta}$ such that $\vec{\zeta}$ must decrease each time a new color is added to $c([A]^n)$.

To keep track of when new colors are added to $c$, we need a concrete representation of colorings $C^n = \{c| c: [\N]^n \to \N\}$. For each $n \in \N$, fix a one-to-one enumeration $\langle \cdot \rangle: [\N]^n \to \N$ of size $n$ subsets of $\N$. We then define the represented space $\mathbf{C}^n = (C^n, \delta_{\mathbf{C}^n})$ of $\N$-colorings of $n$-element subsets of $\N$ by 
    \[
        \delta_{\mathbf{C}^n}(p)(a_0, a_1, \dots, a_{n-1}) = p(\langle a_0, a_1, \dots, a_{n-1} \rangle).
    \]
Note that $\deltarep{\mathbf{C}^n}$ is bijective. We identify $c \in C^n$ with $\delta_{\mathbf{C}^n}^{-1}(c)$. Thus, we write $c(i)$ for $c(\langle a_0, a_1, \dots, a_{n-1} \rangle)$ where $i = \langle a_0, a_1, \dots, a_{n-1} \rangle$.

While we needed rank-witnessing representations of countable ordinals to define $\ACCN^\alpha$, we only need enumerations of the ordinals below either $\alpha$ or $\beta$ to define $\RT^n_{\alpha,\beta}$. For each $0 < \gamma < \omega_1$, fix an enumeration $w_\gamma: \N \to \gamma$, possibly with repetitions. When $\gamma < \omega$, we also require that $w_\gamma$ be computable. 

Let $S(\alpha)$ be the set of strictly decreasing sequences of ordinals strictly less than $\alpha$.

For an ordinal $\alpha$ and a finite sequence $\vec{v} = v_0v_1\dots v_{k-1} \in\N^{<\N}$, we associate an element of $S(\alpha)$ by
\[
s_\alpha(\vec{v}) = \gamma_0 \gamma_1 \dots  \gamma_{k-1}
\]
where $\gamma_{-1}=\alpha$ and for $i \geq 0$,
\[
    \gamma_i = w_{\gamma_{i-1}}(v_i).
\]
Note that $s_\alpha$ is not total on $\N^{<\N}$: if $\gamma_i=0$ for some $i<k - 1$, then $s_\alpha(\vec{v})$ is not defined.

Given a coloring $c \in C^n$ and $i \in \N$, we say that $i = \langle a_0, \dots, a_{n-1} \rangle$ \emph{adds a new color} to $c$ if and only if $c(i) \not\in c(\{j: j < i\})$. Let $N(c) = \{i: i \text{ adds a new color to } c\}$ be the set of indices which add new colors to $c$. Note that $N(c)$ is infinite if and only if $c$ has infinitely many colors. Once we define $\RT^n_{\alpha,\beta}$ for $\alpha < \omega_1$, this will imply that all instances contain only finitely many colors.

Next, we describe how we encode $(c, \vec{\eta})$ in a way such that $\vec{\eta}$ decreases whenever a new color is added to $c$. We will require that our names interleave the information from $c$ and from $\vec{\eta}$. For each coloring $c \in C^n$ let
\[
    m_0<\dots<m_{k-1}
\]
be the stages at which we add a new color to $c$. Then, $c$ has form
\[
    c=\theta_0\theta_1\dots\theta_{k-1} p
\]
such that for $0\le i<k$,
\[
    \theta_0\theta_1\dots\theta_i = c|_{m_i + 1}. 
\]
In other words, each $\theta_i$ ends exactly when the $i$'th new color appears. For $\vec{v}=(v_0,\dots,v_{k-1})\in\N^{<\N}$ define the interleaving $c \smashtimes \vec{v}$ of $c$ and $\vec{v}$ by \footnote{We borrow the symbol $\smashtimes$ from topology, in which it is sometimes used to denote the smash product of two pointed topological spaces. While we claim no mathematical connection to the smash product of pointed spaces, the author believes that ``smash product'' is an apt name for what $\smashtimes$ does to its arguments as defined here.}
\[
    c \smashtimes \vec{v}
    :=
    \theta_0\hash v_0 \theta_1\hash v_1\cdots\theta_{k-1}\hash v_{k-1}p.
\]
We define the same operation $\theta \smashtimes \vec{v}$ for initial segments of colorings $\theta \sqsubset c \in C^n$ and $\vec{u} \sqsubset \vec{v}$ such that $|\vec{u}| = |\range(\theta)| =: r $ by
\[
    \theta \smashtimes \vec{u} = \theta_0\hash u_0 \theta_1\hash u_1\cdots\theta_{r-1}\hash u_{r-1}\theta_r,
\]
where $\theta = \theta_0\theta_1\dots\theta_{r}$ and, for $i < r$, the string $\theta_i$ ends when a new color is added to $\theta$, as above. 

Fix $\alpha<\omega_1$. We define 
\[
C^n \smashtimes S(\alpha) = \{c \smashtimes \vec{v}: c \in C^n, \vec{v} \in \dom(s_\alpha), \text{ and } c \smashtimes \vec{v} \text{ is defined}\}.
\] 
Hence, if $c \smashtimes \vec{v} \in C^n \smashtimes S(\alpha)$, then $s_\alpha(\vec{v})$ is a decreasing sequence of ordinals strictly less than $\alpha$ whose decreases correspond with new colors being added to $c$. 

We then define the representation $\deltarep{C^n_\alpha}: \subset \extendedBaireSpace \tto C^n \times S(\alpha)$ by
\[
    \dom(\deltarep{C^n_\alpha}) = C^n \smashtimes S(\alpha)
\]
and
\[
    \deltarep{C^n_\alpha}(c\smashtimes \vec{v})=\left(c,s_{\alpha}(\vec{v})\right).
\]
and the represented space
\[
    \mathbf{C^n_\alpha}=\left(\range(\deltarep{C^n_\alpha}),\deltarep{C^n_\alpha}\right)
\]
For $\alpha = \N$, we define $\mathbf{C}^n_\N$ to be the set of colorings of $[\N]^n$ with finitely many colors. 

For $c \smashtimes \vec{v} \in C^n \smashtimes S(\alpha)$, let $\theta \sqsubset c$ be an initial segment of $c$ and let $\vec{v}' \sqsubset \vec{v}$ be an initial segment of $\vec{v}$. In the same way that $c \smashtimes \vec{v}$ is a name for $(c, s_{\alpha}(\vec{v}))$, we say that $\theta \smashtimes \vec{v}'$ is a name for $(\theta, s_{\alpha}(\vec{v}'))$. 

Now fix $c\in C^n$ and infinite $A\subseteq\N$. Suppose that $|c([A]^n)|\leq k$. To get the correct strong Weihrauch degree of $\RT^n_{k,j}$ for $j < k \in \N$, we need to allow solutions to choose their ordinals without knowing exactly when new colors are added to $A$. For a sequence $\vec{v} = (v_0,\dots,v_{k-1})$, define a set $A \smashtimes_c \vec{v}$ of $c$-respecting names for $(A,\vec{v})$ by
\[
    A \smashtimes_c \vec{v} := \{\phi_0 \hash v_0 \phi_1 \hash v_1 \cdots \phi_{k-1} \hash v_{k-1} q: \phi_0 \phi_1 \cdots \phi_{k-1} q = A \text{ and } (\forall i \leq k)(|c([\phi_0 \phi_1 \cdots \phi_{i -1}]^n)| \leq i)\}.
\]
Hence, elements of $A \smashtimes_c \vec{v}$ can only add new colors to $A$ if $\vec{v}$ is long enough.  

For a fixed $c$ and $\beta < \omega_1$, we then define $2^\omega \smashtimes_c S(\beta) \subset \extendedBaireSpace$ by
\[
    2^\omega \smashtimes_c S(\beta) = \bigcup \{A \smashtimes_c \vec{v}: A \subset \N, \vec{v} \in \dom(s_\beta), \text{ and } A \smashtimes_c \vec{v} \text{ is defined}\}.
\]

We define the representation $\delta_{\mathbf{R}^c_\beta}: \subset \extendedBaireSpace \tto 2^\omega \times S(\beta)$ by $\dom(\delta_{\mathbf{R}^c_\beta}) = 2^\omega \smashtimes_c S(\beta)$ and for all $x \in A \smashtimes_c \vec{v}$, we set $\deltarep{R^c_\beta}(x) = (A, s_\beta(\vec{v}))$. Then, we have the represented space
\[
    \mathbf{R}^c_\beta = \left(\range(\deltarep{R^c_\beta}), \deltarep{R^c_\beta}\right).
\]

For $\phi \sqsubset A$ and $\vec{v'} = v_0 \dots v_{r-1}\sqsubset \vec{v}$, we define 
\[
\phi \smashtimes_c \vec{v}' = \{\rho: (\exists q \in A \smashtimes_c \vec{v})(\rho \sqsubset q) \text{ and }\rho = \phi_0\hash v_0 \dots \phi_{r-1}\hash v_{r-1} \phi_r \text{ such that } \phi_0\phi_1\dots \phi_r = \phi\}. 
\]
We say that $\rho$ is a name for $(\phi, s_\beta(\vec{v}'))$ in $\mathbf{R}^c_\beta$ if $\rho \sqsubset q$ for some $q \in A \smashtimes_c \vec{v}$ such that $\deltarep{R^c_\beta}(q) = (A, s_\beta(\vec{v}))$, $\phi \sqsubset A$ and $\vec{v}' \sqsubset \vec{v}$.

We are now ready to give the final coding of $\RT^n_{\alpha,\beta}$.
\begin{Definition}
    \label{def:AchromaticRamseyTheorem}
    Fix $\beta<\omega_1$ and $\alpha<\omega_1$ [respectively, $\alpha=\N$]. For each $n\in\N$, define
    \[
        \RT^n_{\alpha,\beta}: \mathbf{C^n_\alpha} \tto \extendedBaireSpace
    \]
by $\RT^n_{\alpha,\beta}(c,s) = \dom(\deltarep{R^c_\beta})$ [respectively, $\RT^n_{\N,\beta}(c) = \dom(\deltarep{R^c_\beta})$]. 
\end{Definition}

Thus, instances of $\RT^n_{\alpha,\beta}$ are names for elements of $C^n \smashtimes S(\alpha)$ and solutions to $(c,s)$ are names for elements of $\mb{R}^c_\beta$. 

The choice of enumerations $w_\gamma$ does not change the strong continuous Weihrauch degree of $\RT^n_{\alpha,\beta}$, since from any two enumerations $w$ and $w'$ of $\gamma$ one can computably transform names for $(c,s)$ in $\mathbf{C^n_\alpha}$ to a name for $(c,s)$ in the version of $\mathbf{C^n_\alpha}$ defined using $w'$, and similarly for names for $(A,s')$ in $\mathbf{R}^c_\beta$ (see Theorem~\ref{thm:BasicThinSetReductions}). For finite ordinals, such oracles are always computable, so choice of enumeration of $j$ and $k$ does not change the strong Weihrauch degree of $\RT^{n}_{k,j}$. For $\omega > k > j$, the problems $\RT^n_{k,j}$ as defined here are Weihrauch equivalent to their standard definitions. For the forward directions of these reductions, we send $c$ to $(c, k -1, \dots, 1, 0)$ and send $(c, s)$ to $c$. For the backward directions, we send $A$ to $ \hash(j-1)\hash(j-2)\hash \dots \#0 A \in A \smashtimes_c (j -1, \dots, 1, 0)$ (where we assume without loss of generality that $w_h(i) = i$ for all $i < h \in \N$) and send $(A, s)$ to $A$.

\begin{RemarkNum}
    \label{rem:DealingWithNames}
    We conclude this subsection with a notational note. When a $\mathbf{C}^n_\alpha$-name $p$ is given for $(c,\vec{\eta})$, each initial segment $\sigma \sqsubset p$ has a unique $\theta \sqsubset c$ and $\vec{\eta'} \sqsubset \vec{\eta}$ such that $\sigma$ is a name for $(\theta, \vec{\eta'})$. Similarly, when a $\mathbf{R}^c_\beta$-name $q$ is given for $(A,\vec{\zeta})$, each initial segment $\rho \sqsubset q$ has a unique $\phi \sqsubset A$ and $\vec{\zeta'} \sqsubset \vec{\zeta}$ such that $\rho$ is a name for $(\phi, \vec{\zeta'})$. In this paper, we will never refer to multiple names for the same instance or solution of $\RT^n_{\alpha,\beta}$, so when names $p$ and $q$ are given for $(c,\vec{\eta})$ and $(A,\vec{\zeta})$, respectively, we shall write $(\theta, \vec{\eta'}) \sqsubset (c,\vec{\eta})$ and $(\phi, \vec{\zeta'}) \sqsubset (A,\vec{\zeta})$ to refer to the unique initial segments of $\sigma \sqsubset p$ and $\rho \sqsubset q$ which are names for $(\theta, \vec{\eta'})$ and $(\phi, \vec{\zeta'})$, respectively.
    
    Note that for fixed $\theta \sqsubset c$ and $\vec{\eta'} \sqsubset \vec{\eta}$, the existence of a name $\sigma \sqsubset p$ for $(\theta, \vec{\eta'})$ depends on the structure of $p$. A similar statement holds for $\phi \sqsubset A$ and $\vec{\zeta'} \sqsubset \vec{\zeta}$. Nevertheless, it does hold that for each $\theta \sqsubset c$ there exists $\vec{\eta'} \sqsubset \vec{\eta}$ such that there is a name $\sigma \sqsubset p$ for $(\theta, \vec{\eta'})$, and for each $\vec{\eta'} \sqsubset \vec{\eta}$ there exists $\theta \sqsubset c$ such that there is a name $\sigma \sqsubset p$ for $(\theta, \vec{\eta'})$. Again, a similar statement holds for $\phi \sqsubset A$ and $\vec{\zeta'} \sqsubset \vec{\zeta}$. We will adhere to this limitation when referring to $(\theta, \vec{\eta'})$ as an initial segment of $(c,\vec{\eta})$ and $(\phi, \vec{\zeta'})$ as an initial segment of $(A,\vec{\zeta})$.
\end{RemarkNum}

\subsection{On the Size of the Thin Set}
In this section, we compute the least $\gamma$ such that $\RT^n_{\alpha,\beta}$ has a rank-$\gamma$ discontinuity. Our primary method is to construct winning strategies for Player~1 and Player~2 in the discontinuity game of different ranks. Throughout, we refer to the moves in Definition~\ref{def:DiscontinuityGame} as Step~1 through Step~5 moves.

We immediately obtain that no instances of $\RT^n_{\alpha,0}$ have a solution, since any infinite subset of $\N$ must have at least one color. Hence, $\RT^n_{\alpha,0}$ has a rank-0 discontinuity. Furthermore, if $\alpha \leq \beta$, then $\RT^n_{\alpha,\beta}$ is continuous via the continuous function which sends a name of $(c, (\gamma_0, \gamma_1, \dots, \gamma_k))$ to a name of $(\N, (\gamma_0, \gamma_1, \dots, \gamma_k))$. Thus, $\RT^n_{\alpha,\beta}$ has no discontinuities when $\alpha \leq \beta$. This leaves the cases when $\alpha > \beta > 0$.
 
Before proceeding, we establish some basic reductions between $\RT^n_{\alpha,\beta}$ and $\RT^n_{\nu,\mu}$.

\begin{Theorem}
    \label{thm:BasicThinSetReductions}
    Fix any $n \in \N$. Suppose $\beta \leq \mu < \nu \leq \alpha \in \omega_1 \cup \{\N\}$. Then,
    \[
        \RT^n_{\alpha,\beta} \geq_\sW^* \RT^n_{\nu,\mu}.
    \]
    Furthermore, if $\omega > \alpha$, then
    \[
        \RT^n_{\alpha,\beta} \geq_\sW \RT^n_{\nu,\mu}.
    \]
\end{Theorem}
\begin{proof}
    First suppose that $\alpha < \omega_1$. Let
    \[
        W_1 = \{(s,s')\in \N^{<\N}\times\N^{<\N}: s_\alpha(s)=s_\nu(s')\}
    \]
    and
    \[
        W_2 = \{(s,s')\in \N^{<\N}\times\N^{<\N}: s_\beta(s)=s_\mu(s')\}.
    \]
    We show that
    \[
        \RT^n_{\alpha,\beta} \geq_\sW^{W_1 \oplus W_2} \RT^n_{\nu,\mu}
    \]
    in the sense that the forward and backwards reductions can be computed from $W_1 \oplus W_2$. If $\alpha < \omega$ then $W_1 \oplus W_2$ is computable, yielding the strong computable Weihrauch reduction.

    Since $\alpha \geq \nu$, there is a $W_1$-computable function which takes $C^n_{\nu}$-names of $(c,\vec{\eta})$ to $C^n_{\alpha}$-names of $(c,\vec{\eta})$. Let our forward function $\Phi$ be such a function. Since $\beta \leq \mu$, there is a $W_2$-computable function which takes names for $(A, \vec{\zeta})$ in $\RT^n_{\alpha,\beta}(c,\vec{\eta})$ to names for $(A,\vec{\zeta})$ in $\RT^n_{\nu,\mu}(c,\vec{\eta})$. Let our backward function $\Psi$ be such a function. 

    The case for $\alpha = \N$ is similar, but we only need to send $(c,\vec{\eta})$ to $c$. 
\end{proof}

Now we are ready to prove the main theorems of this section.

\begin{Theorem}
    \label{thm:rankOfRT1}
    Let $\beta < \alpha \in \omega_1 \cup \{\N\}$. Then,
    \begin{enumerate}
        \item \label{itm:rank1-omega-beta} $\RT^1_{\alpha,\beta}$ has a rank-$\beta$ discontinuity and
        \item \label{itm:rank1-omega-gamma} $\RT^1_{\alpha,\beta}$ does not have a rank-$\gamma$ discontinuity for any $\gamma < \beta$.
    \end{enumerate}
\end{Theorem}
\begin{proof}
    The proof amounts to showing that in optimal play of the $\gamma$-$\RT^1_{\alpha,\beta}$ game, Player~2 must decrease their ordinal exactly once in each round. If $\beta > \gamma$, then Player~2 can stay above $\gamma$ while doing this, while if $\beta \leq \gamma$ then their ordinal reaches 0 before the game ordinal. We now provide the details. 

    Proof of (\ref{itm:rank1-omega-beta}): Since $\RT^1_{\N,\beta} \geq_\W^* \RT^1_{\alpha,\beta}$ for all $\alpha < \omega_1$, it suffices to argue for $\alpha < \omega_1$. First, we prove the following claim. 
    \begin{Claim}
        Let $\theta \in \N^{<\N}$ be an initial segment of a coloring and $\phi \in 2^{<\N}$ be an initial segment of a set such that $\theta$ has a finite extension $\hat{\theta} \sqsupset \theta$ such that $|\hat{\theta}(\phi)| = |\hat{\theta}(\N)| =: k$.  Let $\vec{\zeta} = \zeta_1 \zeta_2\dots \zeta_{r} \in S(\beta)$ be a nonempty decreasing sequence of ordinals with length $r \geq k$. Let $\vec{\eta} \in S(\alpha)$ be a proper initial segment of $\beta \zeta_1 \zeta_2 \dots \zeta_k$. 

        Then, Player~1 has a winning strategy in the $\zeta_k$-$\RT^1_{\alpha,\beta}|[\llb (\theta, \vec{\eta})\rrb \times \llb (\phi, \vec{\zeta} ) \rrb ]$ game.
    \end{Claim}
    \begin{proof}[Proof of Claim]
        We argue by induction on $\zeta_k$ and build a winning strategy $\Gamma$ for Player~1.
        Let $\hat{\theta} \sqsupset \theta$ be such that $|\hat{\theta}(\phi)| = |\hat{\theta}(\N)| = k$. Let $M \in \N$ be such that $M \not\in \hat{\theta}(\N)$.
        Let $c$ extend $\hat{\theta}$ by coloring each not already colored element with $M$. Then, $|c(\N)| = k + 1$. At Step~1, have Player~1 play $(c, \beta \zeta_1 \dots \zeta_k) \sqsupset (\theta, \vec{\eta})$. Since $\vec{\eta} \sqsubset \beta \zeta_1 \dots \zeta_{k}$ is a proper initial segment, this is a legal move. Let $(A, \hat{\vec{\zeta}}) \sqsupset (\phi, \vec{\zeta})$ be Player~2's Step~2 response. Since $A$ is infinite, there is an initial segment $\phi' \sqsubset A$ such that $M \in c(\phi')$. Since $ \hat{\theta} \sqsubset c$ and $\phi \sqsubset A$, we can further require that $|c(\phi')| = k+1$. Hence, it must be that $\hat{\vec{\zeta}}$ has at least $k + 1$ many entries. If $\zeta_k = 0$, this is impossible and Player~1 wins. 
        
        Otherwise, $\zeta_k > 0$ and we obtain no contradiction from the existence of $(A, \vec{\hat{\zeta}})$. Let $\vec{\zeta'}$ be such that $(\phi',\vec{\zeta'}) \sqsubset (A, \hat{\vec{\zeta}})$ (depending on the name for $(A, \hat{\vec{\zeta}})$, it is possible that $\vec{\zeta'}$ must have length greater than $k+1$) and have $\Gamma$ play $(\phi', \vec{\zeta'})$ at Step~\ref{GameStep:CheckAlphaGreaterThanZero}. Let $\zeta_{k+1}$ be the $(k + 1)$'st entry of $\vec{\zeta'}$ (note that this is consistent with our previous notation because, even if $|\vec{\zeta}| \geq k+1$, we have $\vec{\zeta} \sqsubset \vec{\zeta'}$).  Then, for any Step~\ref{GameStep:Player2PlaysSigma} move $(\theta', \vec{\eta'}) \sqsubset (c, \beta \zeta_1 \dots \zeta_k)$, have $\Gamma$ play $\zeta_{k + 1}$ at Step~\ref{GameStep:Player1ChoosesOrdinal}. 

        We check that the conditions of the claim are satisfied for the $\zeta_{k+1}$-$\RT^1_{\alpha,\beta}|[\llb (\theta',\vec{\eta'}) \rrb \times \llb (\phi',\vec{\zeta'})\rrb]$ game. Since $c(\phi')$ has $k + 1$ many colors, there is an extension $\hat{\theta'} \sqsupset \theta'$ such that $|\hat{\theta'}(\phi')| = |\hat{\theta'}(\N)| = k + 1$. Furthermore, $\vec{\eta'} = \beta \zeta_1 \dots \zeta_{k}$ is a proper initial segment of $\beta \zeta_1 \dots \zeta_{k+1}$. Thus, all conditions of the claim are satisfied, so by the induction hypothesis, Player~1 has a winning strategy in the $\zeta_{k+1}$-$\RT^1_{\alpha,\beta}|[\llb (\theta',\vec{\eta'}) \rrb \times \llb (\phi',\vec{\zeta'})\rrb]$ game. We extend $\Gamma$ by this strategy, completing the proof. 
    \end{proof}
    While it might be possible to reformulate the claim so that it can be directly applied to $\RT^1_{\alpha,\beta}$, it seems simpler to handle the $\beta$'th round separately. We define a Player~1 strategy $\Gamma$ for the $\beta$-$\RT^1_{\alpha,\beta}$ discontinuity game.

    Define $c_0 \equiv 0$ to be the constant coloring which colors every element of $\N$ with $0$. Have $\Gamma$ play $(c_0, \beta)$ at Step~\ref{GameStep:Player1Plays}. Let $(A,\vec{\zeta})$ be Player~2's Step~\ref{GameStep:Player2Playsq} response with $\vec{\zeta} = \zeta_1\dots\zeta_k$. Since $A$ has a color, $\vec{\zeta} \neq \lambda$. Pick $\phi \sqsubset A$ such that $\{i: \phi(i) = 1\} \neq \emptyset$ and have $\Gamma$ play $(\phi, \vec{\zeta})$ at Step~\ref{GameStep:CheckAlphaGreaterThanZero}. Then, for any Step~\ref{GameStep:Player2PlaysSigma} move $(\theta, \vec{\eta}) \sqsubset (c_0, \beta)$, have $\Gamma$ play $\zeta_1$ at Step~\ref{GameStep:Player1ChoosesOrdinal}. Now, we check that the conditions of the claim are satisfied for the $\zeta_1$-$\RT^1_{\alpha,\beta}|[\llb (\theta, \vec{\eta}) \rrb \times \llb (\phi, \vec{\zeta})\rrb]$ game. Since $\phi$ has a color and $c_0$ is constant, there is an extension $\hat{\theta} \sqsupset \theta$ such that $|\hat{\theta}(\phi)| = |\hat{\theta}(\N)| = 1$. Furthermore, $\vec{\eta} \sqsubset \beta$ is a proper initial segment of $\beta \zeta_1$. Thus, all conditions of the claim are satisfied, so by the claim, Player~1 has a winning strategy in the $\zeta_1$-$\RT^1_{\alpha,\beta}|[\llb (\theta, \vec{\eta}) \rrb \times \llb (\phi, \vec{\zeta})\rrb]$ game. We extend $\Gamma$ by this strategy, completing the proof of (\ref{itm:rank1-omega-beta}).

    (\ref{itm:rank1-omega-gamma}) Since $\RT^1_{\N,\beta} \geq_\W^* \RT^1_{\alpha,\beta}$, by Theorem~\ref{Thm:DiscontinuityGameInvariant} it suffices to show that $\RT^1_{\N,\beta}$ does not have a rank-$\gamma$ discontinuity for any $\gamma < \beta$. First, we prove the following claim.

    \begin{Claim}
        Fix $\gamma < \beta$. Let $\theta \in \N^{<\N}$ be an initial segment of a coloring and $\phi \in 2^{<\N}$ be an initial segment of a set such that $\theta|_\phi$ is defined and so that $|\theta(\phi)| \leq k$. Let $\vec{\zeta} = \zeta_1 \zeta_2\dots \zeta_{k} \in S(\beta)$ be a possibly empty decreasing sequence of ordinals with length $k$ such that $\vec{\zeta}\gamma \in S(\beta)$.  

        Then, Player~2 has a winning strategy in the $\gamma$-$\RT^1_{\N,\beta}|[\llb \theta \rrb \times \llb (\phi, \vec{\zeta} ) \rrb ]$ game.
    \end{Claim}
    \begin{proof}[Proof of Claim]
        We argue by induction on $\gamma$. If $\gamma = 0$, then Player~2 can win by playing any solution to the instance given by Player~1's Step~\ref{GameStep:Player1Plays} move. Now suppose $\gamma > 0$ and that the claim holds for all $\gamma' < \gamma$. We define a winning strategy $\Delta$ for Player~2 in the $\gamma$-$\RT^1_{\N,\beta}|[\llb \theta \rrb \times \llb (\phi, \vec{\zeta} ) \rrb ]$ game. For any Step~\ref{GameStep:Player1Plays} move $c\sqsupset \theta$, let $H$ be a monochromatic set for $c$ and let $A$ be the result of replacing $H|_{|\phi|}$ with $\phi$. Then, $|c(A)| \leq k + 1$. Thus, having $\Delta$ play $(A, \vec{\zeta}\gamma)$ at Step~\ref{GameStep:Player2Playsq} is a legal move. For any Step~\ref{GameStep:CheckAlphaGreaterThanZero} move $(\phi', \vec{\zeta'}) \sqsubset (A, \vec{\zeta}\gamma)$, let $\theta' \sqsubset c$ be such that $\theta'|_{\phi'} = c|_{\phi'}$ is defined. Have $\Delta$ play $\theta'$ as their Step~\ref{GameStep:Player2PlaysSigma}. Then, Player~1 plays some $\gamma' < \gamma$ at Step~\ref{GameStep:Player1ChoosesOrdinal}. 

        We check that the conditions of the claim hold for the $\gamma'$-$\RT^1_{\N,\beta}|[\llb \theta' \rrb \times \llb (\phi', \vec{\zeta'})\rrb]$ game. First, we have that $\theta'|_{\phi'}$ is defined. Since $|c(A)| \leq k+1$ and $\phi' \sqsubset A$, we also obtain that $|\theta'(\phi')| \leq k + 1$. Finally, $\vec{\zeta'} \sqsubset \vec{\zeta}\gamma$, so $\vec{\zeta'}\gamma' \in S(\beta)$. Thus, all conditions of the claim are satisfied, so by the induction hypothesis, Player~2 has a winning strategy in the $\gamma'$-$\RT^1_{\N,\beta}|[\llb \theta' \rrb \times \llb (\phi', \vec{\zeta'})\rrb]$ game. We extend $\Delta$ by this strategy, completing the proof of the claim.
    \end{proof}
    We check that claim applies directly to $\RT^1_{\N,\beta} = \RT^1_{\N,\beta}|[\llb \theta \rrb \times \llb (\phi, \vec{\zeta}) \rrb]$ with $\theta = \phi = \vec{\zeta} = \lambda$, the empty string. Since $\{i: \phi(i) = 1\} = \emptyset$, we have that $\theta|_\phi$ is defined and $|\theta(\phi)| = 0$.  Since $\vec{\zeta} = \lambda$ and $\gamma < \beta$, we have that $\vec{\zeta}\gamma = \gamma \in S(\beta)$. Hence, the claim applies and Player~2 has a winning strategy in the $\gamma$-$\RT^1_{\N,\beta}$ discontinuity game for each $\gamma < \beta$. Thus, $\RT^1_{\N,\beta}$ does not have a rank-$\gamma$ discontinuity for any $\gamma < \beta$, completing the proof of (\ref{itm:rank1-omega-gamma}).
\end{proof}

When $n \geq 2$, Player~1 can force Player~2 to add an arbitrary number of colors at each round after the initial round, whereas they could only force a single color when $n = 1$. Hence, the rank of the discontinuity of $\RT^n_{\alpha,\beta}$ for $n \geq 2$  depends on the least ordinal $\gamma$ such that $\beta \leq \omega \cdot \gamma$. 
\begin{Theorem}
    \label{thm:rankOfRTn}
    Let $n \geq 2$ and $\beta < \alpha \in \omega_1 \cup \{\N\}$. Fix an ordinal $\gamma$. 
    \begin{enumerate}
        \item\label{itm:rankn-omega-gamma-geq-beta} If $\omega \cdot \gamma \geq \beta$, then Player~1 has a winning strategy in the $\gamma$-$\RT^n_{\alpha,\beta}$ game.
        \item\label{itm:rankn-omega-gamma-lt-beta} If $\omega \cdot \gamma < \beta$, then Player~2 has a winning strategy in the $\gamma$-$\RT^n_{\alpha,\beta}$ game.
    \end{enumerate}
\end{Theorem}
\begin{proof}
    The proof amounts to showing that in optimal play of the $\gamma$-$\RT^n_{\alpha,\omega \cdot \gamma + N}$ game, Player~1 forces Player~2 to add an arbitrarily large finite number of colors. Then, Player~2 can choose any limit ordinal smaller than $\omega \cdot \gamma$ and stay above it by picking a large enough finite portion of the Cantor normal form. We check the details below.

    (\ref{itm:rankn-omega-gamma-geq-beta}): Since $\RT^n_{\N,\beta} \geq_\W^* \RT^n_{\alpha,\beta}$ for all $\alpha < \omega_1$, it is sufficient to argue for $\alpha < \omega_1$ by the Player~1 direction of Weihrauch invariance of the discontinuity game (Theorem~\ref{Thm:DiscontinuityGameInvariant}). 
    
    First, we prove the following claim.
    \begin{Claim}
        Let $\mu$ be a countable ordinal. Fix $N \in \N$, $\theta \in \N^{<\N}$, $\vec{\eta} = \eta_1\eta_2\dots\eta_m \in S(\alpha)$ (or $\vec{\eta} = \lambda$, in which case we set $\eta_m = \alpha$), $\phi \in 2^{<\N}$ and $\vec{\zeta} = \zeta_1\zeta_2\dots \zeta_r \in S(\beta)$. Let $k \leq r$ be the number of colors $ k = |\theta([\phi]^n)|$ defined by $\theta$ on $\phi$. Suppose that the following properties are satisfied.
        \begin{enumerate}
            \item  \label{RTnabPf:LastElementOfEta1Big}  $\zeta_k \leq \omega \cdot \mu + N < \eta_m$, and
            \item $|\{i: \phi(i) = 1 \}| > N$
        \end{enumerate}
        Then, Player~1 has a winning strategy in the rank-$\mu$ discontinuity game for $\RT^n_{\alpha,\beta}|[\llb (\theta, \vec{\eta}) \rrb \times \llb (\phi, \vec{\zeta})\rrb].$
    \end{Claim}
    \begin{proof}[Proof of Claim]
        Suppose that the claim is true for all $\mu' < \mu$. We define a coloring $c \sqsupset \theta$ which adds only $N + 1$ many colors to $\theta$ while forcing any infinite set extending $\phi$ to add at least $N$ colors. Let $M$ be such that $M > \theta(a_1,\dots,a_n)$ for all $a_1,\dots,a_n \in \N$ for which the latter is defined. Let $K \subset \{i: \phi(i) = 1\}$ be the set of $N$ least elements of $\phi$. Define $c$ by
        \[
            c(a_1,a_2,\dots,a_n) = \begin{cases}
                \theta(a_1,\dots,a_n) &\text{if } \theta(a_1,\dots,a_n) \text{ is defined}\\
                M + a_i + 1&\text{if } \theta(a_1,\dots,a_n) \text{ is not defined and } i \text{ is least such that } a_i \in K\\
                M &\text{if } \theta(a_1,\dots,a_n) \text{ is not defined and } \neg(\exists i) (a_i \in K). 
            \end{cases}
        \]
        Then, $c$ has exactly $N + 1$ more colors than $\theta$. Furthermore, since every infinite set $A \sqsupset \phi$ contains each element of $K$ as well as infinitely many elements outside of $K$, the set $A$ has exactly $N +1$ more colors than $\phi$.
        
       Since $c$ adds at most $N + 1$ many colors to $\theta$ and $\eta_m > \omega \cdot \mu + N$ we can set $\Gamma$'s first move to be 
       \[
       \Gamma(\lambda):= (c, \vec{\eta'}),
       \]
       where
       \[
            \vec{\eta'} = \vec{\eta}(\omega \cdot \mu + N)(\omega\cdot\mu + (N - 1))\dots(\omega \cdot \mu + 1)(\omega\cdot \mu).
       \] 
        We have that $\zeta_k \leq \omega \cdot \mu + N$, that $\theta([\phi]^n)$ has $k$ colors and that any infinite set $A \sqsupset \phi$ has $N+1$ more colors than $\theta([\phi]^n)$. Hence, all valid Player~2 moves (if there are any; it is possible that $|c([\phi]^n)| > r$) have form 
        \[
        (A, \vec{\zeta} \zeta'_1 \zeta'_2 \dots \zeta'_{k'})
        \] 
        such that $\zeta_{k'} < \omega \cdot \mu$. If $k' = 0$ then write $\zeta'_{k'}$ for the last element of $\vec{\zeta}$. Let ${\vec{\zeta'}} =  \vec{\zeta} \zeta'_1 \zeta'_2 \dots \zeta'_{k'}$.  By ordinal division, there are $\mu'$ and $N'$ such that 
        \begin{equation}
        \label{eq:satisfyFirstpt1}
        \zeta'_{k'} = \omega \cdot \mu' + N'
        \end{equation}
        Since $\zeta'_{k'} < \omega \cdot \mu$ we also have that 
        $\mu' < \mu$.
        Then, we have $\Gamma$ play 
        \[
            (\phi', {\vec{\zeta'}})
        \]
        for some long enough $\phi' \sqsubset A$ so that this is a legal move and so that 
        \begin{equation}
            \label{eq:SatisfySecond}
            |\{i: \phi'(i) = 1\}| > N'.
        \end{equation}
        In response to any Player~2 Step~\ref{GameStep:Player2PlaysSigma} move $(\theta', \vec{\xi}) \sqsubset (c, \vec{\eta'})$, we have Player~1 play $\mu'$. 

        We now check that the induction hypothesis holds for the $\mu'$-$\RT^n_{\alpha,\beta}|[\llb (\theta', \vec{\eta'}) \rrb \times \llb (\phi', \vec{\zeta'})\rrb]$ game.
        \begin{enumerate}
            \item  By equation~\ref{eq:satisfyFirstpt1}, we have $\zeta'_{k'} \leq \omega \cdot \mu' + N'$. Furthermore, the last entry of $\vec{\xi}$ is greater than or equal to the last entry of $\vec{\eta'}$ which is $\omega \cdot \mu$ which is greater than $\omega \cdot \mu' + N'$ because $\mu > \mu'$.
            \item Equation~\ref{eq:SatisfySecond} is the second condition of the claim. 
        \end{enumerate}        
        This completes the proof of the claim.
    \end{proof}
    
    Due to lack of a $\phi \in 2^{<\N}$, the claim does not apply directly to the initial position of the $\gamma$-$\RT^n_{\alpha,\beta}$ game; indeed, Player~2 can always play a monochromatic set in the $\gamma$'th round. Hence, we treat the $\gamma$'th round separately.

    We construct a winning strategy $\Gamma$. We set $\Gamma$'s first move 
    $\Gamma(\lambda)$ to $(c_0, \beta)$ with $c_0 \equiv 0$, the constant coloring. 
    Consider some Player~2 response $(A, \vec{\zeta})$ 
    with  $\vec{\zeta} = \zeta_1\zeta_2\dots \zeta_r$.  
    By ordinal division, there are $\mu < \omega_1$ and $N < \omega$ such that $\zeta_{1} = \omega \cdot \mu + N$. Since $\omega\cdot \gamma \geq \beta > \zeta_1$, we have $\mu < \gamma$. Let $\phi$ be an initial segment of $A$ such that $|\{i: \phi(i) = 1\}| > N$ and have $\Gamma$ play $(\phi, \vec{\zeta})$ at Step~\ref{GameStep:CheckAlphaGreaterThanZero}. 
    In response to any $(\theta, \vec{\eta})$ played by Player~2 at Step~\ref{GameStep:Player2PlaysSigma}, we have $\Gamma$ play $\mu$ at Step~\ref{GameStep:Player1ChoosesOrdinal}. 

    We set $k = 1$ so that $\zeta_k = \zeta_1$. If $\vec{\eta} = \lambda$ then we set $\eta_m = \alpha$. Otherwise, $\vec{\eta} = \beta$ and we set $\eta_m = \beta$. In either case, $\eta_m \geq \beta > \zeta_k = \omega \cdot \mu + N$. We check that the conditions of the claim are satisfied for the $\mu$-$\RT^n_{\alpha,\beta}|[\llb (\theta, \vec{\eta}) \rrb \times \llb (\phi, \vec{\zeta})\rrb]$ game. 
    \begin{enumerate}
        \item We have already checked that $\zeta_k \leq \omega \cdot \mu + N < \eta_m$.
        \item We have that $|\{i: \phi(i) = 1\}| > N$ by construction. 
    \end{enumerate}

    Thus, by the claim, Player~1 has a winning strategy in the $\mu$-$\RT^n_{\alpha,\beta}|[\llb (\theta, \vec{\eta}) \rrb \times \llb (\phi, \vec{\zeta})\rrb]$ discontinuity game, so $\Gamma$ can play according to this winning strategy. This completes the proof of (\ref{itm:rankn-omega-gamma-geq-beta}).

    (\ref{itm:rankn-omega-gamma-lt-beta}): Since $\RT^n_{\N,\beta} \geq_\W^* \RT^n_{\alpha,\beta}$ for all $\alpha \leq \omega_1$, it is sufficient to argue for $\RT^n_{\N,\beta}$ by the Player~2 direction of Weihrauch invariance of the discontinuity game (Theorem~\ref{Thm:DiscontinuityGameInvariant}). First, we prove the following claim by induction on ordinals $\mu$ and $\nu$. 
    \begin{Claim}
        Let $\mu < \nu$ be countable ordinals with $\omega \cdot \nu < \beta$. Let $\phi \in 2^{<\N}$ and let $\theta \in \N^{<\N}$ be an initial segment of a coloring long enough such that $\theta|_{[\phi]^n}$ is defined. Let $\vec{\zeta} = \zeta_1  \zeta_2 \dots \zeta_k \in S(\beta)$ be such that either $\vec{\zeta} = \lambda$ or $\zeta_k \geq \omega \cdot \nu$ and also such that $(\phi, \vec{\zeta})$ is an initial segment of some element of $\mb{R}^c_\beta$ for some $c \sqsupset \theta$. If $\vec{\zeta} = \lambda$ then $\zeta_k$ is not yet defined and we set $\zeta_k := \beta \geq \omega \cdot \nu$. Then, Player~2 has a winning strategy in the $\mu$-$\RT^n_{\N,\beta}|[\llb \theta \rrb \times \llb(\phi, \vec{\zeta})\rrb]$ discontinuity game.
    \end{Claim}
    \begin{proof}[Proof of Claim]
        Suppose that the claim is true for all $\mu' < \mu$ and $\nu' < \nu$ with $\mu' < \nu'$. We define a Player~2 strategy $\Delta$. Let $c \sqsupset \theta$ be a Player~1 move in Step~\ref{GameStep:Player1Plays}. To build $\Delta$'s response, consider the coloring $c': [\N]^n \to \mathcal{P}(\N)$ defined by
        \[
            \begin{split}
            c'(a_1, a_2, \dots, a_n) = \{i \in \N: (\E b_1,\dots,b_n \in (\phi \cup \{a_1, \dots, a_n\}))( c(b_1, \dots, b_n) = i)\}.
            \end{split}
        \]  
        Since $c$ has finite range, so does $c'$.
        Let $K = |\{i: \phi(i) = 1\}|$ and let 
        \[
            N = \sum_{i = 0}^{n} \binom{K}{i} \binom{n}{n-i},    
        \]
        where $\binom{a}{b}$ denotes $a$ choose $b$, the binomial coefficient. Since $K = |\{i: \phi(i) = 1\}|$, the number $N$ is an upper bound on the number of ways one can choose $n$ many elements from $\phi \cup \{a_1, \dots, a_n\}$.  Hence,
        for any $a_1, \dots, a_n \in \N$, there are at most $N$ many colors in $c'(a_1, \dots, a_n)$. Since $c$ has finitely many colors, so does $c'$. Fix an infinite homogeneous set $H \subset \N$ for $c'$ such that $i \not\in H$ for all $i \in \dom(\phi)$. Then, $|c([H]^n)| \leq N$. Since $c([\phi]^n) \subset c'(a_1, \dots, a_n)$ for any $a_1, \dots, a_n \in \N$, we have that $|c([H \cup \phi]^n)| \leq N$ as well. Since $\mu < \nu$ and $\zeta_k \geq \omega\cdot \nu$, we have that $\zeta_k > \omega \cdot \mu + N$. Hence,  we can set 
        \[
        \Delta(c)= (H \cup \phi, \vec{\zeta}  (\omega \cdot \mu + N)  (\omega \cdot \mu + (N-1))  \dots (\omega \cdot \mu + 1) (\omega \cdot \mu)),
        \] a $\RT^n_{\N, \beta}$-solution to $c$ in $\llb(\phi,\vec{\zeta})\rrb$.

        In response to any initial segment $(\phi', \vec{\zeta'}) \sqsubset \Delta(c)$ played by Player~1 at Step~\ref{GameStep:CheckAlphaGreaterThanZero}, we have $\Delta$ play an initial segment $\theta' \sqsubset c$ at Step~\ref{GameStep:Player2PlaysSigma} so that $\theta'|_{[\phi']^n} = c|_{[\phi']^n}$ is defined. Then, at Step~\ref{GameStep:Player1ChoosesOrdinal}, Player~1 chooses an ordinal $\mu' < \mu$. By the induction hypothesis applied to $\mu'$ and $\nu' := \mu$, $\Delta$ has a winning strategy in the $\mu'$-$\RT^n_{\N,\beta}|[\llb \theta' \rrb \times \llb(\phi', \vec{\zeta'})\rrb]$ discontinuity game, so $\Delta$ can play according to this winning strategy. This completes the proof of the claim. 
    \end{proof}
    Since we have only assumed that $\beta > \omega \cdot \gamma$, the claim is not strong enough to directly show that $\RT^n_{\N,\beta}|[\llb \lambda \rrb  \times \llb \lambda \rrb] = \RT^n_{\N,\beta}$ does not have a rank-$\gamma$ discontinuity. However, $\Delta$ can play a truly monochromatic set on its first turn of the game so that it only needs to decrease its ordinal by $1$, staying  $\geq \omega \cdot \gamma$ while Player~1 must decrease their ordinal. 

    Let $c: [\N]^n \to \N$ be any coloring with finitely many colors played by Player~1 at Step~\ref{GameStep:Player1Plays} of the $\gamma$-$\RT^n_{\N,\beta}$ game. Let $H$ be a monochromatic set for $c$. Since $\beta > \omega\cdot\gamma$, we can define $\Delta(c) = (H, \omega \cdot \gamma)$. Then, for any $(\phi, \vec{\zeta}) \sqsubset (H, \omega \cdot \gamma)$ played by Player~1 at Step~\ref{GameStep:CheckAlphaGreaterThanZero}, we have $\Delta$ play a long enough initial segment $\theta \sqsubset c$ so that $\theta|_{[\phi]^n} = c|_{[\phi]^n}$ is defined. 

    Then, for any $\mu < \gamma =: \nu$ played by Player~1 at Step~\ref{GameStep:Player1ChoosesOrdinal}, the conditions of the claim are satisfied for the $\mu$-$\RT^n_{\N,\beta}|[\llb \theta \rrb \times \llb (\phi, \vec{\zeta}) \rrb]$ game. We extend $\Delta$ by this winning strategy, obtaining a winning strategy for Player~2 in the $\gamma$-$\RT^n_{\N, \beta}$ game. This completes the proof of (\ref{itm:rankn-omega-gamma-lt-beta}).
\end{proof}

\subsection{Guessability with Identified Errors}
In this section, we show the non-reduction of $\RT^n_{\alpha,\beta}$ to $\ACCN$. For $n \in \N$ and $\beta < \alpha \in \omega_{1} \cup \{\N\}$, we show that $\RT^n_{\alpha, \beta} \not\leq_\W^* \ACCN$ by showing that $\RT^n_{\alpha,\beta}$ is not continuously $2$-guessable with errors identified, motivated by Pauly's \cite{paulyMoreIndivisibilityQ2024} use of $k$-guessability to separate problems related to the tree pigeonhole principle. 

The inclusion of identified errors is necessary since $\RT^n_{\omega + 1, \omega}$ is continuously $2$-guessable. In fact, it is computable with a single mind change. To see this, let $\Theta_1(c,\vec{\eta}) = (\N, 0)$ for all $\vec{\eta}$ of length 1. Then, $c$ has only one color, so $(\N,0) \in \RT^n_{\omega +1, \omega}(c,\vec{\eta})$. If $c$ has at least two colors, then our solution process changes its mind and switches to $\Theta_2$.

The function $\Theta_2$ only needs to guess a solution to $(c, \vec{\eta})$ when $c$ has at least two colors. In this case, $\Theta_2$ can wait until $|\vec{\eta}| \geq 2$. Since $\vec{\eta} \in S(\omega + 1)$, the second element of $\vec{\eta} = \eta_0\eta_1 \dots \eta_{k-1}$ must be a finite ordinal $\eta_1 < \omega$. Then, $\Theta_2$ can replace $\eta_0$ with $\eta_1 + 1$, yielding 
\[
\Theta_2(c, \vec{\eta}) = (\N, (\eta_1 + 1)\eta_1 \dots \eta_{k-1}) \in \RT^n_{\omega + 1, \omega}(c, \vec{\eta}).
\]
Hence, for any $(c,\vec{\eta}) \in \dom(\RT^n_{\omega + 1, \omega})$, at least one of $\Theta_1(c,\vec{\eta})$ and $\Theta_2(c,\vec{\eta})$ is a solution, so $\RT^n_{\omega + 1, \omega}$ is continuous with a mind change. 

It will be helpful to separate the notion of a functional identifying its errors.
\begin{Definition}
    Let $\Theta: \subset \N^\N \tto (\N \cup \{\bot\})^\N$ be a partial functional and $P: \subset \N^\N \tto \N^\N$ be a problem. We say that $\Theta$ \emph{identifies its errors} with respect to $P$ if for each $p \in \dom(P)$, if $\Theta(p) \not\in P(p)$, then there is a $j$ such that $\Theta(p)(j)\converges = \bot$.
\end{Definition}
The next lemma gives a way to preserve identifying errors when composing with Weihrauch reductions.
\begin{Lemma}
    \label{Lem:PreserveIdentifyingErrors}
    Suppose that $P \leq_\W^* Q$ via continuous functionals $\Phi$ and $\Psi$. Let $\Theta: \subset \N^\N \tto (\N \cup \{\bot\})^\N$ be a partial functional which identifies its errors with respect to $Q$. Then, there is a functional $\Theta': \subset \N^\N \tto (\N \cup \{\bot\})^\N$ that identifies its errors with respect to $P$ such that whenever $\Theta(\Phi(p)) \in Q(\Phi(p))$ we have that $\Theta'(p) = \Psi(p, \Theta(\Phi(p)))$.
\end{Lemma}
\begin{proof}
    We define $\Theta'$ as follows. For each $p \in \dom(P)$, we have $\Theta'(p)$ copy $\Psi(p, \Theta(\Phi(p)))$ while simultaneously watching for a $\bot$ symbol to appear in $\Theta(\Phi(p))$. If $\Theta'$ sees a $\bot$ symbol in $\Theta(\Phi(p))$, then $\Theta'(p)$ begins outputting $\bot$ forever.

    We check that $\Theta'$ satisfies the conditions. Fix a $p \in \dom(P)$. If $\Theta'(p) \in P(p)$, then there is no $\bot$ symbol in $\Theta'(p)$, so $\Theta'(p) = \Psi(p, \Theta(\Phi(p)))$, as desired.
    
    If $\Theta'(p) \not\in P(p)$, then $\Theta(\Phi(p)) \not\in Q(\Phi(p))$, so there is a $j$ such that $\Theta(\Phi(p))(j)\converges = \bot$. The computation $\Theta'(p)$ eventually sees this $\bot$ and outputs $\bot$. Hence, $\Theta'$ identifies its errors with respect to $P$.
\end{proof}

\begin{Definition}
    \label{Def:GuessableWithErrors}
    Let $P:\subset \N^\N \tto \N^\N$ be a problem. Fix $n \in \N$. We say that $P$ is continuously $n$-\emph{guessable with identified errors} if there are partial continuous functionals $\Theta_0, \Theta_1, \dots, \Theta_{n-1}: \subset \N^\N \tto (\N \cup \{\bot\})^\N$ which identify their errors such that for each $p \in \dom(P)$ there is $i < n$ such that $\Theta_i(p) \in P(p)$.
\end{Definition}

The problem $\ACCN$ is computably $2$-guessable with identified errors via the functions 
\[
    \Theta_i(p) = \begin{cases}
        i 0^k \bot^\N &\text{if } \hash^k i \sqsubset p\\
        i 0^\N &\text{otherwise}
    \end{cases}
\]
for $i \in \{0,1\}$. 
Next we show that being $n$-guessable with identified errors is Weihrauch invariant. 
\begin{Theorem}
    \label{Thm:GuessableWithIdentifiedErrorsWeihrauchInvariant}
    Let $P$ and $Q$ be problems such that $P \leq_\W^* Q$. If $Q$ is continuously $n$-guessable with identified errors, then so is $P$.
\end{Theorem}
\begin{proof}
    Let $\Phi$ and $\Psi$ be the continuous functionals witnessing $P \leq_\W^* Q$. Let $\widehat{\Theta}_0, \widehat{\Theta}_1, \dots, \widehat{\Theta}_{n-1}$ witness that $Q$ is continuously $n$-guessable with identified errors. For $i < n$, let $\Theta_i$ be the witness to Lemma~\ref{Lem:PreserveIdentifyingErrors} applied to $\widehat{\Theta}_i$. Then, each $\Theta_i$ identifies its errors for $P$. Furthermore, if $\widehat{\Theta}_i(\Phi(p)) \in Q(\Phi(p))$ then $\Theta_i(p) \in P(p)$. Since at least one of the $\widehat{\Theta}_i$ is correct for each $p$, the same is true for the $\Theta_i$. 
\end{proof}

To show that $\RT^n_{\alpha,\beta} \not\leq_\W^* \ACCN$, it therefore suffices to show that $\RT^n_{\alpha,\beta}$ is not continuously $2$-guessable with identified errors.

\begin{Theorem}
    \label{Thm:RTNotGuessableWithErrors}
    Let $n \in \N$ and $\beta < \alpha \in \omega_1 \cup \{\N\}$. Then, $\RT^n_{\alpha,\beta}$ is not continuously $2$-guessable with identified errors.
\end{Theorem}
\begin{proof}
    We argue by induction. Suppose that the theorem is true for all $\mu < \beta$ and $\nu < \alpha$ with $\mu < \nu$. Suppose some continuous functionals $\Theta_0$ and $\Theta_1$ $2$-guess $\RT^n_{\alpha,\beta}$ with identified errors.  Let $c \equiv 0$ be a constant coloring of $[\N]^n$. Then, $(c,\beta)$ is an instance of $\RT^n_{\alpha,\beta}$. We diagonalize against $\Theta_0$ and $\Theta_1$ by waiting for them to either give up or to each output a color, and therefore commit to an ordinal. Once they do this, we may use the following claim to apply the inductive hypothesis.
    
    \begin{Claim}
        Fix any $(\theta, \beta) \sqsubset (c,\beta)$. Let $m \in \N$ be such that there are $a_1, a_2, \dots a_n < m$ satisfying $(a_1,a_2,\dots, a_n) \in \dom(\theta)$. Then, there is a forward functional $\Phi$ such that the following is true. 
    
        Fix a solution $(H, \vec{\zeta}) \in \RT^n_{\alpha,\beta}(c,\beta)$ such that $H$ contains an element less than $m$. Let $\zeta$ be the first element of $\vec{\zeta}$. Then, there is $(\phi, \vec{\zeta'}) \sqsubset (H, \vec{\zeta})$ such that $\Phi$ is a forward function in a reduction witnessing that
        \[
            \RT^n_{\alpha,\beta}|[\llb (\theta, \beta)\rrb \times \llb (\phi,\vec{\zeta'})\rrb]\geq_\W^*
            \RT^n_{\beta,\zeta}.
        \]  
    \end{Claim}
\begin{proof}[Proof of claim]
    We define $\Phi$ witnessing such reductions. Fix $(H, \vec{\zeta}) \in \RT^n_{\alpha,\beta}(c, \beta)$ and $(d, \vec{\nu}) \in \mathbf{C^n_{\beta}}$. Let $\zeta < \beta$ be the first element of $\vec{\zeta}$. Fix $\phi \sqsubset H$ such that there is $a \in \phi$ satisfying $a < m$. By Remark~\ref{rem:DealingWithNames}, there is $\vec{\zeta'} \sqsubset \vec{\zeta}$ such that $(\phi, \vec{\zeta'}) \sqsubset (H, \vec{\zeta})$.  Recall that, since $\theta \sqsubset c$, we have that $\theta =0$ wherever it is defined. Let $\hat{d} \sqsupset \theta$ be the coloring defined by
    \[
        \hat{d}(a_1, \dots, a_n) = \begin{cases}
            0 &\text{if } (\E i)(a_i < m) \text{ or } \theta(a_1, \dots, a_n) = 0\\
            1 + d(a_1 - m, a_2 - m, \dots, a_n - m) &\text{otherwise}.
        \end{cases}
    \]
    Then, we let $\Phi(d, \vec{\nu}) = (\hat{d}, \beta\vec{\nu})$. More precisely, for a name $p$ of $(d, \vec{\nu})$, we compute a name $\Phi(p)$ for $(\hat{d}, \beta\vec{\nu})$ by first writing out our name for $(\theta, \beta)$ and then continuing to write out $\hat{d}$ with interruptions from $\vec{\nu}$ as we find them. Since $\hat{d}$ gains a new color only when this algorithm sees a new color in $d$, the resulting element of $\extendedBaireSpace$ is a $\mathbf{C^n_\alpha}$-name for $(\hat{d}, \beta\vec{\nu})$.

    Now we build the backward function $\Psi$ witnessing the reduction for specific $(\phi,\zeta)$. Let $(\widehat{A}, \vec{\gamma}) \sqsupset (\phi, \vec{\zeta'})$ be a $\RT^n_{\alpha,\beta}$ solution of $(\hat{d}, \beta\vec{\nu})$. Let $A = \{a: a + m \in \widehat{A}\}$. Let $\vec{\gamma}'$ be $\vec{\gamma}$ without its first element $\zeta$. Then, $A$ is an infinite set and $\vec{\gamma}' \in S(\zeta)$. 

    Also, $0  \in \hat{d}([\widehat{A}]^n)$ because $\phi \sqsubset \widehat{A}$ and there is $i < m$ such that $\phi(i) = 1$. Furthermore, every non-zero color in $\hat{d}([\widehat{A}]^n)$ corresponds with a color in $d([A]^n)$. Thus, $|d([A]^n)| \leq |\hat{d}([\widehat{A}]^n)| - 1$, so $(A,\vec{\gamma}')$ is a $\RT^n_{\beta,\zeta}$-solution to $(d, \vec{\nu})$. We define $\Psi((\hat{d},\beta\vec{\nu}),(\widehat{A},\vec{\gamma})) = (A, \vec{\gamma}')$. This completes the reduction.
\end{proof}
    
    Let $p$ be a $\mathbf{C^n_\alpha}$ name for $(c,\beta)$. There are two cases. 
    
    Case 1: Suppose that either $\Theta_0(p)$ or $\Theta_1(p)$ contains a $\bot$ symbol. 
    Without loss of generality, assume that $\Theta_1(p)$ contains the $\bot$ symbol. Then, $\Theta_0(p)$ is a $\mathbf{R}^c_\beta$-name for some $(H,\vec{\zeta}) \in \RT^n_{\alpha,\beta}((c,\beta))$.
    Let $\zeta$ be the first element of $\vec{\zeta}$. By Remark~\ref{rem:DealingWithNames}, there are $\phi \sqsubset H$ and $\vec{\zeta'} \sqsubset \vec{\zeta}$ such that $(\phi,\vec{\zeta'}) \sqsubset (H,\vec{\zeta})$. 

    Since $\Theta_0$ and $\Theta_1$ are continuous, there is $(\theta, \beta) \sqsubset (c, \beta)$ such that $\Theta_0(\llb (\theta, \beta) \rrb) \subset \llb (\phi, \vec{\zeta'}) \rrb$ and all elements of $\Theta_1(\llb(\theta,\beta)\rrb)$ contain a $\bot$ symbol. Hence, $\Theta_0$ is a continuous realizer of $\RT^n_{\alpha,\beta}|[\llb (\theta, \beta)\rrb \times \llb (\phi, \vec{\zeta'}) \rrb]$.
    By the claim, we have $\RT^n_{\alpha,\beta}|[\llb (\theta, \beta)\rrb \times \llb (\phi, \vec{\zeta'}) \rrb] \geq_\W^* \RT^n_{\beta,\zeta}$. However, since $\beta > \zeta$, the problem $\RT^n_{\beta,\zeta}$ is not continuous, a contradiction. 

    Case 2: Suppose that $\Theta_0(p) = (H_0,\vec{\zeta}_0)$ and $\Theta_1(p) = (H_1,\vec{\zeta}_1)$ both output a solution to $p$. 
    Let $\zeta_0$ be the first element of $\vec{\zeta}_0$ and let $\zeta_1$ be the first element of $\vec{\zeta}_1$. 
    By Remark~\ref{rem:DealingWithNames} and continuity of $\Theta_0$ and $\Theta_1$, there are $(\theta,\beta) \sqsubset (c,\beta)$, $(\phi_0, \vec{\zeta}_0') \sqsubset (H_0, \vec{\zeta}_0)$ and $(\phi_1, \vec{\zeta}_1') \sqsubset (H_1, \vec{\zeta}_1)$ such that $\phi_0$ and $\phi_1$ are each nonempty; $\Theta_0(\llb (\theta, \beta) \rrb) \subset \llb (\phi_0, \vec{\zeta}_0') \rrb$; and $\Theta_1(\llb (\theta, \beta) \rrb) \subset 
    \llb (\phi_1, \zeta_1) \rrb$. Fix an $m$ so that each of $\phi_0 \subset \N$ and $\phi_1 \subset \N$ contains an element less than $m$.
    Suppose without loss of generality that $\zeta_0 \geq \zeta_1.$
    For brevity, we write $P_i = \RT^n_{\alpha,\beta}|[\llb (\theta, \beta)\rrb \times \llb (\phi_i, \zeta_i) \rrb]$ for $i \in \{0,1\}$. 
    By the claim, there is a forward function $\Phi$ yielding the reductions $P_0 \geq_\W^* \RT^n_{\beta,\zeta_0}$ and $P_1 \geq_\W^* \RT^n_{\beta,\zeta_1}$. 
    Furthermore, the identity function $\id$ is a forward function witnessing $\RT^n_{\beta,\zeta_1} \geq_\W^* \RT^n_{\beta,\zeta_0}$, so the forward functions witnessing these reductions commute at the level of names. 
    
    We use these reductions to define functionals $\Theta_0'$ and $\Theta_1'$ witnessing that $\RT^n_{\beta,\zeta_0}$ is continuously $2$-guessable with identified errors. 
    Let $\Psi$ be the backward function witnessing $\RT^n_{\beta,\zeta_1} \geq_\W^* \RT^n_{\beta,\zeta_0}$, let $\Psi_0$ be the backward function witnessing $P_0 \geq_\W^* \RT^n_{\beta,\zeta_0}$ and let $\Psi_1$ be the backward function witnessing that $P_1 \geq_\W^* \RT^n_{\beta,\zeta_1}$. Define $G(p, q) = \Psi(p, \Psi_1(\Phi(p),q))$ to be the standard witness that Weihrauch reducibility is transitive; that is, $G$ is a backward function witnessing that $P_1 \geq_\W^* \RT^n_{\beta,\zeta_0}$ with forward function $\Phi \circ \id = \Phi$. 

    Define $\Theta_0'$ as the witness to Lemma~\ref{Lem:PreserveIdentifyingErrors} applied to $\Theta_0$ and the reduction $P_0 \geq_\W^* \RT^n_{\beta,\zeta_0}$ with forward function $\Phi$ and backward function $\Psi_0$. Define $\Theta_1'$ as the witness to Lemma~\ref{Lem:PreserveIdentifyingErrors} applied to $\Theta_1$ and the reduction $P_1 \geq_\W^* \RT^n_{\beta,\zeta_0}$ with forward function $\Phi$ and backward function $G$. Then, $\Theta_0'$ and $\Theta_1'$ are partial continuous functionals which identify their errors with respect to $\RT^n_{\beta,\zeta_0}$. Furthermore, $\Theta_0'(p)$ is correct if $\Theta_0(\Phi(p))$ is correct and $\Theta_1'(p)$ is correct if $\Theta_1(\Phi(p))$ is correct. Since at least one of $\Theta_0(\Phi(p))$ and $\Theta_1(\Phi(p))$ is correct, at least one of $\Theta_0'(p)$ and $\Theta_1'(p)$ is correct. Hence, $\RT^n_{\beta,\zeta_0}$ is continuously $2$-guessable with identified errors, contradicting the inductive hypothesis.
\end{proof}

\subsection{Conclusion of Section~\ref{Sec:AchromaticRamseyTheorems}}
We have shown that both $\RT^1_{\alpha,\beta}$ and $\RT^n_{\alpha,\omega \cdot \beta}$ have rank~$\beta$ discontinuities for $n \geq 2$ and no discontinuities of lower rank. Furthermore, we have shown that, although $\ACCN^\beta$ also has a rank-$\beta$ discontinuity, neither of these thin-set principles is reducible to $\ACCN^\beta$. We have also shown that if $\beta \leq \mu < \nu \leq \alpha$, then
    \[
        \RT^n_{\alpha,\beta} \geq_\sW^* \RT^n_{\nu,\mu}.
    \]
The situation for $\RT^n_{k,1}$ with $k \in \N$ is well understood (see \cite{brattkaUNIFORMCOMPUTATIONALCONTENT2017,doraisUniformRelationshipsCombinatorial2015, pateyThesis,hirschfeldtJockuschNotionsComputabilitytheoreticReduction2016}). Some of the techniques in the literature may generalize to the ordinal valued case, but this is beyond the scope of this paper. We leave the remaining comparisons between the problems $\RT^n_{\alpha,\beta}$ to future work.

%% file: Conclusion.tex
We conclude by responding directly to a few related notions in the literature. 
\subsection{Diagonalization Opportunities and Modified Discontinuity Games}
\label{Sec:DiagonalizationOpps}
For a set $X \subset \B$, let $\closure(X)$ denote the closure of $X$ in the usual topology on Baire space. 

We note a curious connection between having diagonalization opportunities and the discontinuity games defined in this paper. Hirschfeldt and Jockusch \cite{hirschfeldtJockuschNotionsComputabilitytheoreticReduction2016} define having diagonalization opportunities for computable functionals. They are motivated by an argument similar to the one at the beginning of Section~\ref{Sec:AchromaticRamseyTheorems} showing that $\TS^n_\N$ is discontinuous.

When this is extended to continuous functionals, it turns out that it characterizes the Player~2 wins in a slightly modified version of the discontinuity game. Suppose we allow Player~2 to play \emph{any} point $q$ at Step~\ref{GameStep:Player2Playsq} and also allow the game to last infinitely many rounds, with Player~2 winning if the play is infinite and Player~1's winning condition remaining the same. Since playing any $q \not \in \closure(P(p))$ is a losing move for Player~2, this is equivalent to the infinite-run discontinuity game for $P'$, where $P'(p) = \closure(P(p))$. 
\begin{Definition}[\cite{hirschfeldtJockuschNotionsComputabilitytheoreticReduction2016}]
    Let $P$ be a problem. We say that a continuous functional $\Upsilon: \subset \B \to \B$ has a \emph{diagonalization opportunity} for $P$ if there exists a $p \in \dom(P)$, $\sigma \sqsubset p$ and $\rho \sqsubset \Upsilon(p)$ such that $\Upsilon(\llb \sigma \rrb) \subset \llb \rho \rrb$ and $\llb \rho\rrb \cap P(p) = \emptyset$. We say that $P$ has \emph{diagonalization opportunities for continuous functionals} if every continuous functional which is total on $\dom(P)$ has a diagonalization opportunity for $P$. \footnote{$\TS^n_k$ does not have diagonalization opportunities in the way we have defined it here. Hirschfeldt and Jockusch address this by further stipulating that the functionals in question always yield an element of $2^{\N}$ coding an infinite subset of $\N$.}
\end{Definition}

We note that this is not a Weihrauch-invariant notion. Indeed, let $\ACCN'$ be the problem defined by
\[
    \ACCN'(\hash^\N) = \{\hash^n i q: n \in \N, i \in \N, q \in \extendedBaireSpace\},
\]
and, for $j \in \N$ and $r \in \extendedBaireSpace$,
\[
    \ACCN'(\hash^m j r) = \{\hash^n i q: n \in \N, j \neq i \in \N, q \in \extendedBaireSpace\},
\]
allowing a solver to output $\hash$'s before committing to an element of $\N$. 

Then, $\ACCN'$ is strongly Weihrauch equivalent to $\ACCN$, but $\ACCN'$ does not have diagonalization opportunities for continuous functionals (witnessed by the constant functional $\Upsilon \equiv \hash^\N$), while $\ACCN$ does. Nevertheless, the following theorem holds. 

\begin{Theorem}
    \label{Thm:DiagonalizationOpportunitiesGame}
    Let $P: \subset \B \tto \B$ be a problem and let $P'$ be the problem defined by $P'(p) = \closure(P(p))$. Then, Player~2 has a winning strategy in the infinite-run discontinuity game for $P'$ if and only if $P$ does not have diagonalization opportunities for continuous functionals.
\end{Theorem}
\begin{proof}
    ($\impliedby$): Suppose that $P$ does not have diagonalization opportunities for continuous functionals. Let $\Upsilon$ be a continuous functional which is total on $\dom(P)$ and does not have a diagonalization opportunity for $P$. Define Player~2 strategy $\Delta_\Upsilon$ by having $\Delta_\Upsilon$ always play $\Upsilon(p)$ in response to Player~1 move $p \in \dom(P)$, and in response to $\rho \sqsubset \Upsilon(p)$, $\Delta_\Upsilon$ plays the shortest $\sigma \sqsubset p$ such that $\Upsilon(\llb \sigma \rrb) \subset \llb \rho \rrb$, depicted below. 
    \begin{equation*}
			\begin{array}{c|c}
				\text{Player~1} (\Gamma)& \text{Player~2} (\Delta_\Upsilon) \\
				p & \Upsilon(p)\\
				\rho & \sigma \text{ such that } \Upsilon(\llb \sigma \rrb) \subset \llb \rho \rrb \\
		\end{array}
	\end{equation*}

    Since $\Upsilon$ does not have diagonalization opportunities for $P$, the set $\llb \rho \rrb \cap P(p)$ is non-empty for each $\rho \sqsubset \Upsilon(p)$, so $\Upsilon(p) \in \closure{(P(p))}$ and hence $\Delta_\Upsilon$ always plays valid Player~2 moves in the discontinuity game for $P'$. Thus, $\Delta_\Upsilon$ cannot lose, so is a winning strategy for Player~2 in the infinite-run discontinuity game for $P'$.

    ($\implies$): Let $\Delta$ be a winning strategy for Player~2 in the infinite-run discontinuity game for $P'$. We define a continuous functional $\Upsilon$ which is total on $\dom(P)$ and does not have a diagonalization opportunity for $P$. The idea is to use $\Delta$ to define $\Upsilon$ one entry at a time on a collection of strings whose cones cover $\dom(P)$. We do this by running branching runs of the discontinuity game for $P'$ simultaneously, with each run defining $\Upsilon$ on a different string and with each having Player~2 play according to $\Delta$. We index the runs of the game by points in a set $X \subset \N^\N$ with finite partial runs of the game indexed by elements of $\N^{<\N}$. For each $\xi \in \{\xi: \xi \sqsubset x \text{ for some } x \in X\}$, we will define $p_\xi$, $q_\xi$, $\rho_\xi$, and $\sigma_\xi$ to be moves in the game as in the table below. 
    \begin{equation*}
            \begin{array}{c|c}
                \text{Player~1} & \text{Player~2} (\Delta) \\
                p_\xi & q_\xi\\
                \rho_\xi & \sigma_\xi\\
        \end{array}
    \end{equation*}
    Then, we set
    \[
        \Upsilon(\llb \sigma_{x|_i} \rrb) \subset \llb \rho_{x|_i} \rrb.
    \]
    
    Let $\sigma_{\lambda} = \rho_{\lambda} = \lambda$. If $\dom(P) = \emptyset$, then the nowhere defined function witnesses that $P$ does not have diagonalization opportunities, so we may assume that $\dom(P)$ is nonempty.  Let $p \in \dom(P)$ be arbitrary. We set $p_\lambda := p$. For $\xi \in \N^{<\N}$ suppose that $p_{\xi}$, $q_{\xi}$, $\sigma_{\xi}$ and $\rho_{\xi}$ are defined. For $r \in \dom(P)$, let 
    \[
        s_r = \Delta(p_{\lambda}, \rho_{\lambda}, \dots, p_{\xi}, \rho_{\xi}, r) \in P'(r) 
    \] 
    be $\Delta$'s response to Player~1 extending the game with index $\xi$ by playing $p = r$.
    We apply Lemma~\ref{Lemma:CoverByNodes} to $\llb \sigma_{\xi} \rrb$ with the relation $R$ defined by 
    \[
        R(\sigma) \iff (\E r \sqsupset \sigma)(\Delta(p_{\lambda}, \rho_{\lambda}, \dots, p_{\xi}, \rho_{\xi}, r, s_r|_{|\xi| + 1}) = \sigma).
    \] 
    Our construction ensures that the sequence of moves indexed by initial segments of $\xi$ is a valid partial run of the discontinuity game for $P'$ in which Player~2 plays according to $\Delta$,  so each $p \sqsupset \sigma_\xi$ has a $\sigma \sqsubset p$ satisfying $R(\sigma)$. Hence, we may apply Lemma~\ref{Lemma:CoverByNodes} to obtain a (possibly finite) sequence $\sigma_{\xi\cat0}, \sigma_{\xi \cat1} \dots$ so that $R(\sigma_{\xi \cat i})$ holds for each $i$ and so that the $\llb \sigma_{\xi \cat i}\rrb$ form a disjoint cover of $\llb \sigma_\xi \rrb \cap \dom(P)$. For each $i$, let $p_{\xi \cat i}$ be the $r \sqsupset \sigma_{\xi \cat i}$ promised by $R(\sigma_{\xi \cat i})$, let $q_{\xi\cat i} = s_{p_{\xi \cat i}}$ and let $\rho_{\xi \cat i} = q_{\xi \cat i}|_{|\xi| + 1}$, matching the round of the game described in our definition of $R$. 

    Since the $\llb \sigma_{\xi \cat i}\rrb$ are disjoint covers of $\dom(P)$ for each fixed $\xi$, our definition of $\Upsilon$ is well defined as a continuous partial function on Baire space and total on $\dom(P)$. 

    Finally, we show that $\Upsilon$ does not have a diagonalization opportunity. Fix  $p \in \dom(P)$, $\sigma \sqsubset p$ and $\rho \sqsubset \Upsilon(p)$ such that $\Upsilon(\llb \sigma \rrb) \subset \llb \rho \rrb$. Then, there is some $\xi \in \N^{<\N}$ such that $\rho = \rho_\xi$ and $\sigma \sqsupset \sigma_\xi$. Since $\Delta$ is a winning strategy for Player~2 in the infinite-run discontinuity game for $P'$ and $\sigma_\xi = \Delta(p_{\lambda}, \rho_{\lambda}, \dots, p_{\xi}, \rho_{\xi})$, we have that the partial run with index $\xi$ in which Player~2 follows $\Delta$ begins the $P'|[\llb \sigma_\xi \rrb \times \llb \rho_\xi \rrb]$-game which $\Delta$ continues on to win. Furthermore, $\sigma_\xi \sqsubset p$, so $p$ is a valid Player~1 move. Since $\Delta$ is a winning strategy, Player~2 does not lose when Player~1 plays $p$, so there is $q \in \closure(P(p)) \cap \llb \rho_\xi \rrb$. By properties of closures, there is $q' \in P(p) \cap \llb \rho_\xi \rrb$. Hence, $\Upsilon$ does not have a diagonalization opportunity for $P$.
\end{proof}

Although having diagonalization opportunities is not Weihrauch invariant, Theorem~\ref{Thm:DiagonalizationOpportunitiesGame} and the proof that $\TS^n_\N$ is discontinuous at the beginning of Section~\ref{Sec:AchromaticRamseyTheorems} suggest that a more systematic account of this notion (or a similar one) may be fruitful.

Loosely motivated by the exponent lifting technique of Hirschfeldt and Jockusch \cite{hirschfeldtJockuschNotionsComputabilitytheoreticReduction2016}, we discover that Player~1 having a strategy in the rank-$\alpha$ discontinuity game which defeats continuous functionals of fixed complexity is preserved by computable Weihrauch reductions.  

\begin{Theorem}
    \label{Thm:Computable-WeihruachInvariant}
    Fix a set $Z \subset \N$ and a countable ordinal $\alpha < \omega_1$. Let $P: \subset \B \tto \B$ and $Q: \subset \B \tto \B$ be problems such that $P \leq_\W Q$. 
    Suppose that there is a Player~1 strategy $\Gamma$ for the $\alpha$-$P$ game which defeats the Player~2 strategy $\Delta_\Upsilon$ for all $Z$-computable $\Upsilon$, where $\Delta_\Upsilon$ is defined as in the proof of Theorem~\ref{Thm:DiagonalizationOpportunitiesGame}. 
    Then, there is a Player~1 strategy $\widehat{\Gamma}$ for the $\alpha$-$Q$ game which defeats the strategy $\Delta_{\widehat{\Upsilon}}$ for all $Z$-computable functionals $\widehat{\Upsilon}$.
\end{Theorem}
\begin{proof}
    Suppose that $P \leq_\W Q$ via $\Phi$ and $\Psi$. Given a $Z$-computable functional $\widehat{\Upsilon}$, define $\Upsilon(p) = \Psi(p,\widehat{\Upsilon}(\Phi(p))).$ Then, the $\widehat{\Gamma}$ defined in the proof of Theorem~\ref{Thm:DiscontinuityGameInvariant} plays against $\Delta_{\widehat{\Upsilon}}$ when $\Gamma$ plays against $\Delta_{\Upsilon}$. Since $\Gamma$ wins its run of the game, so does $\widehat{\Gamma}$.
\end{proof}

Our proof of Theorem~\ref{Thm:Computable-WeihruachInvariant} does not extend to continuous Weihrauch reducibility because we lose the $Z$-computability of $\Upsilon$. 

\subsection{Ranks and Represented Spaces}

Consider the following notion of Cantor-Bendixson rank of a set $A \subset \B$ relative to a represented space $\mb{X} = (X, \delta_X)$.

\begin{Definition}
    \label{Def:RankOfSpace}
    Let $\mathbf{X} = (X, \deltarep{X})$ be a represented space and let $A \subseteq \dom(\deltarep{X})$.
    \begin{itemize}
        \item  We say that an open set $U$ is an \emph{$\mathbf{X}$-$\leq\!\!0$-isolating neighborhood for a point $p$ in $A$} if $p \in U \cap A$ and for all $q \in U \cap A$, we have $\deltarep{X}(p) = \deltarep{X}(q)$.
        \item For $\alpha > 0$, we say that an open set $U$ is an \emph{$\mathbf{X}$-$\leq\!\!\alpha$-isolating neighborhood for a point $p$ in $A$} if $p \in U \cap A$ and for all points $q \in U \cap A$ with $\deltarep{X}(p) \neq \deltarep{X}(q)$ there is $\beta < \alpha$ and open set $V$ such that $V$ is an $\mathbf{X}$-$\leq\!\!\beta$-isolating neighborhood of $q$. 
        \item If $p \in A$ has an $\mathbf{X}$-$\leq\!\!\alpha$-isolating neighborhood then we write $\rank^\mathbf{X}_A(p) \leq \alpha$.
        \item By a temporary abuse of notation, we define $\rank^\mathbf{X}_A(p) = \alpha$ if $\alpha$ is the least ordinal such that $\rank^\mathbf{X}_A(p) \leq \alpha$. If no such $\alpha$ exists then we write $\rank^\mathbf{X}_A(p) = \infty$.
        \item We say that $\rankp^\mathbf{X}(A) = \alpha$ for the least ordinal $\alpha$ such that $\rank^\mathbf{X}_A(p) \leq \alpha$ for all $p \in A$. If no such $\alpha$ exists then we write $\rankp^\mathbf{X}(A) = \infty$. 
        \item We write $\rankp(\mathbf{X})$ for $\rankp^\mathbf{X}(\dom(\deltarep{X}))$.
    \end{itemize}
\end{Definition}

In an early draft of this paper, it was claimed that a problem $P: \subset \mb{X} \tto \mb{Y}$ has a discontinuity of rank $\alpha$ if and only if $\realizerver{P}$ is discontinuous on a set $A \subset \B$ with $\rankp^{\mb{X}}(A) \leq \alpha$. The author found a hole in the proof and the following was posed as an open question. 

\begin{Question}
    \label{Question:Rank1Counterexample}
    Does there exist a represented space $\mb{X}$ and problem $P: \subset \mb{X} \tto \B$ such that 
    \begin{enumerate}
        \item $P$ is discontinuous,
        \item $\rankp^\mb{X}(\dom(\realizerver{P})) = 1$,
        \item and Player~2 has a winning strategy in the rank-$1$ discontinuity game for $\realizerver{P}$?
    \end{enumerate}
\end{Question}

Jun Le Goh  has since found a counter-example with $\rankp(\dom(\realizerver{P})) = 2$ (personal communication). 

\subsection{Other Rankings of Discontinuity}
We compare our results to those of some other notions of ranking discontinuity introduced by Peter Hertling. 
\subsubsection{Levels of Discontinuity and Mind Changes}

In \cite{hertlingTopologicalComplexityContinuous1996}, Hertling defines the levels of discontinuity of a single valued function by iteratively removing points outside of the closure of the set of discontinuities. The level of discontinuity of a function is the least ordinal $\alpha$ such that the $\alpha$-th iteration of this process yields the empty set, or $\infty$ if this never happens. In contrast, the ranks of Definition~\ref{Def:RankOfDiscontinuity} are local. Furthermore, they collapse to being either $0$, $1$, or $2$ for single valued functions.

De Brecht shows that Hertling's levels characterize when a single-valued function between countable $T_0$ spaces is continuous with $\alpha$ many mind changes \cite{debrechtLevelsDiscontinuityLimitcomputability2014}. We obtain a sufficient condition in terms of our ranking for multivalued functions: if $P$ has domain of pointwise Cantor-Bendixson rank $\alpha$, then $P$ is continuously solvable with $\alpha$ many mind changes, which we shall prove below. This echoes the result of Luo and Schulte \cite{luoMindChangeEfficient2006} that if the pointwise Cantor-Bendixson rank of a collection $\mathcal{L}$ of languages is $\leq \alpha$ then $\mathcal{L}$ is learnable with $\alpha$ many mind changes.
Let $w_{\alpha}$ be an enumeration of $\alpha$ as in Section~\ref{Sec:AchromaticRamseyTheorems}. In the following definition, the symbol $\hash$ is used to represent a mind change.
\begin{Definition}
    Let $P: \subset \B \tto \B$ be a problem and let $\alpha$ be a countable ordinal. We define the problem $\MC_\alpha(P): \subset \B \tto \extendedBaireSpace$ by $\dom(\MC_{\alpha}(P)) = \dom(P)$ and 
    \[
        \MC_\alpha(P)(p) = \{\phi_0 \hash s_1 \phi_1 \dots \hash s_k q :\phi_i \in \N^{<\N},  \alpha > w_\alpha(s_1) > \dots > w_{\alpha}(s_k), \text{ and } q \in P(p)\}.
    \] 
\end{Definition}

\begin{Theorem}
    \label{Thm:MCAlphaContinuous}
    Let $P: \subset \B \tto \B$ be a problem and let $\alpha$ be a countable ordinal. If $P(p)$ is nonempty for all $p \in \dom(P)$ and $\rankp(\dom(P)) \leq \alpha$ then $\MC_\alpha(P)$ is continuous. 
\end{Theorem}
\begin{proof}
    We argue by induction. Suppose that the theorem is true for all $\beta < \alpha$. If $\rankp(\dom(P)) < \alpha$ then we may apply the induction hypothesis along with the fact that being continuous in $\beta$ mind changes implies being continuous in $\alpha$ mind changes for all $\alpha > \beta$. 

    Hence, we may assume that $\rankp(\dom(P)) = \alpha$. We give a continuous algorithm for computing $\MC_{\alpha}(P)$. Since $\rankp(\dom(P)) = \alpha$ we may apply Lemma~\ref{Lem:CoverByIsolating} to obtain $\sigma_1, \sigma_2, \dots$ and $a_1 \in \llb \sigma_1 \rrb$, $a_2 \in \llb \sigma_2 \rrb$, $\dots$ such that for each $i \in \N$, $\sigma_i$ is a $\rank_{\dom(P)}(a_i)$-isolating neighborhood of $a_i$ in $\dom(P)$ and such that $\dom(P) = \bigsqcup_{i \in \N} \dom(P) \cap \llb \sigma_i \rrb$.
    
    For each $i \in \N$, fix a $b_i \in P(a_i)$. Given $p \in \dom(P)$, we describe how to continuously produce an element of $\MC_\alpha(P)(p)$. First, read $p$ until discovering that $\sigma_i \sqsubset p$ for some $i \in \N$. Then, start outputting $b_i$. If it turns out that $p = a_i$, then we have succeeded. Otherwise, we find some $\tau \sqsubset p$ such that $\tau \not\sqsubset a_i$ after outputting some initial segment $\phi \sqsubset b_i$. Since $\llb \sigma \rrb$ is $\leq\alpha$-isolating of $a_i$, we have that each point of $\dom(P) \cap \llb \tau \rrb$ has rank strictly less than $\alpha$. Hence, $\rank_{\dom(P)}(p) = \beta$ for some $\beta < \alpha$. Find a $\sigma' \sqsubset p$ such that $\llb \sigma'\rrb$ is a $\beta$-isolating neighborhood for $p$. By the inductive hypothesis, $\MC_\beta(P|_{\llb \sigma' \rrb})$ is continuous, say via some function $\Phi$. We then output $\phi \hash s \Phi(p)$, where $s$ is such that $w_\alpha(s) = \beta$. Since the last guess of $\Phi$ is correct, we have succeeded.
\end{proof}

Note that Theorem~\ref{Thm:MCAlphaContinuous} implies that $\MC_{\alpha}(P)$ never has a discontinuity of rank $\beta$ such that $0 < \beta \leq \alpha$. This makes sense because Player~2 can always play a solution to Player~1's instance while committing to only a single additional mind change; an alternative proof to Theorem~\ref{Thm:MCAlphaContinuous} uses the idea that $\MC_{\alpha}$ has access to $\alpha + 1$ many attempts, while Player~1 can only force Player~2 to change their mind about $\alpha$ many of these attempts in the rank-$\alpha$ discontinuity game.

\subsubsection{Ranks of Forest Classes Representing a Weihrauch Degree}
Hertling \cite{hertlingForestsDescribingWadge2020} defines a notion of reducibility between labelled forests which corresponds to Weihrauch reducibility. He notes that a natural notion of rank of discontinuity emerges from this framework by taking the rank of the forest class representing a Weihrauch degree. While there are similarities in the construction, this notion of rank of discontinuity is different from the one we have defined here. Indeed, in Section~\ref{sec:SingleValuedFunctions}, we observe that the rank of a point of discontinuity for single valued functions collapses to be either $0$, $1$, or $2$. On the other hand, Hertling exhibits examples of single valued functions whose forest class has rank $\omega$ \cite[cf. Example 8.6]{hertlingForestsDescribingWadge2020}. A rough heuristic for the source of the difference is that the recursion in Definition~\ref{Def:RankOfDiscontinuity} takes place while considering instances and solutions one at a time while the recursive step in Hertling's notion simultaneously takes into account all solutions to all instances in an open subset of the domain. It would be interesting to understand the relationship between these two notions of rank of discontinuity more precisely.